\documentclass[12pt, reqno]{amsart}

\usepackage{mathrsfs}
\usepackage{amsfonts}
\usepackage[centertags]{amsmath}
\usepackage{amssymb,array}
\usepackage{amsthm}
\usepackage{stmaryrd}
\usepackage{graphicx}
\usepackage{caption}
\usepackage{ytableau}
\usepackage{enumerate,enumitem}
\usepackage{MnSymbol}
\SetLabelAlign{center}{\hfill #1\hfill}
\usepackage[textwidth=16cm, hmarginratio=1:1]{geometry}
\usepackage{tikz-cd, tikz}
\usepackage{rotating}
\usepackage[all,cmtip]{xy}
\usepackage[titletoc, title]{appendix}
\usepackage{amscd}
\usepackage[colorlinks]{hyperref}
\usepackage{float}
\usepackage{mathtools}
\usepackage{dsfont} 
\hypersetup{
bookmarksnumbered,
pdfstartview={FitH},
breaklinks=true,
linkcolor=blue,
urlcolor=blue,
citecolor=blue,
bookmarksdepth=2
}
\usepackage{adjustbox}

\theoremstyle{plain}
   
   \newtheorem{theorem}{Theorem}[section]

   \newtheorem{lemma}[theorem]{Lemma}
   
   \newtheorem{conjecture}[theorem]{Conjecture}
   
\theoremstyle{definition}
   \newtheorem{definition}[theorem]{Definition}
   
   \newtheorem{example}[theorem]{Example}

   \newtheorem{remark}[theorem]{Remark}

\numberwithin{equation}{section}

\newcommand{\CC}{{\mathbb {C}}}

\newcommand{\RR}{{\mathbb {R}}}
\newcommand{\ZZ}{{\mathbb {Z}}}

\newcommand{\ch}{{\operatorname{ch}}}

\newcommand{\SSYT}{{\rm SSYT}}

\newcommand{\n}{{\bf n}}

\newcommand{\tchi}{\widetilde{\chi}_q}
\newcommand{\tchiqt}{\widetilde{\chi}_{q,t}}

\DeclareMathOperator*{\diag}{diag}

\DeclareMathOperator{\Gr}{Gr}

\newcommand\scalemath[2]{\scalebox{#1}{\mbox{\ensuremath{\displaystyle #2}}}}

\newlength{\mysizetiny}
\newlength{\mysizesmall}
\newlength{\mysize}
\newlength{\mysizelarge}
\newenvironment{nouppercase}{%
	\renewcommand{\uppercasenonmath}[1]{}}{}

\begin{document} 

\title{Tropical Geometry, Quantum Affine Algebras, and Scattering Amplitudes} 

\author{Nick Early and Jian-Rong Li}
\address{Nick Early, Max Planck Institute for Mathematics in the Sciences, Leipzig, Germany.}
\email{\href{mailto:nick.early@mis.mpg.de}{nick.early@mis.mpg.de}}  

\address{Jian-Rong Li, Faculty of Mathematics, University of Vienna, Oskar-Morgenstern-Platz 1, 1090 Vienna, Austria.} 
\email{lijr07@gmail.com}

\date{}
	\begin{nouppercase}
		\maketitle
	\end{nouppercase}

\begin{abstract} 

The goal of this paper is to make a connection between tropical geometry, representations of quantum affine algebras, and scattering amplitudes in physics. The connection allows us to study important and difficult questions in these areas: 
\begin{enumerate} 
\item We give a systematic construction of prime modules (including prime non-real modules) of quantum affine algebras using tropical geometry. We also introduce new objects which generalize positive tropical Grassmannians.

\item We propose a generalization of Grassmannian string integrals in physics, in which the integrand is a product indexed by prime modules of a quantum affine algebra. We give a general formula of $u$-variables using prime tableaux (corresponding to prime modules of quantum affine algebras of type $A$) and Auslander-Reiten quivers of Grassmannian cluster categories.  

\item We study limit $g$-vectors of cluster algebras. This is another way to obtain prime non-real modules of quantum affine algebras systematically. Using limit $g$-vectors, we construct new examples of non-real modules of quantum affine algebras. 
\end{enumerate}

\end{abstract}

\setcounter{tocdepth}{1}
\tableofcontents

\section{Introduction}

Quantum groups were introduced independently by Drinfeld \cite{Dri85} and Jimbo \cite{Jim85} around 1985. A quantum affine algebra is a Hopf algebra that is a $q$-deformation of the universal enveloping algebra of an affine Lie algebra \cite{CP94}. Quantum affine algebras have many applications to physics, for example, to the theory of solvable lattice models in quantum statistical mechanics \cite{Ba82, IFR92}, integrable systems \cite{DFJMN93}.  Quantum affine algebras also have many connections to different areas of mathematics, for example, cluster algebras \cite{HL10}, KLR algebras \cite{KKKO18}, geometric representation theory \cite{Nak04}, representations of affine Hecke algebras and $p$-adic groups \cite{CP97, Gu22, LM}. 

Let $\mathfrak{g}$ be a simple Lie algebra over $\mathbb{C}$ and $U_q(\widehat{\mathfrak{g}})$ the corresponding quantum affine algebra \cite{CP94}. Chari and Pressley have classified simple finite dimensional $U_q(\widehat{\mathfrak{g}})$-modules. They proved that every simple finite dimensional $U_q(\widehat{\mathfrak{g}})$-module corresponds to an $I$-tuple of Drinfield polynomials, where $I$ is the set of vertices of the Dynkin diagram of $\mathfrak{g}$, and equivalently, corresponds to a dominant monomial $M$ in certain formal variables $Y_{i,a}$, $i \in I$, $a \in \CC^*$. The simple $U_q(\widehat{\mathfrak{g}})$-module corresponding to $M$ is denoted by $L(M)$. 

A simple $U_q(\widehat{\mathfrak{g}})$-module $L(M)$ is called prime if it is not isomorphic to $L(M') \otimes L(M'')$ for any non-trivial modules $L(M')$, $L(M'')$, see \cite{CP97}. When $\mathfrak{g}=\mathfrak{sl}_2$, all prime modules of $U_q(\widehat{\mathfrak{sl}_2})$ are Kirillov-Reshetikhin modules \cite{CP97}. Kirillov-Reshetikhin modules are simple $U_q(\widehat{\mathfrak{g}})$-modules which correspond to dominant monomials of the form $Y_{i,s}Y_{i,s+2d_i} \cdots Y_{i,s+2rd_i}$, where $i \in I$, and $d_i$'s are diagonal entries of a diagonal matrix $D$ such that $DC$ is symmetric, $C$ is the Cartan matrix of $\mathfrak{g}$ (we choose $D$ such that $d_i$'s are as small as possible). In general, to classify prime modules of $U_q(\widehat{\mathfrak{g}})$ is an important and difficult problem in representation theory, see for example, \cite{BCM18, CMY13, CP97, HL10,  MouS22}.  

Hernandez and Leclerc \cite{HL10} made a breakthrough to the problem of constructing prime modules of $U_q(\widehat{\mathfrak{g}})$ using the theory of cluster algebras \cite{FZ02}. For every simple Lie algebra $\mathfrak{g}$ over $\CC$, they constructed a cluster algebra with initial cluster variables given by certain Kirillov-Reshetikhin modules. Using cluster algebras, prime modules can be generated using a procedure called mutation. The prime modules generated in this way are cluster variables. They conjectured that all cluster variables (resp. cluster monomials) are real prime modules (resp. real modules), and all real prime modules (resp. real modules) are cluster variables (resp. cluster monomials). Here a simple $U_q(\widehat{\mathfrak{g}})$-module $L(M)$ is called real if $L(M) \otimes L(M)$ is still simple \cite{Lec}. One direction of their conjecture ``all cluster variables (resp. cluster monomials) are real prime modules (resp. real modules)'' is proved by Qin in \cite{Qin17} and by Kang, Kashiwara, Kim, Oh, and Park in \cite{KKKO18, KKOP20, KKOP21, KKOP22}. The other direction ``all real prime modules (resp. real modules) are cluster variables (resp. cluster monomials)'' of the conjectural is widely open \cite{HL21}. It is shown in \cite{DLL19} that all prime snake modules of types $A, B$ are cluster variables and in \cite{DR20} that all snake modules of types $A, B$ are cluster monomials. 

Recently, Lapid and Minguez \cite{LM} have classified all real simple $U_q(\widehat{\mathfrak{sl}_k})$-modules (in the language of representations of $p$-adic groups) satisfying a certain condition called regular. This classification is surprisingly related to the classification of rationally smooth Schubert varieties in type $A$ flag varieties. They also gave more conjectures and results in a more recent work \cite{LM20}. 

By the results in \cite{KKKO18, KKOP20, KKOP21, KKOP22, Qin17}, using the procedure of mutations, one can generate a large family of prime modules (these prime modules are cluster variables). On the other hand, there are many prime modules which are not real (and thus not cluster variables). These non-real modules are also important in applications. For example, it is shown in \cite{ALS21, DFGK, HP20} that non-real prime tableaux (corresponding to non-real prime modules by \cite{CDFL}) determine the so-called square roots and are used to construct algebraic letters in the computations of Feynman integrals in the study of scattering amplitudes in $\mathcal{N}=4$ super Yang-Mills theory in physics.   

The goal of this paper is to make a connection between tropical geometry, representations of quantum affine algebras, and scattering amplitudes in physics. The connection allows us to study important and difficult questions in these areas: use tropical geometry to construct prime modules (conjecturally we can obtain all prime modules including prime non-real modules) of quantum affine algebras, to propose a generalization of Grassmannian string integrals in physics.

\subsection{} \label{subsec:tropical geometry and grassmannian cluster algebra}

First consider the case where $\mathfrak{g}$ is of type $A$, i.e. $\mathfrak{g} = \mathfrak{sl}_k$ for some positive integer $k$. Simple modules of $U_q(\widehat{\mathfrak{sl}_k})$ correspond to dual canonical basis elements of a quotient $\CC[\Gr(k,n,\sim)]$ of the Grassmannian cluster algebra $\CC[\Gr(k,n)]$ \cite{HL10, CDFL}. We define Newton polytopes ${\bf N}_{k,n}^{(d)}$ by using the formula for dual canonical basis elements of $\CC[\Gr(k,n)]$, \cite{CDFL} as follows. Denote by $\mathcal{T}^{(0)}_{k,n}$ the set of all one-column tableaux which are cyclic shifts of the one-column tableau with entries $\{1,2,\ldots, k-2,k\}$. For $d \ge 0$, we define (see Definition \ref{def:newton polytopes for grassmannian cluster algebras version 1})
\begin{align*}
{\bf N}^{(d)}_{k,n} = \text{Newt}\left(\prod_{T \in \mathcal{T}^{(d)}_{k,n}} \text{ch}_T(x_{i,j})\right),
\end{align*}
where $\mathcal{T}^{(d)}_{k,n}$ ($d \ge 1$) is the set of tableaux which correspond to facets of ${\bf N}^{(d-1)}_{k,n}$, $\ch_T(x_{ij})$ is the evaluation of $\ch_T = \ch(T)$ on the web matrix \cite{SW05} (see Section \ref{subsec:tropical grassmannians}) and $\ch(T)$ is given in Theorem 5.8 and Definition 5.9 in \cite{CDFL}, see Section \ref{subsec:Grassmannian cluster algebras and semistandard Young tableaux}. 

The facets of the Newton polytope ${\bf N}_{k,n}^{(0)}$ have been classified in recent work of the first author \cite{E2021}; there are exactly of them $\binom{n}{k}-n$, one for each of $\binom{n}{k}-n$ generalized positive roots. 

We give a procedure to construct the highest $l$-weights (equivalently, the tableaux corresponding to the highest $l$-weight monomials, see \cite{CDFL}) of simple $U_q(\widehat{\mathfrak{sl}}_k)$-modules from facets of ${\bf N}_{k,n}^{(d)}$ explicitly, see Section \ref{sec:rays to prime modules}. We conjecture that for every $k \le n$ and $d \ge 0$, every facet of ${\bf N}_{k,n}^{(d)}$ gives a prime module of $U_q(\widehat{\mathfrak{sl}_k})$ and every prime $U_q(\widehat{\mathfrak{sl}_k})$-module corresponds to a facet of the Newton polytope ${\bf N}^{(d)}_{k,n}$ for some $d$, see Conjecture \ref{conj:prime modules are rays}). The procedure gives a systematical way to construct prime modules of $U_q(\widehat{\mathfrak{sl}_k})$. 

We prove Conjecture \ref{conj:prime modules are rays} in the case of ${\bf N}_{3,9}^{(1)}$ (resp. ${\bf N}_{4,8}^{(1)}$) in Section \ref{sec:Coarsest subdivisions and prime modules}: all the facets of ${\bf N}_{3,9}^{(1)}$ (resp. ${\bf N}_{4,8}^{(1)}$) correspond to prime modules of $U_q(\widehat{\mathfrak{sl}_3})$ (resp. $U_q(\widehat{\mathfrak{sl}_4})$). 

The study of Newton polytopes of Laurent expansions of cluster variables was initiated by Sherman
and Zelevinsky in their study of rank $2$ cluster algebras \cite{SZ04}. It was further developed in \cite{BS18, Fei22, Kal14, LLR20, LP22, MS22}. Our definition of Newton polytopes in this paper involve not only cluster variables but also other prime elements in the dual canonical basis of cluster algebras. 

We also introduce another version (non-recursive) of Newton polytopes ${\bf N'}_{k,n}^{(d)}$ for Grassmannian cluster algebras, see Definition \ref{def:another version of Newton polytopes for Grkn}.

The normal fans $\mathcal{N}({\bf N}_{k,n}^{(d)})$ are generalizations of tropical Grassmannians ${\rm Trop}^+ G(k,n)$ (see \cite{SW05}), see Section \ref{subsec:tropical fans for Grassmannian cluster algebras}.  It would be an interesting and important problem to find a combinatorial model which describes the facets of $\mathcal{N}({\bf N}_{k,n}^{(d)})$, and to relate them to scattering amplitudes.

\subsection{}
We generalize the construction in Section \ref{subsec:tropical geometry and grassmannian cluster algebra} to general quantum affine algebras $U_q(\widehat{\mathfrak{g}})$. For any simple Lie algebra $\mathfrak{g}$ over $\CC$ and $\ell \ge 1$, we define a sequence of Newton polytopes ${\bf N}^{(d)}_{\mathfrak{g}, \ell}$ ($d \in \ZZ_{\ge 0}$) recursively. Let $\mathcal{M}$ be the set of all equivalence classes of Kirillov-Reshetikhin modules of $U_q(\widehat{\mathfrak{g}})$ in Hernandez and Leclerc's category $\mathcal{C}_{\ell}$ \cite{HL10, HL16}. Denote $\mathcal{M}^{(0)}=\mathcal{M}$ and $\mathcal{M}^{(d+1)}$ ($d \ge 0$) the collection of equivalence classes of $U_q(\widehat{\mathfrak{g}})$-modules which correspond to facets of
\begin{align*}
{\bf N}^{(d)}_{\mathfrak{g}, \ell} := \text{Newt}\left(\prod_{[L(M)] \in \mathcal{M}^{(d)}} \tchi(L(M))/M \right),
\end{align*} 
where $\tchi(L(M))$ is the truncated $q$-characters \cite{FR98} of the $U_q(\widehat{\mathfrak{g}})$-module $L(M)$, see Section \ref{sec:Newton polytopes for quantum affine algebras}. The definition of ${\bf N}^{(d)}_{\mathfrak{g}, \ell}$ is recursive and at each step, we give an explicit construction of simple $U_q(\widehat{\mathfrak{g}})$-modules from facets of ${\bf N}^{(d)}_{\mathfrak{g}, \ell}$.

We conjecture that (1) for any $d \ge 0$, every facet of ${\bf N}^{(d)}_{\mathfrak{g}, \ell}$ corresponds to a prime $U_q(\widehat{\mathfrak{g}})$-module and (2) every prime $U_q(\widehat{\mathfrak{g}})$-module corresponds to a facet of the Newton polytope ${\bf N}^{(d)}_{\mathfrak{g}, \ell}$ for some $d$, see Conjecture \ref{conj:newton polytope and prime modules conjecture quantum affine algebra case}. We also introduce another version (non-recursive) of Newton polytopes ${\bf N'}_{\mathfrak{g}, \ell}^{(d)}$ for representations of quantum affine algebras, see Definition \ref{def:another version of Newton polytopes for quantum affine algebras}. 

\subsection{}
In 1969, Z. Koba and K. Nielsen \cite{KN69} introduced an integral representation for the Veneziano-type $n$-point function, parametrized so that a quantum field theoretic phenomenon known as crossing-symmetry, which asserts that particles are indistinguishable from anti-particles traveling back in time, is manifest.  The integrand is expressed as a product of certain cross-ratios, with exponents the kinematic Mandelstam parameters; the cross-ratios are solutions to a particularly combinatorially nice set of binomial algebraic equations, of the form 
$$u_{j_1,j_3} + \prod_{\{j_2,j_4\}} u_{j_2,j_4} = 1,$$
where the product is over all pairs $\{j_2,j_4\}$ such that $j_1<j_2<j_3<j_4$ up to cyclic rotation. These equations, later rediscovered by Brown \cite{Brown} in the context of multiple zeta values and moduli spaces, characterize a certain partial compactification of the moduli space $\mathfrak{M}_{0,n}$ of $n$ distinct points on the Riemann sphere, which is closely related to the tropical Grassmannian ${\rm Trop}G(2,n)$, and in our context a certain subset of it, the \textit{positive} tropical Grassmannian ${\rm Trop}^+G(2,n)$.  For more recent work which is important in our context, see \cite{AHL2019Stringy, AHL2021}.

A generalization of the Koba-Nielsen string integral was announced by Arkani-Hamed, Lam and Spradlin \cite{ALS21} using finite-type (Grassmannian) cluster algebras, where type $A$ cluster algebra corresponds to usual string amplitudes, and developed in detail in \cite{AHL2021}.  However, Grassmannian cluster algebras for $\Gr(k,n)$ with $(k-2)(n-k-2) > 3$ not only have infinitely-many cluster variables, but it turns out that not all physically relevant elements of Lusztig's dual canonical basis can be constructed using a finite sequence of cluster mutations (there are prime non-real elements in the dual canonical basis).  Finding a systematic description of prime modules is also a deep and important problem in theoretical physics, since it is exactly the cases $k=4$ and $n\ge 8$ which are of interest to amplitudes.

In the theory of quantum affine algebras, prime modules do not have an analog in the representation theory of simple Lie algebras: a module is prime if it cannot be decomposed nontrivially as the tensor product of two other modules.  The fact that representations of quantum affine algebras possess both additive and multiplicative structures appears to be very important in our context.

The main physical contribution of this work is to propose the definition of a Grassmannian string integral (as in \cite{AHL2019Stringy}) with an integrand involving a product which is now in general infinite; the key point is that the integrand should be indexed by prime tableaux, in which case our results will be directly applicable to study physical aspects of the integral.  Our proposal is valid for any Grassmannian $\Gr(k,n)$, where $k=2$ corresponds to the usual Koba-Nielsen string integral.  So for $2 \le k \le n-2$ and every $d\ge 1$, we define
\begin{eqnarray*} 
\mathbf{I}^{(d)}_{k,n} & =&  (\alpha')^{a}\int_{\left(\mathbb{R}_{>0}^{n-k-1}\right)^{\times (k-1)}}\left(\prod_{(i,j)}\frac{dx_{i,j}}{x_{i,j}}\right)\left(\prod_{T} \ch_T^{-\alpha'c_T}(x_{i,j})\right),
\end{eqnarray*}
where the second product is over all tableaux $T$ such that the face $\mathbf{F}_T$ corresponding to $T$ (see Section \ref{subsec: the face corresponding to T}) is a (codimension one) facet of $\mathbf{N}^{(d-1)}_{k,n}$, $a$, $\alpha'$, $c_T$ are some parameters, see Section \ref{subsec:stringy integeral for Grkn}, Formulas (\ref{eq: stringy integral finite d}), (\ref{eq: stringy integral infinite}), (\ref{eq: stringy integral using tableaux up to certain columns}), for more details.  We point out that the character polynomials $ \ch_T$ are manifestly positive in the interior of the totally nonnegative Grassmannian, see \cite[Section 5.3]{CDFL}.

We give a general formula of $u$-variables using prime tableaux (corresponding to prime modules of quantum affine algebras of type $A$) and Auslander-Reiten quivers of Grassmannian cluster categories ${\rm CM}(B_{k,n})$ \cite{JKS}. For every mesh 
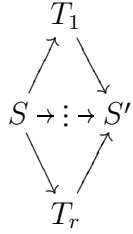
\begin{align*}
\xymatrix@-6mm@C-0.2cm{
& & T_1 \ar[rdd]\\
 \\
&  S  \ar[r]     \ar[rdd]\ar[ruu]&  \ar[r] \vdots   & S'  \\
 \\
& & T_r \ar[ruu]  \\
}    
\end{align*} 
in the Auslander-Reiten quiver of ${\rm CM}(B_{k,n})$, where we label the vertices by tableaux, we define the corresponding $u$-variable as 
\begin{align} \label{eq:u variables in Introduction}
u_{S} = \frac{\prod_{i=1}^{r} \ch_{T_i}}{\ch_{S} \ch_{S'}},
\end{align} 
see Definition \ref{def:u-variables}.  

We conjecture that there are unique integers $a_{T,T'}$, where $T$, $T'$ are prime tableaux, such that $u$-variables (\ref{eq:u variables in Introduction}) are solutions of the system of equations 
\begin{align} \label{eq:u-equations}
u_T  + \prod_{T' \in {\rm PSSYT}_{k,n}} u_{T'}^{a_{T, T'}} = 1,
\end{align} 
where ${\rm PSSYT}_{k,n}$ is the set of all prime tableaux with $k$ rows and with entries in $[n]$, and $T$ runs over all non-frozen prime tableaux in ${\rm PSSYT}_{k,n}$, see Conjecture \ref{conj:u variables are solutions of u equations general case for Grkn}. The equations (\ref{eq:u-equations}) are called $u$-equations. 

General $u$-equations have been introduced in \cite{AFPST21_surface, AFPST21, AFPST23a, AFPST23b} in the setting of representations of quiver with relations and cluster categories of finite type. In our paper, we work in the setting of the Grassmannian cluster algebra $\CC[\Gr(k,n)]$ and the Grassmannian cluster category ${\rm CM}(B_{k,n})$ \cite{JKS}. 

We also give a definition of stringy integral for Grassmannian cluster algebras using $u$-variables. Denote by ${\rm PSSYT}_{k,n}$ the set of all prime tableaux of rectangular shapes and with $k$ rows and with entries in $[n]$. For $k \le n$, we define
    \begin{eqnarray*}
        \mathbf{I}^{(\infty)}_{k,n} & = & (\alpha')^{a}\int_{\left(\mathbb{R}_{>0}^{n-k-1}\right)^{\times (k-1)}} \prod_{i,j} \frac{dx_{i,j}}{x_{i,j}} \prod_{T \in {\rm PSSYT}_{k,n}} (u_T)^{\alpha' U_T},  
    \end{eqnarray*}
where $\alpha'$, $U_T$ are some parameters, and $u_T$ is the $u$-variable corresponding to a prime tableau $T$, See Definition \ref{def:stringy integeral for Grassmannians using u-variables}. This new integrand involves an infinite product of $u$-variables, indexed by prime tableaux. 

 
We also generalize the stringy integral to the setting of general quantum affine algebras and define stringy integrals using prime modules of quantum affine algebras, see Section \ref{subsec:stringy integerals for quantum affine algebras}.

\subsection{}
Recently, Arkani-Hamed, Lam, Spradlin \cite{ALS21}, and Drummond, Foster, G\"{u}rdo\u{g}an, Kalousios \cite{DFGK}, and Henke, Papathanasiou \cite{HP20}, have constructed limit $g$-vectors for the Grassmannian cluster algebra $\CC[\Gr(k,n)]$ using infinite sequence of mutations. These limit $g$-vectors do not correspond to cluster variables in $\CC[\Gr(k,n)]$ but correspond to prime elements in the dual canonical basis of $\CC[\Gr(k,n)]$.

Motivated by these works, we define limit $g$-vectors for any cluster algebra of infinite type, see Definition \ref{def: limit g-vectors}. We say that a facet of a Newton polytope for a quantum affine algebra is a limit facet if the facet corresponds to a $U_q(\widehat{\mathfrak{g}})$-module whose $g$-vector is a limit $g$-vector of the cluster algebra for $U_q(\widehat{\mathfrak{g}})$. We say that a $U_q(\widehat{\mathfrak{g}})$-module corresponds to a limit $g$-vector if the $g$-vector of the module is a limit $g$-vector of the cluster algebra for $U_q(\widehat{\mathfrak{g}})$. We conjecture that every module corresponding to a limit $g$-vector is prime and non-real. This is another way to obtain prime non-real $U_q(\widehat{\mathfrak{g}})$-modules systematically. 

Using limit $g$-vectors, we construct new examples of non-real modules of quantum affine algebras. As an example, we prove that the module $L(Y_{2,-4}Y_{2,0})$ in type $D_4$ is non-real, see Section \ref{subsec:D4 nonreal example}. 

\subsection{}
The paper is organized as follows. In Section \ref{sec:quantum affine algebras}, we recall results of quantum affine algebras and Hernandez-Leclerc's category $\mathcal{C}_{\ell}$. In Section \ref{sec:Grassmannian cluster algebras and Tropical Grassmannians}, we recall results of Grassmannian cluster algebras and tropical Grassmannians. In Section \ref{sec:recursive definition of Newton polytopes}, we define a sequence of Newton polytopes and tropical fans for Grassmannian cluster algebras. In Section \ref{sec:tableaux and generalized root polytopes}, we study relations between semistandard Young tableaux of rectangular shapes and generalized root polytopes. In Section \ref{sec:rays to prime modules}, we construct prime modules from facets. 
In Section \ref{sec:Coarsest subdivisions and prime modules}, we prove Conjecture \ref{conj:prime modules are rays} in the cases of ${\bf N}_{3,9}^{(1)}$ and ${\bf N}_{4,8}^{(1)}$. In Section \ref{sec:Newton polytopes for quantum affine algebras}, we define Newton polytopes and tropical fans for general quantum affine algebras. In Section \ref{sec:physics}, we generalize Grassmannian string integrals and study $u$-equations and $u$-variables. In Section \ref{sec:discussion}, we discuss some future directions of the paper. 

\subsection{Acknowledgements}
The authors would like to thank Nima Arkani-Hamed, Freddy Cachazo, James Drummond, \"{O}mer G\"{u}rdo\u{g}an, Min Huang, 
Jiarui Fei, Lecheng Ren, Marcus Spradlin, and Anastasia Volovich for helpful discussions, and Georgios Papathanasiou for helpful comments.

Data sharing not applicable to this article as no datasets were generated or analysed during the current study.

This research was supported in part by the Munich Institute for Astro-, Particle and BioPhysics (MIAPbP) which is funded by the Deutsche Forschungsgemeinschaft (DFG, German Research Foundation) under Germany's Excellence Strategy–EXC-2094-390783311. JRL is supported by the Austrian Science Fund (FWF): P-34602, Grant DOI: 10.55776/P34602, and PAT 9039323, Grant-DOI 10.55776/PAT9039323. This research received funding from the European Research Council (ERC) under the European Union’s Horizon 2020 research and innovation programme (grant agreement No 725110), Novel structures in scattering amplitudes.  

\section{Quantum Affine Algebras and Hernandez-Leclerc's Category $\mathcal{C}_{\ell}$}
\label{sec:quantum affine algebras}

In this section, we recall results of quantum affine algebras \cite{CP94, FR98}, Hernandez-Leclerc's category $\mathcal{C}_{\ell}$ and cluster algebra structure on the Grothendieck ring of $\mathcal{C}_{\ell}$ \cite{HL10}. 

\subsection{Quantum Affine Algebras} \label{subsec:quantum affine algebras, qcharacters}

Let $\mathfrak{g}$ be a simple finite-dimensional Lie algebra and $I$ the set of vertices of the Dynkin diagram of $\mathfrak{g}$. Denote by $\{ \omega_i: i \in I \}$, $\{\alpha_i : i \in I\}$, $\{\alpha_i^{\vee} : i \in I\}$ the set of fundamental weights, the set of simple roots, the set of simple coroots, respectively. Denote by $P$ the integral weight lattice and
$P^+$ the set of dominant weights. The Cartan matrix is $C = (\alpha_j(\alpha_i^{\vee}))_{i,j \in I}$. Let $D = \diag(d_i: i \in I)$, where $d_i$'s are minimal positive integers such that $DC$ is symmetric.  

Let $z$ be an indeterminate. The quantum Cartan matrix $C(z)$ is defined as follows: for $i \in I$, $C_{i,i}(z) = z^{d_i} + z^{-d_i}$, and for $i \ne j \in I$, $C_{ij}(z) = [C_{ij}]_z$, where $[m]_z = \frac{z^m - z^{-m}}{z-z^{-1}}$, see \cite{FR98}.  

The {\sl quantum affine algebra} $U_q(\widehat{\mathfrak{g}})$ is a Hopf algebra that is a $q$-deformation of the universal enveloping algebra of $\widehat{\mathfrak{g}}$ \cite{Drin87, Jim85}. In this paper, we take $q$ to be a non-zero complex number which is not a root of unity. 

Denote by $\mathcal{P}$ the free abelian group generated by formal variables $Y_{i, a}^{\pm 1}$, $i \in I$, $a \in \CC^*$, denote by $\mathcal{P}^+$ the submonoid of $\mathcal{P}$ generated by $Y_{i, a}$, $i \in I$, $a \in \CC^*$. Let $\mathcal{C}$ denote the monoidal category of finite-dimensional representations of the quantum affine algebra $U_q(\widehat{\mathfrak{g}})$.

Any finite dimensional simple object in $\mathcal{C}$ is a highest $l$-weight module with a highest $l$-weight $m \in \mathcal{P}^+$, denoted by $L(m)$ (see \cite{CP95a}). The elements in $\mathcal{P}^+$ are called dominant monomials.  

Frenkel and Reshetikhin \cite{FR98} introduced the $q$-character map which is an injective ring morphism $\chi_q$ from the Grothendieck ring of $\mathcal{C}$ to $\mathbb{Z}\mathcal{P} = \mathbb{Z}[Y_{i, a}^{\pm 1}]_{i\in I, a\in \mathbb{C}^{\times}}$. For a $U_q(\widehat{\mathfrak{g}})$-module $V$, $\chi_q(V)$ encodes the decomposition of $V$ into common generalized eigenspaces for the action of a large commutative subalgebra of $U_q(\widehat{\mathfrak{g}})$ (the loop-Cartan subalgebra). These generalized eigenspaces are called $l$-weight spaces and generalized eigenvalues are called $l$-weights. One can identify $l$-weights with monomials in $\mathcal{P}$ \cite{FR98}. Then
the $q$-character of a $U_q(\widehat{\mathfrak{g}})$-module $V$ is given by (see \cite{FR98})
\begin{align*}
\chi_q(V) = \sum_{  m \in \mathcal{P}} \dim(V_{m}) m \in \mathbb{Z}\mathcal{P},
\end{align*}
where $V_{m}$ is the $l$-weight space with $l$-weight $m$. 

Let $\mathcal{Q}^+$ be the monoid generated by
\begin{align} \label{eq:Aia}
A_{i, a} = Y_{i, aq_{i}}Y_{i, aq_{i}^{-1}} \left(\prod_{j:C_{ji}=-1}Y_{j, a}^{-1}\right) \left(\prod_{j:C_{ji}=-2}Y_{j, aq}^{-1}Y_{j, aq^{-1}}^{-1}\right) \left(\prod_{j:C_{ji}=-3}Y_{j, aq^{2}}^{-1}Y_{j, a}^{-1}Y_{j, aq^{-2}}^{-1}\right),
\end{align}
where $q_i=q^{d_i}$, $i \in I$. 
There is a partial order $\preccurlyeq$ on $\mathcal{P}$ (cf. \cite{FR98}) defined by 
$ m' \preccurlyeq m \text{ if and only if } m {m'}^{-1} \in \mathcal{Q}^{+}$. 

For $i \in I$, $a \in \mathbb{C}^{\times}$, $k \in \ZZ_{\ge 1}$, the modules
\begin{align*}
X_{i,a}^{(k)} := L(Y_{i,a} Y_{i,aq_i^2} \cdots Y_{i,aq_i^{2k-2}})
\end{align*}
are called {\sl Kirillov-Reshetikhin modules}, where $q_i=q^{d_i}$ and $d_i$'s are entries of the matrix $D$ defined in the beginning of this section. The modules $X_{i,a}^{(1)} = L(Y_{i,a})$ are called {\sl fundamental modules}.

\subsection{Hernandez-Leclerc's Category $\mathcal{C}_{\ell}$ and Truncated $q$-characters} \label{subsec:HL category and cluster algebra}

We recall the definition of Hernandez-Leclerc's category $\mathcal{C}_{\ell}$ \cite{HL10}. 

For integers $a \le b$, we denote $[a,b] = \{i: a \le i \le b\}$ and $[a] = \{i: 1 \le i \le a\}$. 

Let $\mathfrak{g}$ be a simple Lie algebra over $\CC$ and let $\mathcal{C}$ be the category of finite-dimensional $U_q(\widehat{\mathfrak{g}})$-modules. In \cite{HL10}, \cite{HL16}, Hernandez and Leclerc introduced a full subcategory $\mathcal{C}_{\ell} = \mathcal{C}_{\ell}^{\mathfrak{g}}$ of $\mathcal{C}$ for every $\ell \in \mathbb{Z}_{\geq 0}$. Let $I$ be the set of vertices of the Dynkin diagram of $\mathfrak{g}$. We fix $a \in \CC^*$ and denote $Y_{i,s} = Y_{i,aq^s}$, $i \in I$, $s \in \ZZ$. For $\ell \in \ZZ_{\ge 0}$, denote by $\mathcal{P}_\ell$ the subgroup of $\mathcal{P}$ generated by $Y_{i,\xi(i)-2r}^{\pm 1}$, $i \in I$, $r \in [0, d \ell]$, where $d$ is the maximal diagonal element in the diagonal matrix $D$, and $\xi: I \to \ZZ$ is a certain function called height function defined in Section 2.2 in \cite{HL15} and Definition 4.1 in \cite{FHOO22}. Denote by $\mathcal{P}^+_\ell$ the submonoid of $\mathcal{P}^+$ generated by $Y_{i,\xi(i)-2r}$, $i \in I$, $r \in [0, d \ell]$. An object $V$ in $\mathcal{C}_{\ell}$ is a finite-dimensional $U_q(\widehat{\mathfrak{g}})$-module which satisfies the condition: for every composition factor $S$ of $V$, the highest $l$-weight of $S$ is a monomial in $\mathcal{P}^+_\ell$, \cite{HL10}. Simple modules in $\mathcal{C}_{\ell}$ are of the form $L(M)$ (see \cite{CP94}, \cite{HL10}), where $M \in \mathcal{P}_{\ell}^+$.

\begin{remark}
In \cite{HL10}, fundamental modules in $\mathcal{C}_{\ell}$ are $L(Y_{i, -\xi(i)-2r})$, $i \in I$, $r \in [0, \ell]$, see Definition 3.1 in \cite{HL10}. In this paper, we slightly modify $\mathcal{C}_{\ell}$ such that the fundamental modules in $\mathcal{C}_{\ell}$ are $L(Y_{i, \xi(i)-2r})$, $i \in I$, $r \in [0, d \ell]$, where $d$ is the maximal diagonal element in the diagonal matrix $D$.
\end{remark}

For a module $V$ in $\mathcal{C}_{\ell}$, the truncated $q$-character $\tchi(V)$ is the Laurent polynomial obtained from the $q$-character $\chi_q(V)$ defined in Section \ref{subsec:quantum affine algebras, qcharacters} by removing all monomials which have a factor $Y_{i,s}$ or $Y_{i,s}^{-1}$, where $Y_{i,s}$ is not in $\mathcal{P}_{\ell}^+$, see \cite{HL15, HL16}. 

Hernandez and Leclerc constructed a cluster algebra for every category $\mathcal{C}_{\ell}$ of every quantum affine algebra $U_q(\widehat{\mathfrak{g}})$ \cite{HL10, HL16}. The cluster algebra for $\mathcal{C}_{\ell}$ of $U_q(\widehat{\mathfrak{sl}_k})$ is isomorphic to the cluster algebra for a certain quotient $\CC[\Gr(k,n,\sim)]$ (see Section \ref{subsec:Grassmannian cluster algebras and semistandard Young tableaux}) of the Grassmannian cluster algebra $\CC[\Gr(k,n)]$ \cite{CDFL, HL10, Sco06}, $n=k+\ell+1$.

\section{Grassmannian Cluster Algebras and Tropical Grassmannians}	 \label{sec:Grassmannian cluster algebras and Tropical Grassmannians}
In this section, we recall results of Grassmannian cluster algebras, \cite{Sco06, CDFL} and tropical Grassmannians \cite{SS04, SW05}. 

\subsection{Grassmannian Cluster Algebras and Semistandard Young Tableaux}\label{subsec:Grassmannian cluster algebras and semistandard Young tableaux}

For $k \le n$, the Grassmannian $\Gr(k,n)$ is the set of $k$-dimensional subspaces in an $n$-dimensional vector space. In this paper, we denote by $\Gr(k,n)$ (the affine cone over) the Grassmannian of $k$-dimensional subspaces in $\CC^n$, and denote by $\CC[\Gr(k,n)]$ its coordinate ring. This algebra is generated by Pl\"{u}cker coordinates 
\begin{align*}
p_{i_1, \ldots, i_{k}}, \quad 1 \leq i_1 < \cdots < i_{k} \leq n.
\end{align*}
In this paper, we use $p_J$ for Pl\"{u}cker coordinates and $P_J$ for its tropical version, where $J$ is a $k$-element subset of $[n]$.

It was shown by Scott \cite{Sco06} that the ring $\CC[\Gr(k,n)]$ has a cluster algebra structure. Define $\CC[\Gr(k,n,\sim)]$ to be the quotient of $\CC[\Gr(k,n)]$ by the ideal generated by $P_{i, \ldots, i+k-1}-1$, $i \in [n-k+1]$. In \cite{CDFL}, it is shown that the elements in the dual canonical basis of $\CC[\Gr(k,n,\sim)]$ are in bijection with semistandard Young tableaux in ${\rm SSYT}(k, [n],\sim)$.

A semistandard Young tableau is a Young tableau with weakly increasing rows and strictly increasing columns. For $k,n \in \ZZ_{\ge 1}$, we denote by ${\rm SSYT}(k, [n])$ the set of rectangular semistandard Young tableaux with $k$ rows and with entries in $[n]$ (with arbitrarly many columns). The empty tableau is denoted by $\mathds{1}$. 

For $S,T \in {\rm SSYT}(k, [n])$, let $S \cup T$ be the row-increasing tableau whose $i$th row is the union of the $i$th rows of $S$ and $T$ (as multisets), \cite{CDFL}. It is shown in Section 3 in \cite{CDFL} that $S \cup T$ is semistandard for any pair of semistandard tableaux $S, T$.  

We call $S$ a factor of $T$, and write $S \subset T$, if the $i$th row of $S$ is contained in that of $T$ (as multisets), for $i \in [k]$. In this case, we define $\frac{T}{S}=S^{-1}T=TS^{-1}$ to be the row-increasing tableau whose $i$th row is obtained by removing that of of $S$ from that of $T$ (as multisets), for $i \in [k]$. 

A tableau $T \in {\rm SSYT}(k, [n])$ is trivial if each entry of $T$ is one less than the entry below it. For any $T \in {\rm SSYT}(k, [n])$, we  denote by $T_{\text{red}} \subset T$ the semistandard tableau obtained by removing a maximal trivial factor from $T$. For a trivial $T$, one has $T_{\text{red}} = \mathds{1}$. 

Let ``$\sim$'' be the equivalence relation on $S, T \in {\rm SSYT}(k, [n])$ defined by: $S \sim T$ if and only if $S_{\text{red}} = T_{\text{red}}$. We denote by ${\rm SSYT}(k, [n],\sim)$ the set of $\sim$-equivalence classes.

The elements in the dual canonical basis of $\CC[\Gr(k,n,\sim)]$ are in bijection with simple modules in the category $\mathcal{C}_{\ell}$ of $U_q(\widehat{\mathfrak{sl}_k})$ in Section \ref{subsec:Grassmannian cluster algebras and semistandard Young tableaux}, see \cite{HL10, CDFL}.

A one-column tableau is called a fundamental tableau if its content is $[i,i+k] \setminus \{r\}$ for $r \in \{i+1, \ldots, i+k-1\}$. 
Any tableau in $\SSYT(k,[n])$ is $\sim$-equivalent to a unique semistandard tableau whose columns are fundamental tableaux, see Lemma 3.13 in \cite{CDFL}.

By \cite[Theorem 5.8]{CDFL}, for every $T \in \SSYT(k,[n])$, the corresponding element $\widetilde{\ch}(T)$ in the dual canonical basis of $\CC[\Gr(k,n,\sim)]$ is given by 
\begin{align}\label{eq:formula of ch(T)}
\widetilde{\ch}(T) = \sum_{u \in S_m} (-1)^{\ell(uw_T)} {\bf p}_{uw_0, w_Tw_0}(1) p_{u; T'} \in \CC[\Gr(k,n,\sim)],
\end{align}
where $m$ is the number of columns of $T'$, $T'$ is the tableau whose columns are fundamental tableaux and such that $T \sim T'$, $p_{u ; T'}$ is a certain monomial of Pl\"{u}cker coordinates, $w_T$ is a certain permutation in $S_m$, ${\bf p}_{u,v}(q)$ is a Kazhdan-Lusztig polynomial, see \cite{CDFL}. Let $T'' = T' T^{-1}$ and define $\ch(T) = \frac{1}{p_{T''}} \widetilde{\ch}(T)$, where $p_{T''} = p_{{T''}_1} \cdots p_{{T''}_r}$, ${T''}_i$'s are columns of $T''$. We also denote $\ch_T = \ch(T)$. 

\subsection{Relation between Dominant Monomials and Tableaux} \label{subsec:dominant monomials and tableaux}
 
In Section \ref{subsec:HL category and cluster algebra}, we recalled Hernandez and Leclerc's category $\mathcal{C}_{\ell}$. It is shown in Theorem 3.17 in \cite{CDFL} that in the case of $\mathfrak{g}=\mathfrak{sl}_k$, the monoid $\mathcal{P}^+_\ell$ (we take the height function to be $\xi(i)=i-2$, $i \in [k-1]$) of dominant monomials is isomorphic to the monoid of semistandard Young tableaux $\SSYT(k, [n], \sim)$, $n = k+\ell+1$. The correspondence of dominant monomials and tableaux is induced by the following map sending variables $Y_{i,s}$ to fundamental tableaux:
\begin{equation}\label{eq:monomToTableaux} 
Y_{i,s} \mapsto T_{i,s}, 
\end{equation}
where $T_{i,s}$ is a one-column tableau consisting of entries $\frac{i-s}{2}, \frac{i-s}{2}+1, \ldots, \frac{i-s}{2}+k-i-1, \frac{i-s}{2}+k-i+1, \ldots, \frac{i-s}{2}+k$. We denote the monomial corresponding to a tableau $T$ by $M_T$ and denote the tableau corresponding to a monomial $M$ by $T_M$. Note that by the definition of $\mathcal{C}_{\ell}$ and the choice of the height function $\xi(i)=i-2$, $i \in [k-1]$, the indices of $Y_{i, s}$ in the highest $l$-weight monomials of simple modules in $\mathcal{C}_{\ell}$ satisfy $i-s \pmod 2=0$. 

When computing the monomial corresponding to a given tableau, we first decompose the tableau into a union of fundamental tableaux. Then we send each fundamental tableau to the corresponding $Y_{i,s}$. For example, the tableaux $[[1,2,4,6],[3,5,7,8]]$ (each list is a column of the tableau), $[[1,3,5,7],[2,4,6,8]]$ correspond to the modules
\begin{align*}
L(Y_{2,-6}Y_{1,-3}Y_{3,-3}Y_{2,0}), \quad L(Y_{1,-7}Y_{2,-4}Y_{1,-5}Y_{3,-1}Y_{2,-2}Y_{3,1}),
\end{align*}
respectively.   

Recall that a simple $U_q(\widehat{\mathfrak{g}})$-module $L(M)$ is called prime if it is not isomorphic to $L(M') \otimes L(M'')$ for any non-trivial modules $L(M')$, $L(M'')$ \cite{CP97}. A simple $U_q(\widehat{\mathfrak{g}})$-module $L(M)$ is called real if $L(M) \otimes L(M)$ is still simple \cite{Lec}. We say that a tableau $T$ is prime (resp. real) if the corresponding $U_q(\widehat{\mathfrak{sl}_k})$-module $L(M_T)$ is prime (resp. real). The problem of classification of prime $U_q(\widehat{\mathfrak{sl}_k})$-modules in the category $\mathcal{C}_{\ell}$ ($\ell \ge 0$) is equivalent to the problem of classification of prime tableaux in $\SSYT(k, [n], \sim)$, $n = k+\ell+1$, \cite{CDFL}. 


\subsection{Tropical Grassmannians} \label{subsec:tropical grassmannians}


We recall results of tropical Grassmannians \cite{SS04, SW05}.

The tropical Grassmannian ${\rm Trop}\, G(k,n)$, introduced in \cite{SS04}, parametrizes realizable tropical linear spaces; it is the tropical variety of the Pl\"{u}cker ideal of the Grassmannian $\Gr(k,n)$. For general $(k,n)$ the Pl\"{u}cker ideal contains higher degree generators and to calculate ${\rm Trop}\, G(k,n)$ quickly becomes a completely intractable problem, but for $k=2$, ${\rm Trop}\, G(2,n)$ is completely characterized by the tropicalization of the 3-term tropical Pl\"{u}cker relations.  We present the full definition and then immediately specialize to the so-called positive tropical Grassmannians.

\begin{definition}[\cite{SS04}]
    Given $e = (e_1,\dots, e_N) \in \mathbb{Z}_{\ge 0}^N$, denote $\mathbf{x}^e = x_1^{e_1}\dots x_N^{e_N}$.  Let $E \subset\mathbb{Z}^N_{\ge 0}$. If $f = \sum_{e\in E} f_e \mathbf{x}^e$ is nonzero, denote by ${\text{Trop}}\, (f)$ the set of all points $(X_1,\ldots, X_N)$ such that for the collection of numbers $\sum_{i=1}^N e_i X_i$ for $e$ ranging over $E$, the minimum of the collection is achieved at least twice.  We say that ${\text{Trop}}\, (f)$ is the tropical hypersurface associated to $f$.  The \emph{tropical Grassmannian} ${\text{Trop}}\, G(k,n)$ is the intersection of all tropical hypersurfaces ${\text{Trop}}\, (f)$ where $f$ ranges over all elements in the Pl\"{u}cker ideal.

\end{definition}
  On the other hand, in \cite{SW05}, Speyer-Williams introduced the \textit{positive} tropical Grassmannian ${\rm Trop}^+ G(k,n)$, which was later shown independently in \cite{arkani2020positive, SW20} to be equal to the positive Dressian, which is characterized by the 3-term tropical Pl\"{u}cker relations,
$$\pi_{Lac} + \pi_{Lbd} = {\rm min}\{\pi_{Lab} + \pi_{Lcd},\pi_{Lad} + \pi_{Lbc}\},$$
for each pair $(L,\{a,b,c,d\}) \in \binom{\lbrack n\rbrack}{k-2}\times \binom{\lbrack n\rbrack \setminus L}{4}$ with $a<b<c<d$.

Generalized positive roots were defined in \cite{CE2020} and developed in depth in \cite{E2021} in the context of root polytopes and CEGM scattering amplitudes \cite{CEGM2019}.
We now recall the definition of generalized positive roots\footnote{The name originates from \cite[Theorem 4.3]{E2021}, according to which the their convex hull, the generalized root polytope $\mathcal{R}^{(k)}_{n-k}$, admits a flag-unimodular triangulation which specializes to that the triangulation of the type $A$ root polytope of Gelfand-Graev-Postnikov in the context of hypergeometric systems \cite{GGP}.}.

We use $\binom{\lbrack n\rbrack}{k}^{nf}$ to denote the set of $k$-element subsets of $[n]$ which are nonfrozen, i.e., not of the form $[i,i+k-1]$ up to cyclic shifts.
\begin{definition}[\cite{CE2020}]
    Given any $J = \{j_1<\cdots<j_k\} \in \binom{\lbrack n\rbrack}{k}^{nf}$, the \textit{generalized positive root} $\gamma_J$ is the linear function on the space $\mathbb{T}^{k-1,n-k}$:
\begin{eqnarray}\label{eq: gen Positive root}
    \gamma_J = \sum_{t=j_1}^{j_2-2}\alpha_{1,t} + \sum_{t=j_2-1}^{j_3-3}\alpha_{2,t} + \cdots + \sum_{t=j_{k-1}-(k-2)}^{j_k-k}\alpha_{k-1,t}.
\end{eqnarray}
\end{definition}
When there is no risk of confusion we also call the vector $v_J$ dual to $\gamma$ a generalized positive root:
$$v_J = \sum_{t=j_1}^{j_2-2}e_{1,t} + \sum_{t=j_2-1}^{j_3-3}e_{2,t} + \cdots + \sum_{t=j_{k-1}-(k-2)}^{j_k-k}e_{k-1,t}.$$

Denote
\begin{align*}
X(k,n) = \left\{ g \in \Gr(k,n): \prod_{J} p_J(g) \ne 0 \right\}/(\CC^*)^n.
\end{align*}
We construct an embedding 
$$(\mathbb{CP}^{n-k-1})^{\times (k-1)} \hookrightarrow X(k,n)$$
of a Cartesian product of projective spaces into $X(k,n)$.

Define a $(k-1)\times (n-k-1)$ polynomial-valued matrix $M_{k,n}$ with entries $m_{i,j}(x)$, with $(i,j) \in \lbrack 1,k-1\rbrack \times \lbrack 1,n-k\rbrack$, defined by 
\begin{eqnarray*}
	m_{i,j}(\{x_{a,b}: (a,b) \in \lbrack i,k-1\rbrack \times \lbrack 1,j\rbrack\}) & = & (-1)^{k+i} \sum_{1\le b_i\le b_{i+1}\le \cdots \le b_{k-1}\le j}  x_{i,b_{i}}x_{i+1,b_{i+1}}\cdots x_{k-1,b_{k-1}}.
\end{eqnarray*}

For the embedding $(\mathbb{CP}^{n-k-1})^{\times k-1} \hookrightarrow X(k,n)$, we construct a $k\times (n-k)$ matrix $M$ (called web matrix \cite{SW05}) with $M_{k,n}$ as its upper right block:
\begin{eqnarray}\label{eq: pos parametrization}
    M & = & \begin{bmatrix}
	1 &  &  & 0 & m_{1,1}& \cdots  & m_{1,n-k} \\
	& \ddots &  &  &\vdots & \ddots & \vdots \\
	&  & 1 &  & m_{k-1,1} & & m_{k-1,n-k} \\
	0&  &  & 1 & 1 & \cdots & 1 \\
\end{bmatrix}.
\end{eqnarray}
For instance, for rank $k=3$ we have
\begin{eqnarray}
	M_{3,6} & = & \begin{bmatrix}
		x_{1,1}x_{2,1} & x_{1,1}x_{2,12}+x_{1,2}x_{2,2}  & x_{1,1}x_{2,123} +x_{1,2}x_{2,23} +x_{1,3}x_{2,3} \\
		-x_{2,1} & -x_{2,12} & -x_{2,123}\nonumber
	\end{bmatrix}
\end{eqnarray}
and for the embedding we have
$$M = \begin{bmatrix}
	1 & 0 & 0 & x_{1,1}x_{2,1} & x_{1,1}x_{2,12}+x_{1,2}x_{2,2} & x_{1,1}x_{2,123} + x_{1,2}x_{2,23} + x_{1,3}x_{2,3} \\
	0& 1 & 0 & -x_{2,1} & -x_{2,12} & -x_{2,123}\\
	0 & 0 & 1 & 1 & 1 & 1 
	\nonumber
\end{bmatrix}.$$
Here we abbreviate for example $x_{i,23} = x_{i,2} + x_{i,3}$.

For any $k$-subset $J$ of $\lbrack n\rbrack$, define $e_J = \sum_{j\in J}e_j.$  Further denote by $\{e^J: J \in \binom{\lbrack n\rbrack}{k}\}$ the standard basis of $\mathbb{R}^{\binom{n}{k}}$.

\begin{definition} [{\cite{LZ98}}] \label{def:weakly separted k-subsets}
    A pair of $k$-element subsets $I,J$ is said to be \textit{weakly separated} if the difference of indicator vectors $e_I - e_J$ alternates sign at most twice, that is one does not have the pattern $e_a-e_b+e_c-e_d$ for $a<b<c<d$, up to cyclic rotation.
\end{definition}

\begin{definition}[{\cite{SSW17}}] \label{def:noncrossing pairs}
A pair $I=\{i_1< \ldots < i_k\}$, $J=\{j_1<\ldots<j_k\}$ of $k$-subsets of $[n]$ is said to be noncrossing if for each $1 \le a < b \le k$, either the pair $\{i_a, i_{a+1}, \ldots, i_b\}$, $\{j_a, j_{a+1}, \ldots, j_b\}$ is weakly separated, or $\{i_{a+1}, \ldots, i_{b-1}\} \ne \{j_{a+1}, \ldots, j_{b-1}\}$. 
\end{definition}

\begin{remark}
   Here it is easy to see that Definition \ref{def:weakly separted k-subsets} is a slight reformulation of the original construction given in \cite{LZ98}.  Similarly, Definition \ref{def:noncrossing pairs} is clearly equivalent to the one given in \cite{SSW17}.
\end{remark}

Denote by ${\bf NC}_{k,n}$ the poset of all collections of pairwise noncrossing nonfrozen $k$-element subsets of $[n]$, ordered by inclusion.

Denote $\mathbb{T}^{k-1,n-k} = (\mathbb{T}^{n-k})^{\times (k-1)}$, where $\mathbb{T}^{n-k} = \RR^{n-k}/\RR(1, \ldots, 1)$.

\section{Newton Polytopes and Tropical Fans for Grassmannian Cluster Algebras} \label{sec:recursive definition of Newton polytopes}

In this section, we define Newton polytopes and tropical fans for Grassmannian cluster algebras. 

\subsection{Newton polytopes for Grassmannian cluster algebras} \label{subsec:Newton polytopes for Grassmannian cluster algebras}

In what follows, we give a recursive construction of a collection of Newton polytopes ${\bf N}_{k,n}^{(d)}$ ($d \in \ZZ_{\ge 0}$), starting from the Planar Kinematics (PK) polytope $\Pi_{k,n}$ \cite{CE2020}, see also \cite{E2021}, which is equal to ${\bf N}^{(0)}_{k,n}$ in the present notation.

For any tableau $T \in \SSYT(k, [n])$, we evaluate $\ch_T$ on the web matrix $M$ in (\ref{eq: pos parametrization}) and obtain a polynomial in $x_{i,j}$-coordinates. 
Note that there is a monomial transformation relating the $x_{i,j}$ coordinates to the so-called $X$-coordinates \cite{FG09} in cluster algebras: $X_{ij} = \frac{x_{k-i,j}}{x_{k-i,j+1}}$.

For any tableaux $T$ with columns $T_1,\ldots, T_r$, define 
\begin{align} \label{eq: vT}
\gamma_T = \gamma_{T_1} + \cdots + \gamma_{T_r}, \quad v_T = v_{T_1} + \cdots + v_{T_r},
\end{align}
where $\gamma_{T_i}=\gamma_{J_i}$ and $v_{T_i} = v_{J_i}$ are defined in Section \ref{subsec:tropical grassmannians} and $J_i$ is the sorted content of the one-column tableau $T_i$. 

\begin{definition} \label{def:newton polytopes for grassmannian cluster algebras version 1}
Let $\mathcal{T}^{(0)}_{k,n}$ be the set of all one-column tableaux which are obtained by cyclic shifts of the one-column tableau with entries $1,2,\ldots,k-1,k+1$. For $d \ge 0$, we define recursively
\begin{align*}
{\bf N}_{k,n}^{(d)} = {\rm Newt}\left(\prod_{T \in \mathcal{T}_{k,n}^{(d)}} \ch_T(x_{i,j}) \right),
\end{align*}
where $\mathcal{T}_{k,n}^{(d+1)}$ is the set of all tableaux which correspond to facets of ${\bf N}_{k,n}^{(d)}$. 
\end{definition}
More precisely, 
$$\mathcal{T}^{(d+1)}_{k,n} = \left\{T: \gamma_T \text{ is minimized on a facet of } {\bf N}^{(d)}_{k,n}\right\},$$
see Section \ref{sec:rays to prime modules} for details on computing $\mathcal{T}^{(d)}_{k,n}$.

In particular, the so-called \textit{Planar Kinematics} (PK) polytope \cite{CE2020}, denoted there   
$$ \Pi_{k,n} = \text{Newt}\left(\prod_{j=1}^n \frac{p_{j,j+1,\ldots, j+k-2,k}}{p_{j,j+1,\ldots, j+k-2,k-1}} \right),$$
is the same as ${\bf N}^{(0)}_{k,n}$ noting that when we evaluated on the matrix $M$ all denominators are monomials in the $x_{i,j}$ coordinates. For example, evaluating on the matrix $M$ we have
$$\mathbf{N}^{(0)}_{2,6} = {\rm Newt}\left( \frac{x_{1,12} x_{1,23}x_{1,34}x_{1,1234}}{x_{1,1}x_{1,2}x_{1,3}x_{1,4}} \right)$$
and
\begin{eqnarray*}
   \mathbf{N}^{(0)}_{3,6} & =& {\rm Newt}\left( \frac{\left(x_{1,1} x_{2,1}+x_{1,1} x_{2,2}+x_{1,2} x_{2,2}\right)  \left(x_{1,2} x_{2,2}+x_{1,2} x_{2,3}+x_{1,3} x_{2,3}\right)\left(x_{1,123}\right)\left(x_{2,123}\right) }{x_{1,1} x_{1,2} x_{1,3} x_{2,1} x_{2,2} x_{2,3}}\right),
\end{eqnarray*}
where we abbreviate for example $x_{1,123} = x_{1,1} + x_{1,2} + x_{1,3}$.

On the other hand, the polytope ${\bf N}^{(1)}_{k,n}$, which is closely related\footnote{In particular, the normal fan of $\mathbf{N}^{(1)}_{k,n}$ has the following property: its cones are in bijection with the cones in the positive tropical Grassmannian $\text{Trop}^+G(k,n)$.} to the positive tropical Grassmannian, is given by 
$${\bf N}^{(1)}_{k,n} = \text{Newt}\left(\prod_{J \in \binom{\lbrack n\rbrack}{k}} p_J\right).$$

\begin{remark}
We are concerned with the facets of polytopes ${\bf N}^{(0)}_{k,n},{\bf N}^{(1)}_{k,n},\ldots,{\bf N}^{(d)}_{k,n}$.  Motivated in part by work of Arkani-Hamed, Frost, Plamondon, Salvatori, and Thomas, \cite{AFPST21_surface} on polyhedra modeled on punctured surfaces which they call surfacehedra, having infinitely many Minkowski summands, ultimately ($d 
\to \infty$) we are interested in a new object (denoted by ${\bf N}^{(\infty)}_{k,n}$) which again has infinitely many Minkowski summands (and infinitely many facets).  In our proposal, these Minkowski summands are by construction in bijection with prime tableaux in $\SSYT(k, [n])$ (equivalently prime modules of the quantum affine algebra in the category $\mathcal{C}_{\ell}$).
\end{remark}

We define another version of Newton polytopes non-recursively. For $k \le n$ and $r \in \ZZ_{\ge 1}$, denote by $\SSYT_{k,n}^r$ the set of all tableaux in $\SSYT(k,[n])$ with $r$ or less columns. 
\begin{definition} \label{def:another version of Newton polytopes for Grkn}
For $k \le n$ and $d \in \ZZ_{\ge 1}$, we define 
\begin{align*}
{\bf N'}_{k,n}^{(d)} = {\rm Newt}  \left( \prod_{T \in \SSYT_{k,n}^d} \ch_T(x_{i,j})\right).
\end{align*}
\end{definition}

\subsection{Tropical Fans for Grassmannian Cluster Algebras} \label{subsec:tropical fans for Grassmannian cluster algebras}

Recall that given a polytope $P$ in a real vector space $V$, its normal fan $\mathcal{N}(P)$ is the polyhedral complex on the dual space $V^*$, (closed) faces consist of all linear functionals minimized on a given face of $P$. In the following, we describe the normal fan of the Newton polytope ${\bf N}_{k,n}^{(d)}$ defined in Section \ref{subsec:Newton polytopes for Grassmannian cluster algebras}.

The evaluation of $\ch_T = \ch(T)$ on the web matrix $M$ \cite{SW05} (see Section \ref{subsec:tropical grassmannians}), we obtain a subtraction free polynomial in the $x_{i,j}$ coordinates. For example, for $\Gr(2,5)$, we have that 
\begin{align*}
& p_{1,2}(M) = p_{1,3}=p_{1,4}=p_{1,5} =1, \ p_{2,3}(M) = x_{1,1}, \ p_{3,4}(M) = x_{1,2}, \ p_{4,5}(M)=x_{1,3}, \\
& p_{2,4}(M) = x_{1,1} + x_{1,2}, \ p_{2,5}(M) = x_{1,1} + x_{1,2} + x_{1,3}, \ p_{3,5}(M) = x_{1,2} + x_{1,3}.
\end{align*}

Recall that we denote by $\mathcal{T}_{k,n}^{(0)}$ the set of all tableaux obtained by cyclic shifts of the one-column tableau with entries $1,2,\ldots, k-1,k+1$. By tropicalizing all $\ch_T(M)$, $T \in \mathcal{T}_{k,n}^{(0)}$, we obtain piecewise linear functions in the space of dimension $(k-1)(n-k-1)$ parametrized by $y_{i,j}$ ($y_{i,j}$ is the tropical version of $x_{i,j}$).  Such a function is linear on a collection of cones; these cones assemble to define a polyhedral fan.  The common refinement of these fans is the normal fan $\mathcal{N}(\mathbf{N}^{(0)}_{k,n})$ of the Newton polytope $\mathbf{N}^{(0)}_{k,n}$. By \cite[Corollary 10.5]{E2021}, the set of rays of $\mathcal{N}(\mathbf{N}^{(0)}_{k,n})$ is given by 
$$\left\{{\rm Ray}(v_J): J \in \binom{\lbrack n\rbrack}{k}^{nf}\right\},$$
where ${\rm Ray}(v) = \{c v:c\ge 0\}$ is the ray in the direction of a vector $v$.

For $d\ge 1$, the normal fan $\mathcal{N}(\mathbf{N}^{(d)}_{k,n})$ 
can be constructed from $\mathcal{N}(\mathbf{N}^{(d-1)}_{k,n})$ as follows.  Let $\mathcal{T}_{k,n}^{(d)}$ be the set of all tableaux corresponding to rays of $\mathcal{N}(\mathbf{N}^{(d-1)}_{k,n})$, that is 
$$\mathcal{T}_{k,n}^{(d)} = \left\{T: {\rm Ray}(v_T) \text{ is a ray of }\mathcal{N}(\mathbf{N}^{(d-1)}_{k,n}) \right\}.$$
Here $v_T$ is defined in Equation \eqref{eq: vT} and the construction of tableaux from rays is given in Section \ref{sec:rays to prime modules}. Indeed, in Section \ref{sec:rays to prime modules} we construct tableaux from facets of Newton polytopes. The construction of tableaux from rays of normal fans is the same. 

\begin{remark}
Tropical fans for Grassmannian cluster algebras have been defined in \cite{BBRS20, DFGK}, by tropical evaluations of finite sets of cluster variables. The tropical fans $\mathcal{N}(\mathbf{N}^{(d)}_{k,n})$ defined here use not only cluster variables but also other prime elements in the dual canonical basis of $\CC[\Gr(k,n)]$. 
\end{remark}

\subsection{Relation with positive tropical Grassmannians} \label{subsec:map Fkn and relation between normal fans and positive tropical Grassmannians}
Recall that we use $\binom{\lbrack n\rbrack}{k}^{nf}$ to denote the set of $k$-element subsets of $[n]$ which are nonfrozen, i.e., not of the form $[i,i+k-1]$ up to cyclic shifts, and recall that $\{e^J: J \in \binom{\lbrack n\rbrack}{k}\}$ is the standard basis of $\mathbb{R}^{\binom{n}{k}}$. For each $J \in \binom{\lbrack n\rbrack}{k}^{nf}$, recall that $\mathfrak{h}_J \in \mathbb{R}^{\binom{\lbrack n\rbrack}{k}}$ \cite{E2021, E2022} is defined by 
\begin{align} \label{eq:height function hJ}
\mathfrak{h}_J = -\frac{1}{n}\sum_{I \in \binom{\lbrack n\rbrack}{k}}\min\left\{L_1(e_J-e_I),L_2(e_J-e_I),\ldots, L_n(e_J-e_I)   \right\}e^I,
\end{align}
where 
$$L_j(x) = x_{j+1}+2x_{j+2} + \cdots +(n-1)x_{j-1}.$$
 
The \textit{lineality subspace} $\text{Lin}_{k,n}$ of $\mathbb{R}^{\binom{n}{k}}$ is defined by 
\begin{align*}
\text{Lin}_{k,n} = \text{span}\left\{\sum_{J \in {[n] \choose k}, J \ni j} e^{J},\ j=1,\ldots, n\right\},
\end{align*} 
see \cite[Definition 2.1]{E2022}. Clearly, $\dim(\text{Lin}_{k,n}) = n$. 

For each $J = \{j_1,\ldots, j_k\}\in \binom{\lbrack n\rbrack}{k}^{nf}$, define a cubical collection of $k$-element subsets of $\{1,\ldots, n\}$ by 
$$\mathcal{U}(J) = \left\{\{j_1+t_1,\ldots, j_k+t_k\}: t_i \in \{0,1\}, \text{ and } t_i = 0 \text{ whenever } j_i+1 \in J \right\},$$
where addition is modulo $n$, see \cite{E2021}.

Denote by $\omega_J(y)$ \cite{E2021} the tropical planar cross-ratio
$$\omega_J = \sum_{J' \in \mathcal{U}(J)} (-1)^{k-\#(J' \cap J)+1 }P_{J'}(y),$$
where $P_{J'}(y) = \text{Trop}(p_{J'})(y)$
is the tropicalization of the Pl\"{u}cker coordinate $p_{J'}(x)$, evaluated on the web matrix $M=(x_{i,j})_{k \times n}$ in Section \ref{subsec:tropical grassmannians}.	
 
Denote by $\mathcal{F}^{(k)}_n: \mathbb{R}^{(k-1)\times (n-k)} \rightarrow \mathbb{R}^{\binom{n}{k}}\slash\text{Lin}_{k,n}$ the map
	\begin{eqnarray}\label{eq: potential tropical cross ratio}
		\mathcal{F}^{(k)}_n(y) & = & \sum_{J \in \binom{\lbrack n\rbrack}{k}^{nf}}\omega_J(y)\mathfrak{h}_J.
\end{eqnarray}

The normal fan $\mathcal{N}(\mathbf{N}^{(1)}_{k,n})$ defined in Section \ref{subsec:tropical fans for Grassmannian cluster algebras} has been shown \cite[Proposition 11.5]{arkani2020positive} to satisfy the following property: its cones are in bijection with the cones in the positive tropical Grassmannian $\text{Trop}^+G(k,n)$, as defined by Speyer and Williams \cite{SW05}.  In particular, this bijection is achieved via the piecewise-linear map $\mathcal{F}^{(k)}_n$, which is equal (modulo a change in parameterization) to the map $\text{Trop}(\Phi_2)$ defined in \cite[Section 4]{SW05}, see also \cite{SW20}.

\section{Semistandard Young Tableaux and Generalized Root Polytopes} \label{sec:tableaux and generalized root polytopes}

In this section, we study relation between semistandard tableaux and generalized root polytopes. 

\subsection{Isomorphism of Monoids} \label{subsec:isomorphism of monoids}

Recall that the set $\SSYT(k,[n], \sim)$ of all $\sim$-equivalence classes of semistandard Young tableaux of rectangular shape with $k$ rows and with entries in $[n]$ form a monoid under the multiplication ``$\cup$'' \cite{CDFL}, see Section \ref{subsec:Grassmannian cluster algebras and semistandard Young tableaux}. This monoid is isomorphic to the monoid $\mathcal{P}_{\ell}^+$ of dominant monomials in $\mathcal{C}_{\ell}^{\mathfrak{sl}_k}$, $n = k+\ell+1$. 

The vector space $\ZZ_{\ge 0}^{(k-1)\times (n-k)}$ form a monoid $(\ZZ_{\ge 0}^{(k-1)\times (n-k)}, +)$ generated by $e_{i,j}$, $i \in [k-1]$, $j \in [n-k]$, where $e_{i,j}$'s form a standard basis of $\ZZ_{\ge 0}^{(k-1)\times (n-k)}$ ($e_{i,j}$'s also form a standard basis of $\RR^{(k-1)\times (n-k)}$). 

\begin{lemma} \label{lem:isomorphism of monoid of tableaux and monoid of higher roots}
We have an isomorphism of monoids 
\begin{align*}
(\SSYT(k, [n], \sim), \cup) \to (\ZZ_{\ge 0}^{(k-1)\times (n-k)}, +). 
\end{align*}
\end{lemma}

\begin{proof}
For $i \in [k-1]$, $j \in [n-k]$, denote by $T_{i,j}$ the fundamental tableau with entries $[j,j+k]\setminus \{i+j\}$. 

By Lemma 3.13 in \cite{CDFL}, every tableau in $\SSYT(k, [n], \sim)$ is $\sim$-equivalent to the union of a set of fundamental tableaux. The isomorphism $(\SSYT(k, [n], \sim), \cup) \to (\ZZ_{\ge 0}^{(k-1)\times (n-k)}, +)$ is induced by $T_{i,j} \mapsto e_{i,j}$. The inverse isomorphism is given as follows. Every element $v$ in $(\ZZ_{\ge 0}^{(k-1)\times (n-k)}, +)$ can be written as $v = \sum_{i,j} c_{i,j} e_{i,j}$ for some positive integers $c_{i,j}$. Let $T_v = \cup_{i,j} T_{i,j}^{\cup c_{i,j}}$. The inverse isomorphism is given by $v \mapsto T_v$.
\end{proof}

We denote by $T_v$ the tableau in $\SSYT(k, [n], \sim)$ corresponding to $v \in \ZZ_{\ge 0}^{(k-1)\times (n-k)}$ and denote by $v_T$ the element in $\ZZ_{\ge 0}^{(k-1)\times (n-k)}$ corresponding to $T \in \SSYT(k, [n], \sim)$. 

\subsection{Generalized Root Polytopes}

For any collection $\mathcal{J}  = \{J_1,\ldots, J_{m}\}\in \mathbf{NC}_{k,n}$ of nonfrozen subsets $J_i$, define 
\begin{eqnarray*}
	\lbrack \mathcal{J}\rbrack & = & \text{Convex hull}(\left\{0,v_{J_1},\ldots, v_{J_m}\right\}).
\end{eqnarray*}
The generalized root polytope $\mathcal{R}^{(k)}_{n-k}$ is the convex hull of all generalized positive roots $v_J$,
$$\mathcal{R}^{(k)}_{n-k} = \left\{v_J \in \mathbb{T}^{k-1,n-k}: J \in \binom{\lbrack n\rbrack}{k}^{nf} \right\},$$
where we remind that $\mathbb{T}^{k-1,n-k} = (\mathbb{T}^{n-k})^{\times (k-1)}$ and $\mathbb{T}^{n-k} = \RR^{n-k}/\RR(1, \ldots, 1)$.

\begin{theorem} [{\cite[Theorem 1.2]{E2021}}] \label{thm: triangulation}
	The set of simplices $\left\{\lbrack \mathcal{J}\rbrack: \mathcal{J} \in \mathbf{NC}_{k,n} \right\}$ defines a flag, unimodular triangulation of $\mathcal{R}^{(k)}_{n-k}$: simplices in the triangulation have equal volume and are in bijection with pairwise noncrossing collections of nonfrozen $k$-element subsets.  In particular, the set of cones 
	$$\mathcal{C}_\mathcal{J} = \left\{\sum_{J \in \mathcal{J}}c_J v_J: c_J>0\right\}$$
	assemble to define a complete simplicial in $\mathbb{T}^{k-1,n-k}$, and any point in $\mathbb{T}^{k-1,n-k}$ lies in the relative interior of a unique cone in the fan.
\end{theorem}

Now under the isomorphism in Lemma \ref{lem:isomorphism of monoid of tableaux and monoid of higher roots}, one-column tableaux correspond to generalized positive roots; thus, Theorem \ref{thm: triangulation} says that any linear combination of generalized positive roots with real coefficients decomposes uniquely as a linear combination, with positive coefficients, indexed by a pairwise noncrossing collection.  This means that if we restrict to integer coefficients then the triangulation has a beautiful representation-theoretic interpretation in terms of tableaux!  

\begin{example}
    Let 
    $$v = -v_{1,5,9}+2 v_{2,6,10}+3 v_{3,7,11}+4 v_{4,8,12}.$$
    Then $v$ has the following noncrossing decomposition using Theorem \ref{thm: triangulation}:
    \begin{eqnarray*}
        v & = & v_{1,2,6}+v_{2,8,10}+v_{2,9,10}+3 v_{3,8,10}+2 v_{4,6,10}+3 v_{4,7,10}+v_{7,8,10}
    \end{eqnarray*}
\end{example}

\section{From Facets to Prime Modules} \label{sec:rays to prime modules}
In this section, we describe a procedure to produce a simple $U_q(\widehat{\mathfrak{sl}_k})$-module from every facet of the Newton polytope defined in Section \ref{sec:recursive definition of Newton polytopes} and we conjecture that the obtained simple $U_q(\widehat{\mathfrak{sl}_k})$-module is prime. 

\subsection{A Procedure to Produce a Simple $U_q(\widehat{\mathfrak{sl}_k})$-module from a Given Facet}

Adapting the results of \cite{CDFL} (see Section \ref{subsec:dominant monomials and tableaux}), it suffices to give a procedure to produce a semistandard Young tableau from a given facet. 

The Newton polytope ${\bf N}^{(d)}_{k,n}$ defined in Section \ref{sec:recursive definition of Newton polytopes} is described using certain equations and inequalities in its H-representation (represent the polytope by an intersection half-spaces and hyperplanes). Let $F$ be a facet of the Newton polytope ${\bf N}^{(d)}_{k,n}$. The normal vector $v_F$ of $F$ is the coefficient vector in one of the inequalities in the H-representation of ${\bf N}^{(d)}_{k,n}$. If there is an entry of the vector $v_F$ which is negative, then we add some vectors which are coefficients of the equations in the H-representation of ${\bf N}^{(d)}_{k,n}$ to $v_F$ such that the resulting vector $v_F'$ all have non-negative entries.  The vector $v_F'$ can be written as $v_F' = \sum_{i,j} c_{i,j}e_{i,j}$ for some positive integers $c_{i,j}$, where $e_{i,j}$ is the standard basis of $\RR^{(k-1) \times (n-k)}$. By Lemma \ref{lem:isomorphism of monoid of tableaux and monoid of higher roots}, each $e_{i,j}$ corresponds to a fundamental tableau $T_{i,j}$ which is defined to be the one-column tableau with entries $[j,j+k]\setminus \{i+j\}$. The tableau $T_F$ corresponding to $F$ is obtained from $\cup_{i,j} T_{i,j}^{\cup c_{i,j}}$ by removing all frozen factors (if any). 

\begin{conjecture} \label{conj:prime modules are rays}
\begin{enumerate}
\item For $d \in \ZZ_{\ge 0}$, $k \le n$, and every facet $F$ of the Newton polytope ${\bf N}^{(d)}_{k,n}$, we have that the corresponding tableau $T_F$ is prime.

\item For $k \le n$ and every nonfrozen prime tableau $T$ in $\SSYT(k, [n])$, there is $d \ge 1$ and a facet $F$ of ${\bf N}^{(d)}_{k,n}$ such that $T = T_F$. 
\end{enumerate}
\end{conjecture}

Conjecture \ref{conj:prime modules are rays} gives systematic procedure to construct all prime $U_q(\widehat{\mathfrak{sl}_k})$-modules.


\subsection{Example: $\Gr(3,6)$} \label{subsec:normal vectors for Gr36}

In the case of $\CC[\Gr(3,6)]$, we use the web matrix (see Section \ref{subsec:tropical grassmannians})
\begin{align*}
M=\begin{bmatrix}
	1 & 0 & 0 & x_{1,1}x_{2,1} & x_{1,1}x_{2,12}+x_{1,2}x_{2,2} & x_{1,1}x_{2,123} + x_{1,2}x_{2,23} + x_{1,3}x_{2,3} \\
	0& 1 & 0 & -x_{2,1} & -x_{2,12} & -x_{2,123}\\
	0 & 0 & 1 & 1 & 1 & 1 
	\nonumber
\end{bmatrix},
\end{align*}
where we abbreviate for example $x_{2,23} = x_{2,2} + x_{2,3}$.  Evaluating all Pl\"{u}cker coordinates on $M$ and take their product, we obtain a polynomial $p$. The Newton polytope ${\bf N}_{3,6}^{(1)}$ is the Newton polytope defined by the vertices given by the exponents of monomials of $p$. The H-representation of ${\bf N}_{3,6}^{(1)}$ is given by 
\begin{align} \label{eq:hyperplanes and half surface of N1 Gr36}
\begin{split}
& (0, 0, 0, 1, 1, 1) \cdot x - 20 =0, \ (1, 1, 1, 0, 0, 0) \cdot x - 10 = 0, \ (0, 1, 1, 0, 0, 0) \cdot x -4  \ge 0, \\
&  (0, 0, 1, 0, 0, 0) \cdot x -1 \ge 0,\ (0, 0, 0, 0, 1, 1) \cdot x - 11 \ge 0,\ (0, 0, 0, 0, 0, 1) \cdot x - 4 \ge 0, \\
&   (0, 0, 1, 1, 0, 0) \cdot x - 6 \ge 0,\ (0, 0, 0, 0, 1, 0) \cdot x - 4 \ge 0,\ (0, 0, 0, 1, 0, 0) \cdot x -4 \ge 0,\\ 
&   (1, 0, 0, 0, 0, 0) \cdot x - 1 \ge 0,\ (1, 0, 0, 0, 1, 0) \cdot x - 6 \ge 0,\ (1, 1, 0, 0, 1, 1) \cdot x - 16 \ge 0,\\
&   (1, 1, 0, 0, 0, 0) \cdot x - 4 \ge 0,\ (0, 0, 0, 1, 1, 0) \cdot x -11 \ge 0,\ (0, 1, 0, 0, 0, 0) \cdot x - 1 \ge 0,\\ 
&   (1, 0, 0, 0, 1, 1) \cdot x - 14 \ge 0,\ (0, 1, 0, 0, 0, 1) \cdot x - 6 \ge 0,\ (1, 1, 0, 0, 0, 1) \cdot x - 11 \ge 0,
\end{split}
\end{align}
where $(0, 0, 0, 1, 1, 1) \cdot x$ is the inner product of the vectors $(0, 0, 0, 1, 1, 1)$ and $x$.

Now we compute the tableau corresponding to each facet. For example, for the facet $F$ with the normal vector $v_F = (0, 1, 1, 0, 0, 0)$ in the first line of (\ref{eq:hyperplanes and half surface of N1 Gr36}), we have that $v_F = e_{1,2}+e_{1,3}$. The generalized roots $e_{1,2}$, $e_{1,3}$ corresponds to tableaux $\scalemath{0.6}{ \begin{ytableau}
    2 \\ 4 \\ 5
\end{ytableau}}$, $\scalemath{0.6}{ \begin{ytableau}
    3 \\ 5 \\ 6
\end{ytableau}}$ respectively. Removing the frozen factor $\scalemath{0.6}{ \begin{ytableau}
    3 \\ 4 \\ 5
\end{ytableau}}$ in $\scalemath{0.6}{ \begin{ytableau}
    2 \\ 4 \\ 5
\end{ytableau} \cup \begin{ytableau}
    3 \\ 5 \\ 6
\end{ytableau} } = \scalemath{0.6}{ \begin{ytableau}
    2 & 3 \\ 4 & 5 \\ 5 & 6
\end{ytableau}}$, we obtain $T_F = \scalemath{0.6}{ \begin{ytableau}
    2 \\ 5 \\ 6
\end{ytableau}}$. 
Similarly, we have the correspondence in Table \ref{table:correspondence of higher roots rays tableaux dominant monomials Gr36}.

\begin{table} 
\begin{tabular}{|c|c|c|c|c|} 
\hline
generalized roots & facets, hyperplanes & tableaux & modules \\
\hline
 $\gamma _{124}=\alpha _{2,1}$ & $(0,0,0,1,0,0)$ & $[124]$ & $Y_{1,-1}$ \\
 $\gamma _{125}=\alpha _{2,1}+\alpha _{2,2}$ & $(0,0,0,1,1,0)$  &  $[125]$ & $Y_{1,-3}Y_{1,-1}$ \\
 $\gamma _{134}=\alpha _{1,1}$ & $(1,0,0,0,0,0)$  & $[134]$ & $Y_{2,0}$ \\
 $\gamma _{135}=\alpha _{1,1}+\alpha _{2,2}$ & $(1,0,0,0,1,0)$  & $[135]$ & $Y_{1,-3}Y_{2,0}$ \\
 $\gamma _{136}=\alpha _{1,1}+\alpha _{2,2}+\alpha _{2,3}$ & $(1,0,0,0,1,1)$  & $[136]$ & $Y_{1,-5}Y_{1,-3}Y_{2,0}$ \\
 $\gamma _{145} = \alpha _{1,1}+\alpha _{1,2}$ & $(1,1,0,0,0,0)$  & $[145]$ & $Y_{2,-2}Y_{2,0}$ \\
 $\gamma _{146}=\alpha _{1,1}+\alpha _{1,2}+\alpha _{2,3}$ & $(1,1,0,0,0,1)$  & $[146]$ & $Y_{1,-5}Y_{2,-2}Y_{2,0}$ \\
 $\gamma _{235}=\alpha _{2,2}$ & $(0,0,0,0,1,0)$  & $[235]$ & $Y_{1,-3}$ \\
 $\gamma _{236} = \alpha _{2,2}+\alpha _{2,3}$ & $(0,0,0,0,1,1)$  & $[236]$ & $Y_{1,-5}Y_{1,-3}$ \\
 $\gamma _{245} = \alpha _{1,2}$ & $(0,1,0,0,0,0)$ & $[245]$ & $Y_{2,-2}$\\
 $\gamma _{246} = \alpha _{1,2}+\alpha _{2,3}$ & $(0,1,0,0,0,1)$  & $[246]$ & $Y_{1,-5}Y_{2,-2}$ \\
 $\gamma _{256} = \alpha _{1,2}+\alpha _{1,3}$ & $(0,1,1,0,0,0)$  & $[256]$ & $Y_{2,-4}Y_{2,-2}$ \\
 $\gamma _{346} = \alpha _{2,3}$ & $(0,0,0,0,0,1)$  & $[346]$ & $Y_{1,-5}$ \\
 $\gamma _{356} = \alpha _{1,3}$ & $(0,0,1,0,0,0)$  & $[356]$ & $Y_{2,-4}$ \\
 $\gamma _{124}+\gamma _{356} = \alpha _{1,3}+\alpha _{2,1}$ & $(0,0,1,1,0,0)$ & $\text{[[124],[356]]}$ & $Y_{2,-4}Y_{1,-1}$ \\
 $\gamma _{145}+\gamma _{236} = \alpha _{1,1}+\alpha _{1,2}+\alpha _{2,2}+\alpha _{2,3}$ & $(1,1,0,0,1,1)$  & $\text{[[135],[246]]}$ & $Y_{1,-5}Y_{2,-2}Y_{1,-3}Y_{2,0}$ \\
$\gamma_{126} = \alpha_{2,1}+\alpha_{2,2}+\alpha_{2,3}$ & $(0,0,0,1,1,1)$  & $[126]$ & $Y_{1,-5}Y_{1,-3}Y_{1,-1}$\\ 
$\gamma_{156} = \alpha_{1,1}+\alpha_{1,2}+\alpha_{1,3}$ & $(1,1,1,0,0,0)$  & $[156]$ & $Y_{2,-4}Y_{2,-2}Y_{2,0}$ \\
 \hline 
\end{tabular}
\caption{Correspondence of generalized roots, facets (together with two hyperplanes defining the ambient space of ${\bf N}_{3,6}^{(1)}$), tableaux, and dominant monomials of prime modules in the case of $\Gr(3,6)$. In the table, each list in $[[124],[356]]$ is a column of the tableau.}
\label{table:correspondence of higher roots rays tableaux dominant monomials Gr36}
\end{table} 

Moreover, the two hyperplanes in (\ref{eq:hyperplanes and half surface of N1 Gr36}) of the Newton polytope correspond to the following generalized roots and frozen tableaux:
\begin{align*}
& (0,0,0,1,1,1), \ \alpha_{2,1}+\alpha_{2,2}+\alpha_{2,3}, \ [1,2,6],\\ 
& (1,1,1,0,0,0), \ \alpha_{1,1}+\alpha_{1,2}+\alpha_{1,3}, \ [1,5,6].
\end{align*}
The facets of ${\bf N}_{3,6}^{(1)}$ give all the prime modules (non-frozen) in the category $\mathcal{C}_{\ell}$, $\ell=2$. 

By the same computations, we have that ${\bf N}^{(0)}_{3,6}$ has $14$ facets. The Newton polytope ${\bf N}^{(1)}_{3,6}$ has $16$ facets but is not a simple polytope\footnote{A polytope $P$ of dimension $d$ is simple if every vertex has exactly $d$ incident edges.}. The Newton polytope ${\bf N}^{(d)}_{3,6}$ has $16$ facets for any $d \ge 2$ and it is simple. The Newton polytope ${\bf N}^{(d)}_{3,6}$ is the type $D_4$ associahedron for any $d \ge 2$, see Section 6 of \cite{SW05}.

\subsection{Examples: $\Gr(3,8)$, $\Gr(4,8)$} \label{subsec:examples Gr38 Gr48 rays}

In the case of $\CC[\Gr(3,8)]$, the PK polytope ${\bf N}^{(0)}_{3,8}$ has $\binom{8}{3}-8=48$ facets, ${\bf N}^{(1)}_{3,8}$ has $120$ facets (see also Proposition 6.2 in \cite{BBRS20}), ${\bf N}^{(2)}_{3,8}$ has $128$ facets but it is not a simple polytope; ${\bf N}^{(d)}_{3,8}$ has $128$ facets and ${\bf N}^{(d)}_{3,8}$ is a simple polytope for any $d \ge 3$. This agrees with the fact that the dual canonical basis $\CC[\Gr(3,8)]$ has $128$ prime elements (not including frozen variables). 

In the case of $\CC[\Gr(4,8)]$, ${\bf N}^{(0)}_{4,8}$ has $\binom{8}{4}-8=62$ facets, ${\bf N}^{(1)}_{4,8}$ has 360 facets. Four facets of ${\bf N}^{(1)}_{4,8}$ correspond to prime non-real tableaux $\scalemath{0.6}{\begin{ytableau}
    1&2 \\ 3&4 \\ 5&6 \\ 7&8
\end{ytableau}}$, $\scalemath{0.6}{ \begin{ytableau}
    1&3 \\ 2&5 \\ 4&7 \\ 6&8
\end{ytableau}}$, $\scalemath{0.6}{ \begin{ytableau}
1 & 1 & 2 & 3 \\
2 & 4 & 4 & 5 \\
3 & 6 & 6 & 7 \\
5 & 7 & 8 & 8
\end{ytableau}}$, $\scalemath{0.6}{ \begin{ytableau}
1 & 1 & 2 & 4 \\
2 & 3 & 3 & 6 \\
4 & 5 & 5 & 7 \\
6 & 7 & 8 & 8
\end{ytableau}}$, see also \cite{ALS21, DFGK21, HP20, RSV21}. 

\subsection{The Face Corresponding to a Tableau} \label{subsec: the face corresponding to T}

Take a Newton polytope ${\bf N}_{k,n}^{(d)}$. We now describe a procedure to produce a face of ${\bf N}_{k,n}^{(d)}$ for a given semistandard tableau. 

Let $T \in \SSYT(k, [n])$. By Lemma 3.13 in \cite{CDFL}, there is a unique semistandard tableau $T'$ in $\SSYT(k,[n])$ whose columns ${T'}_1, \ldots, {T'}_m$ are fundamental tableaux and $T \sim T'$. 

For each fundamental tableau $T_{i,j}$ with entries $[j,j+k]\setminus \{i+j\}$, we define $\alpha(T_{i,j}) = \alpha_{i,j}$. Let $\gamma_T = \sum_{j=1}^m \alpha({T'}_j)$ and let $\mathbf{F}_{T} = \{x \in {\bf N}_{k,n}^{(d)} : \gamma_T(y) \ge \gamma_T(x), \forall y \in {\bf N}_{k,n}^{(d)} \}$. 
Then $\mathbf{F}_T$ is the face of ${\bf N}_{k,n}^{(d)}$ such that $\gamma_T$ is minimized on $\mathbf{F}_T$. We say that the tableau $T$ corresponds to a facet if $\mathbf{F}_T$ is a facet of ${\bf N}_{k,n}^{(d)}$. 

To compute $\mathbf{F}_{T}$, we compute the monomials in $p$ which are minimized by $\gamma_T$, where $p$ is the polynomial which defines ${\bf N}_{k,n}^{(d)}$. The face $\mathbf{F}_{T}$ is the face of ${\bf N}_{k,n}^{(d)}$ which is the convex hull of the exponent vectors of these monomials. 

Conjecture \ref{conj:prime modules are rays} (2) is equivalent to the following conjecture.
\begin{conjecture} \label{conj:T is prime then FT is of codim 1}
For $k \le n$ and any nonfrozen prime tableau $T \in \SSYT(k, [n])$, there exists $d \ge 0$ such that the face $\mathbf{F}_{T}$ of ${\bf N}_{k,n}^{(d)}$ has codimension $1$. 
\end{conjecture}

We say that two tableaux $T$, $T'$ (resp., two simple modules $L(M)$, $L(M')$) are compatible if $\ch(T)\ch(T') = \ch(T \cup T')$ (resp., $\chi_q(L(M))\chi_q(L(M')) = \chi_q(L(MM'))$). We give a conjecture about compatibility of two prime tableaux (equivalently, two prime modules).

\begin{conjecture} \label{conjecture:compatibility of prime tableaux}
Let $k \le n$ and $T$, $T'$ be two distinct prime tableaux in $\SSYT(k, [n])$. Then $T$, $T'$ are compatible if and only if there exists $d \ge 0$, such that the faces ${\bf F}_T$, ${\bf F}_{T'}$ corresponding to $T$, $T'$ are facets and the intersection of ${\bf F}_T$, ${\bf F}_{T'}$ is nonempty. 
\end{conjecture}

\begin{example} \label{example:F_T is not a facet}
Consider a tableau $T = \scalemath{0.6}{ \begin{ytableau}
    1 & 1 \\ 2 & 3 \\ 4 & 5
\end{ytableau}}$ and ${\bf N}_{3,6}^{(1)}$. We now check that $\mathbf{F}_T$ is not a facet of ${\bf N}_{3,6}^{(1)}$. The tableau $T'$ whose columns are fundamental tableaux and such that $T' \sim T$ is $\scalemath{0.6}{ \begin{ytableau}
    1 & 1 & 2 \\ 2 & 3 & 3 \\ 4 & 4 & 5
\end{ytableau}}$. We have that $\gamma_T = \alpha_{2,1} + \alpha_{1,1} + \alpha_{2,2}$. We compute the exponents of the monomials in $p$ which take the minimal value when applying $\gamma_T$, where $p$ is the polynomial which defines ${\bf N}_{3,6}^{(1)}$. These exponents define the integer lattice points in $\mathbf{F}_T$. 
The affine span of these points is the face $\mathbf{F}_T$ and it is of codimension 2. Therefore $\mathbf{F}_T$ is not a (codimension 1) facet of ${\bf N}_{3,6}^{(1)}$. Similarly, $\mathbf{F}_T$ is not a facet of ${\bf N}_{3,6}^{(d)}$ for any $d \ge 1$. 

On the other hand, $T$ is non-prime. Indeed, we have that $\ch(T) = p_{134}p_{125}$. This agrees with Conjecture \ref{conj:T is prime then FT is of codim 1}. 
\end{example}

We give an example that the facets of two compatible prime tableaux have a nonempty intersection. 
\begin{example}
In Example \ref{example:F_T is not a facet}, we see that  $\ch(\scalemath{0.6}{ \begin{ytableau}
    1 & 1 \\ 2 & 3 \\ 4 & 5
\end{ytableau}}) = p_{134}p_{125}$. So the two tableaux $T =\scalemath{0.6}{ \begin{ytableau}
    1 \\ 3 \\ 4
\end{ytableau}}$,  $T' = \scalemath{0.6}{ \begin{ytableau}
    1 \\ 2 \\ 5
\end{ytableau}}$ are compatible. Both of the faces ${\bf F}_T$ and ${\bf F}_{T'}$ in ${\bf N}_{3,6}^{(1)}$ have codimension $1$. The intersection of the facets ${\bf F}_T$ and ${\bf F}_{T'}$ is nonempty and has codimension $2$. This example verifies Conjecture \ref{conjecture:compatibility of prime tableaux}.  
\end{example}

\section{Proof of Conjecture \ref{conj:prime modules are rays} for ${\bf N}_{3,9}^{(1)}$ and ${\bf N}_{4,8}^{(1)}$} \label{sec:Coarsest subdivisions and prime modules}
In this section, we prove Conjecture \ref{conj:prime modules are rays} for ${\bf N}_{3,9}^{(1)}$ and ${\bf N}_{4,8}^{(1)}$: the facets of ${\bf N}_{3,9}^{(1)}$ (resp. ${\bf N}_{4,8}^{(1)}$) correspond to prime modules of $U_q(\widehat{\mathfrak{sl}_3})$-modules (resp. $U_q(\widehat{\mathfrak{sl}_4})$-modules).  

\subsection{Facets of ${\bf N}_{3,9}^{(1)}$ correspond to prime modules}

There are $471$ facets of ${\bf N}_{3,9}^{(1)}$ (see also \cite{HRZ20, RSV21}). These facets give $471$ 
tableaux in $\SSYT(3, [9])$. We now verify that these tableaux are prime.

Among these tableaux, there are $75$ one-column tableaux. These correspond to all non-frozen Pl\"{u}cker coordinates in $\CC[\Gr(3,9)]$. 

There are $168$ two-column tableaux in the $471$ tableaux. These tableaux can be obtained from the two prime $2$-column tableaux in $\SSYT(3,[6])$ by replacing $1<2<\cdots <6$ by $a_1<a_2<\cdots <a_6$ ($a_i \in [9]$). 

There are $156$ tableaux with $3$ columns in the $471$ tableaux. Totally there are $228$ prime tableaux with $3$ columns in $\SSYT(3, [9])$ (this can be seen by translating the results about the number of indecomposable modules in Grassmannian cluster category ${\rm CM}(B_{3,9})$ in \cite{BBGL}). The $156$ tableaux are part of them. Up to promotion \cite{Sch63, Sch72, Sch77}, these $156$ tableaux are
\begin{align*}
& \scalemath{0.7}{ \begin{ytableau}
1 & 2 & 5 \\
3 & 4 & 7 \\
5 & 6 & 9
\end{ytableau}, \ \begin{ytableau}
1 & 2 & 5 \\
3 & 4 & 8 \\
5 & 6 & 9
\end{ytableau}, \ \begin{ytableau}
1 & 2 & 6 \\
3 & 4 & 8 \\
5 & 7 & 9
\end{ytableau}, \ \begin{ytableau}
1 & 2 & 6 \\
3 & 5 & 8 \\
6 & 7 & 9
\end{ytableau}, \ \begin{ytableau}
1 & 3 & 6 \\
4 & 5 & 8 \\
6 & 7 & 9
\end{ytableau}, \ \begin{ytableau}
1 & 2 & 6 \\
4 & 5 & 8 \\
6 & 7 & 9
\end{ytableau}, \ \begin{ytableau}
1 & 2 & 4 \\
3 & 4 & 7 \\
5 & 6 & 9
\end{ytableau}, \ \begin{ytableau}
1 & 2 & 4 \\
3 & 4 & 8 \\
5 & 6 & 9
\end{ytableau}, \ \begin{ytableau}
1 & 2 & 6 \\
3 & 4 & 8 \\
6 & 7 & 9
\end{ytableau}, } \\
& \scalemath{0.7}{ \begin{ytableau}
1 & 3 & 3 \\
2 & 5 & 6 \\
4 & 7 & 9
\end{ytableau}, \ \begin{ytableau}
1 & 3 & 4 \\
2 & 5 & 6 \\
4 & 7 & 9
\end{ytableau}, \ \begin{ytableau}
1 & 2 & 5 \\
3 & 4 & 7 \\
6 & 8 & 8
\end{ytableau}, \ \begin{ytableau}
1 & 2 & 5 \\
3 & 5 & 8 \\
6 & 7 & 9
\end{ytableau}, \ \begin{ytableau}
1 & 3 & 5 \\
4 & 5 & 8 \\
6 & 7 & 9
\end{ytableau}, \ \begin{ytableau}
1 & 2 & 5 \\
4 & 5 & 8 \\
6 & 7 & 9
\end{ytableau}, \ \begin{ytableau}
1 & 2 & 4 \\
3 & 4 & 8 \\
5 & 7 & 9
\end{ytableau}, \ \begin{ytableau}
1 & 2 & 5 \\
3 & 4 & 7 \\
6 & 7 & 9
\end{ytableau}, \ \begin{ytableau}
1 & 2 & 5 \\
3 & 4 & 8 \\
6 & 7 & 9
\end{ytableau}. }
\end{align*}
The orbit (under promotion) of the last tableau has $3$ tableaux. The sizes of the other orbits are $9$. The last tableau is prime non-real. 

There are $69$ tableaux which have $4$ columns in the $471$ tableaux. Up to promotions, these tableaux are
\begin{align*}
\scalemath{0.62}{ 
\begin{ytableau}
1 & 2 & 4 & 5 \\
3 & 4 & 7 & 7 \\
5 & 6 & 8 & 9
\end{ytableau}, \ \begin{ytableau}
1 & 2 & 4 & 5 \\
3 & 4 & 7 & 8 \\
5 & 6 & 8 & 9
\end{ytableau}, \ \begin{ytableau}
1 & 2 & 2 & 5 \\
3 & 4 & 5 & 8 \\
6 & 6 & 7 & 9
\end{ytableau}, \ \begin{ytableau}
1 & 1 & 3 & 3 \\
2 & 5 & 6 & 6 \\
4 & 7 & 8 & 9
\end{ytableau}, \ \begin{ytableau}
1 & 1 & 3 & 4 \\
2 & 5 & 6 & 6 \\
4 & 7 & 8 & 9
\end{ytableau}, \ \begin{ytableau}
1 & 1 & 4 & 4 \\
2 & 3 & 7 & 7 \\
5 & 6 & 8 & 9
\end{ytableau}, \ \begin{ytableau}
1 & 1 & 3 & 4 \\
2 & 5 & 6 & 7 \\
4 & 7 & 8 & 9
\end{ytableau}, \ \begin{ytableau}
1 & 1 & 3 & 3 \\
2 & 4 & 5 & 6 \\
4 & 7 & 8 & 9
\end{ytableau}, \ \begin{ytableau}
1 & 2 & 4 & 5 \\
3 & 4 & 7 & 8 \\
6 & 7 & 8 & 9
\end{ytableau}.
}
\end{align*}
The orbits of the $6$th and $7$th tableaux have $3$ tableaux, respectively. The sizes of the other orbits are $9$. In \cite{CDHHHL}, cluster variables in terms of tableaux up to $10$ columns in $\CC[\Gr(3,9)]$ are computed extensively. We check directly that these $69$ tableaux are in the list of cluster variables obtained in \cite{CDHHHL}. Therefore these $69$ tableaux are prime. 

There are $3$ tableaux which have $5$ columns in these $471$ tableaux. These three tableaux are promotions of the following tableau
\begin{align*}
\scalemath{0.7}{
\begin{ytableau}
1 & 1 & 2 & 4 & 5 \\
2 & 3 & 4 & 7 & 8 \\
5 & 6 & 7 & 8 & 9
\end{ytableau}. }
\end{align*}
These $3$ tableaux are in the list of cluster variables obtained in \cite{CDHHHL}. Therefore they are prime. 

\subsection{Facets of ${\bf N}_{4,8}^{(1)}$ correspond to prime modules}

There are $360$ facets of ${\bf N}_{4,8}^{(1)}$ (see also \cite{HRZ20, RSV21}). These facets give $360$ 
tableaux in $\SSYT(4, [8])$. We now verify that these tableaux are prime.

Among these tableaux, there are $62$ one-column tableaux. These correspond to all non-frozen Pl\"{u}cker coordinates in $\CC[\Gr(4,8)]$. One-column tableaux are all prime. Therefore these $62$ tableaux are prime.

There are $122$ tableaux with $2$ columns in the $360$ tableaux. The two tableaux $\scalemath{0.6}{\begin{ytableau}
1&2\\3&4\\5&6\\7&8
\end{ytableau}}$ and $\scalemath{0.6}{\begin{ytableau}
1&3\\2&5\\4&7\\6&8
\end{ytableau}}$ are prime non-real tableaux, see Section 8 in \cite{CDFL}. The rest $120$ tableaux are cluster variables \cite{CDHHHL}. Therefore they are all prime. Up to promotion \cite{Sch63, Sch72, Sch77}, these $120$ tableaux are
\begin{align*}
\scalemath{0.66}{
\begin{ytableau}
1 & 3 \\
2 & 5 \\
4 & 6 \\
7 & 8
\end{ytableau},\  \begin{ytableau}
1 & 3 \\
2 & 5 \\
4 & 6 \\
8 & 8
\end{ytableau},\  \begin{ytableau}
1 & 3 \\
2 & 5 \\
4 & 7 \\
8 & 8
\end{ytableau},\  \begin{ytableau}
1 & 3 \\
2 & 6 \\
4 & 7 \\
8 & 8
\end{ytableau},\  \begin{ytableau}
1 & 3 \\
2 & 6 \\
5 & 7 \\
8 & 8
\end{ytableau},\  \begin{ytableau}
1 & 4 \\
2 & 6 \\
5 & 7 \\
8 & 8
\end{ytableau},\  \begin{ytableau}
1 & 4 \\
3 & 6 \\
5 & 7 \\
8 & 8
\end{ytableau},\  \begin{ytableau}
2 & 4 \\
3 & 6 \\
5 & 7 \\
8 & 8
\end{ytableau},\  \begin{ytableau}
1 & 2 \\
3 & 4 \\
5 & 6 \\
8 & 8
\end{ytableau},\  \begin{ytableau}
1 & 2 \\
3 & 4 \\
5 & 7 \\
8 & 8
\end{ytableau},\  \begin{ytableau}
1 & 2 \\
3 & 4 \\
6 & 7 \\
8 & 8
\end{ytableau},\  \begin{ytableau}
1 & 2 \\
3 & 5 \\
6 & 7 \\
8 & 8
\end{ytableau},\  \begin{ytableau}
1 & 2 \\
4 & 5 \\
6 & 7 \\
8 & 8
\end{ytableau},\  \begin{ytableau}
1 & 3 \\
4 & 5 \\
6 & 7 \\
8 & 8
\end{ytableau},\  \begin{ytableau}
2 & 3 \\
4 & 5 \\
6 & 7 \\
8 & 8
\end{ytableau}.}
\end{align*}

There are $132$ tableaux with $3$ columns in the $360$ tableaux. These tableaux are cluster variables \cite{CDHHHL}. Therefore they are all prime. Up to promotion \cite{Sch63, Sch72, Sch77}, these $132$ tableaux are
\begin{align*}
& 
\scalemath{0.66}{\begin{ytableau}
1 & 1 & 3 \\
2 & 4 & 5 \\
3 & 6 & 7 \\
5 & 7 & 8
\end{ytableau},\  \begin{ytableau}
1 & 1 & 3 \\
2 & 4 & 5 \\
4 & 6 & 7 \\
5 & 8 & 8
\end{ytableau},\  \begin{ytableau}
1 & 1 & 2 \\
2 & 4 & 5 \\
3 & 6 & 6 \\
5 & 7 & 8
\end{ytableau},\  \begin{ytableau}
1 & 1 & 3 \\
2 & 4 & 5 \\
4 & 6 & 7 \\
5 & 7 & 8
\end{ytableau},\  \begin{ytableau}
1 & 2 & 3 \\
4 & 4 & 5 \\
5 & 6 & 7 \\
7 & 8 & 8
\end{ytableau},\  \begin{ytableau}
1 & 1 & 4 \\
2 & 3 & 6 \\
4 & 5 & 7 \\
6 & 8 & 8
\end{ytableau},\  \begin{ytableau}
1 & 1 & 3 \\
2 & 4 & 5 \\
4 & 6 & 6 \\
5 & 7 & 8
\end{ytableau},\  \begin{ytableau}
1 & 1 & 2 \\
2 & 3 & 4 \\
3 & 6 & 7 \\
5 & 7 & 8
\end{ytableau},\  \begin{ytableau}
1 & 1 & 3 \\
2 & 2 & 5 \\
3 & 4 & 7 \\
5 & 6 & 8
\end{ytableau}, } \\
& \scalemath{0.66}{  \begin{ytableau}
1 & 3 & 4 \\
2 & 5 & 5 \\
4 & 7 & 7 \\
6 & 8 & 8
\end{ytableau},\  \begin{ytableau}
1 & 1 & 3 \\
2 & 2 & 5 \\
4 & 4 & 7 \\
5 & 6 & 8
\end{ytableau},\  \begin{ytableau}
1 & 3 & 3 \\
2 & 5 & 6 \\
4 & 7 & 7 \\
6 & 8 & 8
\end{ytableau},\  \begin{ytableau}
1 & 2 & 3 \\
2 & 5 & 5 \\
4 & 6 & 7 \\
6 & 7 & 8
\end{ytableau},\  \begin{ytableau}
1 & 1 & 3 \\
2 & 3 & 5 \\
4 & 5 & 7 \\
6 & 7 & 8
\end{ytableau},\  \begin{ytableau}
1 & 1 & 4 \\
2 & 3 & 6 \\
5 & 5 & 7 \\
6 & 7 & 8
\end{ytableau},\  \begin{ytableau}
1 & 1 & 4 \\
2 & 3 & 6 \\
4 & 5 & 7 \\
6 & 7 & 8
\end{ytableau},\  \begin{ytableau}
1 & 1 & 3 \\
2 & 3 & 6 \\
4 & 5 & 7 \\
6 & 7 & 8
\end{ytableau},\  \begin{ytableau}
1 & 1 & 3 \\
2 & 2 & 4 \\
3 & 4 & 7 \\
5 & 6 & 8
\end{ytableau},\  \begin{ytableau}
1 & 3 & 3 \\
2 & 5 & 5 \\
4 & 7 & 7 \\
6 & 8 & 8
\end{ytableau}.
}
\end{align*}
There are $34$ tableaux with $2$ columns in the $360$ tableaux. The two tableaux $\scalemath{0.6}{\begin{ytableau}
1 & 1 & 2 & 3 \\
2 & 4 & 4 & 5 \\
3 & 6 & 6 & 7 \\
5 & 7 & 8 & 8
\end{ytableau}}$ and $\scalemath{0.6}{\begin{ytableau}
1 & 1 & 2 & 4 \\
2 & 3 & 3 & 6 \\
4 & 5 & 5 & 7 \\
6 & 7 & 8 & 8
\end{ytableau}}$ are prime non-real tableaux, see Section IV in \cite{ALS21}. The rest $32$ tableaux are cluster variables \cite{CDHHHL}. Therefore they are all prime. Up to promotion \cite{Sch63, Sch72, Sch77}, these $32$ tableaux are
\begin{align*}
\scalemath{0.66}{\begin{ytableau}
1 & 2 & 2 & 3 \\
3 & 4 & 4 & 5 \\
5 & 6 & 6 & 7 \\
7 & 8 & 8 & 8
\end{ytableau},\  \begin{ytableau}
1 & 1 & 3 & 3 \\
2 & 2 & 4 & 5 \\
3 & 4 & 6 & 7 \\
5 & 6 & 7 & 8
\end{ytableau},\  \begin{ytableau}
1 & 1 & 2 & 3 \\
2 & 4 & 4 & 5 \\
3 & 6 & 7 & 7 \\
5 & 7 & 8 & 8
\end{ytableau},\  \begin{ytableau}
1 & 1 & 3 & 3 \\
2 & 2 & 5 & 5 \\
4 & 4 & 6 & 7 \\
5 & 6 & 7 & 8
\end{ytableau}.}
\end{align*}
There are $10$ tableaux with $5$ columns in the $360$ tableaux. These tableaux are cluster variables \cite{CDHHHL}. Therefore they are all prime. Up to promotion \cite{Sch63, Sch72, Sch77}, these $10$ tableaux are
\begin{align*}
\scalemath{0.66}{\begin{ytableau}
1 & 1 & 1 & 2 & 3 \\
2 & 2 & 4 & 4 & 5 \\
3 & 3 & 6 & 7 & 7 \\
5 & 6 & 7 & 8 & 8
\end{ytableau},\  \begin{ytableau}
1 & 1 & 1 & 3 & 3 \\
2 & 2 & 4 & 5 & 5 \\
3 & 4 & 6 & 7 & 7 \\
5 & 6 & 7 & 8 & 8
\end{ytableau}.}
\end{align*}

\section{Newton Polytopes and Tropical Fans for Quantum Affine Algebras} \label{sec:Newton polytopes for quantum affine algebras}
Readers interested in physical applications and stringy integrals of Grassmann type may skip to Section \ref{sec:physics}.

In this section, for any simple Lie algebra $\mathfrak{g}$ over $\CC$, we define Newton polytopes associated to the quantum affine algebra $U_q(\widehat{\mathfrak{g}})$ using truncated $q$-characters of simple $U_q(\widehat{\mathfrak{g}})$-modules.  

\subsection{Newton Polytopes for Quantum Affine Algebras} \label{subsec:Newton polytopes for quantum affine algebras}

Let $\mathcal{M}$ be the set of all equivalence classes of Kirillov-Reshetikhin modules of $U_q(\widehat{\mathfrak{g}})$ in $\mathcal{C}_{\ell}$. For simplicity, we also write an equivalence class $[L(M)]$ as $L(M)$. 

\begin{definition} \label{def: newton polytope for quantum affine algebra version 1}
Let  $\mathcal{M}^{(0)}=\mathcal{M}$. We define recursively
\begin{align*}
{\bf N}^{(d)}_{\mathfrak{g}, \ell} = \text{Newt}\left(\prod_{L(M) \in \mathcal{M}^{(d)}} \tchi(L(M))/M \right),
\end{align*}
where $\mathcal{M}^{(d+1)}$ ($d \ge 0$) is the collection of equivalence classes of simple $U_q(\widehat{\mathfrak{g}})$-modules which correspond to facets of ${\bf N}^{(d)}_{\mathfrak{g}, \ell}$. 
\end{definition}

We will explain how to compute these Newton polytopes in the following subsections and give a construction of highest $l$-weights of simple $U_q(\widehat{\mathfrak{g}})$-modules which correspond to facets of ${\bf N}^{(d)}_{\mathfrak{g}, \ell}$ in Section \ref{subsec:from facets to prime modules general quantum affine algebra}. 



\begin{remark}
In type $A$, the definition of the Newton polytopes ${\bf N}^{(d)}_{\mathfrak{sl}_k, \ell}$ is slightly different from the definition of the Newton polytopes ${\bf N}^{(d)}_{k,n}$ ($n=k+\ell+1$) in Section \ref{sec:recursive definition of Newton polytopes} for Grassmannian cluster algebras. Here ${\bf N}^{(0)}_{\mathfrak{sl}_k, \ell}$ is defined using all Kirillov-Reshetikhin modules of $U_q(\widehat{\mathfrak{sl}_k})$. Finite dimensional simple $U_q(\widehat{\mathfrak{sl}_k})$-modules in $\mathcal{C}_{\ell}$, $n = k+\ell+1$, correspond to tableaux in $\SSYT(k, [n], \sim)$ \cite{CDFL}. In Section \ref{sec:recursive definition of Newton polytopes}, ${\bf N}^{(0)}_{k,n}$ is defined using all cyclic shifts of the one-column tableau with entries $1,2,\ldots, k-1, k+1$. These one-column tableaux correspond to a set of minimal affinizations of $U_q(\widehat{\mathfrak{sl}_k})$ \cite{Cha95, CDFL, CP96a}. 
\end{remark}

We now define another version of Newton polytopes for quantum affine algebras non-recursively. For $d \in \ZZ_{\ge 1}$, denote by $\mathcal{P}_{\ell}^{+, d}$ the set of all dominant monomials in $\mathcal{P}_{\ell}^+$ with degrees less or equal to $d$.
\begin{definition} \label{def:another version of Newton polytopes for quantum affine algebras}
For a simple Lie algebra $\mathfrak{g}$ over $\CC$, $\ell \ge 1$, and $d \in \ZZ_{\ge 1}$, we define 
\begin{align*}
{\bf N'}_{\mathfrak{g},\ell}^{(d)} = {\rm Newt}  \left( \prod_{M \in \mathcal{P}_{\ell}^{+,d}} \tchi(L(M))/M \right).
\end{align*}
\end{definition}

\subsection{Truncated $q$-characters and F-polynomials} \label{subsec:q character and F polynomials}

In \cite{HL16}, for every $\ell \ge 0$, Hernandez and Leclerc constructed an algebra $A_{\ell}$ defined by a quiver with potential using their initial seed for the cluster algebra $K_0(\mathcal{C}_{\ell})$. They introduced certain distinguished $A_{\ell}$-modules $K(m)$ for every simple $U_q(\widehat{\mathfrak{g}})$-module $L(m)$. 

Recall that \cite{Lec} a simple $U_q(\widehat{\mathfrak{g}})$-module $L(m)$ is real if $L(m) \otimes L(m)$ is simple. Hernandez and Leclerc (Conjecture 5.3 in \cite{HL16}) conjectured that for every real simple $U_q(\widehat{\mathfrak{g}})$-module $L(m)$, the truncated $q$-character $\tchi(L(m))$ of $L(m)$ is equal to $mF_{K(m)}$, where $F_{K(m)}$ is the F-polynomial of $K(m)$, \cite{DWZ08, DWZ10}. 

By Theorem 4.1 in \cite{FM00} (Conjecture 1 in \cite{FR98}), we have that for every simple $U_q(\widehat{\mathfrak{g}})$-module $L(m)$ (not necessarily real), $\tchi(L(m))$ is equal to $m$ times a polynomial in $A_{i,a}^{-1}$ ($i \in I$, $a \in C^*$) with constant term $1$, where $A_{i,a}^{-1}$ is defined in (\ref{eq:Aia}). Denote by $v_{i,s} = A_{i,s}^{-1}$, $i \in I$, $s \in \ZZ$, and we fixed $a \in \CC^*$ and write $A_{i,s}=A_{i,aq^s}$. Given a simple module $L(m)$, after factoring out $m$ and replacing $A_{i,s}^{-1}$ by $v_{i,s}$ in $\tchi(L(m))$, we obtain a polynomial in $v_{i,s}$. In Section \ref{subsec:from facets to prime modules general quantum affine algebra}, we will use the polynomials in $v_{i,s}$ corresponding to simple modules to compute Newton polytopes in Definitions \ref{def: newton polytope for quantum affine algebra version 1} and \ref{def:another version of Newton polytopes for quantum affine algebras}. 

\subsection{$g$-vectors and highest $l$-weights} \label{subsec:g vectors and highest l weights}

By results in \cite[Section 5.2.2]{HL16}, see also \cite[Section 2.6]{DR20}, \cite[Section 7]{CDFL}, we have that for any simple $U_q(\widehat{\mathfrak{g}})$-module 
$L(M)$, its $g$-vector is obtained as follows. The dominant monomial $M$ can be written as $M = \prod_{i,s} Y_{i,s}^{a_{i,s}}$ for some non-negative integers $a_{i,s}$, where the product runs over all fundamental modules $L(Y_{i,s})$ in $\mathcal{C}_{\ell}$. On the other hand, $M$ can also be written as $M = \prod_{j} M_j^{g_j}$ for some integers $g_j$, where the product runs over all initial cluster variables and frozen variables $L(M_j)$ in $\mathcal{C}_{\ell}$. The $g_j$'s are the unique solution of $\prod_{i,s} Y_{i,s}^{a_{i,s}} = \prod_{j} M_j^{g_j}$. With a chosen order, $g_j$'s form the $g$-vector of $L(M)$. A simple module $L(M)$ is determined by its $g$-vector uniquely. 

\begin{remark}
In this paper, every element in the dual canonical basis of $K_0(\mathcal{C}_{\ell})$ and $\CC[\Gr(k,n)]$ has a $g$-vector in the above sense even if it is not a cluster monomial. 
\end{remark}

Given a simple module $L(M)$, let $g_M$ be the vector obtained from the $g$-vector of $L(M)$ by forgetting the entries corresponding to the frozens. We also call $g_M$ the $g$-vector $L(M)$. 

Fix an order of the initial cluster variables (not including frozens), say $z_1, \ldots, z_m$. Given any vector $g 
\in \ZZ^{m}$, the monomial $M'=z_1^{g_1} \cdots z_m^{g_m}$ can be written as $M' = AB^{-1}$ for two dominant monomials $A, B$. The monomial $B$ is of the form $B = \prod_{i \in I} Y_{i, s_1}^{u_{i,1}} \cdots Y_{i, s_{r_i}}^{u_{i,r_i}}$ for some positive integers $r_i$, $u_{i,j}$. Let $B' = \prod_{i \in I} (Y_{i, \xi(i)} \cdots Y_{i, \xi(i)+2 \ell})^{ \max( u_{i, j}: j = 1, \ldots, r_i ) }$. Then $M' B'$ is a dominant monomial and $M'B'$ cannot be written as a product of a dominant monomial and a frozen variable. We denote $M_g = M'B'$ and say that $L(M_g)$ corresponds to the vector $g$. 

\subsection{From facets to prime modules} \label{subsec:from facets to prime modules general quantum affine algebra}

Fix an order of the initial cluster variables. Every $v_{i,s}$ is the X-variable at a vertex of the initial quiver of $K_0(\mathcal{C}_{\ell})$ (see Lemma 4.15 in \cite{HL16}). We order the variables $v_{i,s}$ according to the order of initial cluster variables. Given a facet ${\bf F}$ of ${\bf N}^{(d)}_{\mathfrak{g}, \ell}$ with the outward normal vector $v_{\bf F}$ of ${\bf F}$, let $L(M_{\bf F})$ be the simple $U_q(\widehat{\mathfrak{g}})$-module corresponding to $v_{\bf F}$, see Section \ref{subsec:g vectors and highest l weights}.

Recall that two simple $U_q(\widehat{\mathfrak{g}})$-modules $L(M)$, $L(M')$ are called compatible if the identity $\chi_q(L(M))\chi_q(L(M')) = \chi_q(L(MM'))$ holds. 
\begin{conjecture} \label{conj:newton polytope and prime modules conjecture quantum affine algebra case}

Let $\mathfrak{g}$ be a simple Lie algebra over $\CC$ and $\ell \ge 1$. We have the following.
\begin{enumerate}

\item For any $d \ge 0$, every facet of ${\bf N}^{(d)}_{\mathfrak{g}, \ell}$ corresponds to a prime $U_q(\widehat{\mathfrak{g}})$-module in $\mathcal{C}_{\ell}$.

\item For every prime $U_q(\widehat{\mathfrak{g}})$-module (nonfrozen) $L(M)$ in $\mathcal{C}_{\ell}$, there exists $d\ge 0$ such that $L(M)$ corresponds to a facet of the polytope ${\bf N}^{(d)}_{\mathfrak{g}, \ell}$. 

\item For any two distinct prime modules in $\mathcal{C}_{\ell}$, they are compatible if and only if there is some $d \ge 0$ such that there are two facets of ${\bf N}_{\mathfrak{g}, \ell}^{(d)}$ corresponding to them and the intersection of these two facets is nonempty.
\end{enumerate}
\end{conjecture}

In the following subsections, we compute some examples of ${\bf N}^{(d)}_{\mathfrak{g}, \ell}$ and ${\bf N'}^{(d)}_{\mathfrak{g}, \ell}$.  

\subsection{Example: $\mathfrak{g}$ is of type $A_1$ and $\ell=2$}

Consider the case of type $A_1$ and $\ell=2$. The category $\mathcal{C}_{\ell}$ has $5$ prime modules (not including frozen variables). We choose a height function $\xi(1)=-1$, see Section \ref{subsec:HL category and cluster algebra}.

The truncated $q$-characters of Kirillov-Reshetikhin modules (not including frozen variables) in $\mathcal{C}_{\ell}$ are
\begin{align*}
& \tchi(L(Y_{1,-1})) = Y_{1,-1}, \quad \tchi(L(Y_{1,-3})) = Y_{1,-3}+Y_{1, -1}^{-1} = Y_{1,-3}(1+v_{1,-2}), \\
& \tchi(L(Y_{1,-5})) = Y_{1,-5}+Y_{1, -3}^{-1} = Y_{1,-3}Y_{1,-5}(1+v_{1,-4}), \quad \tchi(L(Y_{1,-1}Y_{1,-3})) = Y_{1,-1}Y_{1,-3}, \\
& \tchi(L(Y_{1,-3}Y_{1,-5})) = Y_{1,-3} Y_{1,-5} + \frac{Y_{1,-5}}{Y_{1,-1}} + \frac{1}{Y_{1,-1} Y_{1,-3}} =  Y_{1,-3}Y_{1,-5}(1+v_{1,-2}+v_{1,-4}v_{1,-2}). 
\end{align*}

We take the order of the initial cluster variables as $L(Y_{1,-1})$, $L(Y_{1,-3}Y_{1,-1})$. The corresponding order of variables $v_{i,s}$ is $v_{1,-2}$, $v_{1,-4}$. The Newton polytope ${\bf N}_{\mathfrak{sl}_2, 2}^{(0)}$ is given by the following half-spaces:
\begin{align} 
\begin{split}
(-1, 0) \cdot x + 2 \ge 0, \ (1, -1) \cdot x + 1 \ge 0, \  (1, 0) \cdot x + 0 \ge 0, \ (0, 1) \cdot x + 0 \ge 0, \ (0, -1) \cdot x + 2 \ge 0. 
\end{split}
\end{align}
The outward normal vectors  
\begin{align*}
(1, 0), \ (-1, 1), \ (-1, 0), \ (0, -1), \ (0, 1)
\end{align*}
of these facets are exactly the $g$-vectors of prime modules in $\mathcal{C}_{2}^{\mathfrak{sl}_2}$. These facets correspond to the following prime modules respectively:
\begin{align*}
L(Y_{1,-1}), \ L(Y_{1,-3}), \ L(Y_{1,-3}Y_{1,-5}), \ L(Y_{1,-5}), \ L(Y_{1,-1}Y_{1,-3}). 
\end{align*}

The Newton polytope ${\bf N}_{\mathfrak{sl}_2, 2}^{(d)}$ also has $5$ facets for any $d \ge 1$. 

\subsection{Example: $\mathfrak{g}$ is of type $A_2$, $\ell=2$}

In the case of type $A_2$, $\ell=2$, we choose the height function $\xi(1)=-1$, $\xi(2)=0$. There are $16$ prime modules (not including the two frozens) in the category $\mathcal{C}_{\ell}$. We have the following truncated $q$-characters of Kirillov-Reshetikhin modules (not including initial cluster variables and frozen variables) in $\mathcal{C}_{\ell}$:
\begin{align*}
& \tchi(L(Y_{1,-3})) = \frac{1}{Y_{2,0}} + \frac{Y_{2,-2}}{Y_{1,-1}} + Y_{1,-3} = Y_{1,-3}( 1 + v_{1,-2} + v_{1,-2} v_{2,-1} ), 
\end{align*}
\begin{align*}
& \tchi(L(Y_{1,-5})) = \frac{1}{Y_{2,-2}} + \frac{Y_{2,-4}}{Y_{1,-3}} + Y_{1,-5} = Y_{1,-5}(1 + v_{1,-4} + v_{1,-4}v_{2,-3}),   
\end{align*}
\begin{align*}
& \tchi(L(Y_{2,-2})) = \frac{Y_{1,-1}}{Y_{2,0}} + Y_{2,-2} = Y_{2,-2}(1+v_{2,-1}), 
\end{align*}
\begin{align*}
& \tchi(L(Y_{2,-4})) = \frac{1}{Y_{1,-1}} + \frac{Y_{1,-3}}{Y_{2,-2}} + Y_{2,-4} = Y_{2,-4}(1 + v_{2,-3} + v_{2,-3}v_{1,-2}), 
\end{align*}
\begin{align*}
& \tchi(L(Y_{1,-3} Y_{1,-5})) 
= Y_{1,-3} Y_{1,-5} + \frac{Y_{1,-5}}{Y_{2,0}} + \frac{1}{Y_{2,0} Y_{2,-2}} + \frac{Y_{2,-4}}{Y_{2,0} Y_{1,-3}} + \frac{Y_{2,-2} Y_{1,-5}}{Y_{1,-1}} + \frac{Y_{2,-2} Y_{2,-4}}{Y_{1,-1} Y_{1,-3}} \\
& =  Y_{1,-3} Y_{1,-5}( 1+v_{1,-2}+v_{1,-2} v_{1,-4}+v_{1,-2} v_{2,-1}+v_{1,-2} v_{2,-1} v_{1,-4}+v_{1,-2} v_{2,-1} v_{1,-4} v_{2,-3}), 
\end{align*}
\begin{align*}
&  \tchi(L(Y_{2,-2}Y_{2,-4})) = Y_{2,-2} Y_{2,-4} + \frac{Y_{1,-1} Y_{2,-4}}{Y_{2,0}} + \frac{Y_{1,-1} Y_{1,-3}}{Y_{2,0} Y_{2,-2}}  = Y_{2,-2} Y_{2,-4}( 1 + v_{2,-1} v_{2,-3} + v_{2,-1} ), 
\end{align*}

We take the order of the initial cluster variables as $L(Y_{1,-1})$, $L(Y_{1,-3}Y_{1,-1})$, $L(Y_{2,0})$, $L(Y_{2,-2}Y_{2,0})$. The corresponding order of the variables $v_{i,s}$ is $v_{1,-2}$, $v_{1,-4}$, $v_{2,-1}$, $v_{2,-3}$. The Newton polytope ${\bf N}_{\mathfrak{sl}_3, 2}^{(0)}$ is given by the following half-spaces:
\begin{align*}
& (-1, 0, 0, 0) \cdot x + 3 \ge 0, \  (0, -1, 0, 0) \cdot x + 2 \ge 0, \   (0, 0, -1, 0) \cdot x + 4 \ge 0,   \\
& (0, 1, 0, -1) \cdot x + 2 \ge 0, \   (0, 0, 1, -1) \cdot x + 2 \ge 0, \   (0, 0, 1, 0) \cdot x + 0 \ge 0,  \\
&  (0, 0, 0, 1) \cdot x + 0 \ge 0, \   (1, -1, 0, 0) \cdot x + 1 \ge 0, \   (1, 0, 0, 0) \cdot x + 0 \ge 0,  \\
& (1, 0, -1, 0) \cdot x + 2 \ge 0, \   (0, 1, 0, 0) \cdot x + 0 \ge 0, \   (-1, 0, 0, 1) \cdot x + 2 \ge 0, \   (0, 1, 1, -1) \cdot x + 1 \ge 0.
\end{align*}
The outward normal vectors of these facets correspond to the following prime modules respectively:  
\begin{align*}
& L( Y_{1,-1} ), \ L(Y_{1,-1}Y_{1,-3}), \ L(Y_{2,0}), \ L(Y_{1,-5}Y_{2,-2}Y_{2,0}), \ L(Y_{2,-2}), \ L(Y_{2,-4}Y_{2,-2}), \ L(Y_{2,-4}), \\ 
& L(Y_{1,-3}), \ L(Y_{1,-5}Y_{1,-3}), \ L(Y_{1,-5}Y_{1,-3}Y_{2,0}), \ L(Y_{1,-5}), \ L(Y_{2,-4}Y_{1,-1}), \ L(Y_{1,-5}Y_{2,-2}). 
\end{align*}

We have the following truncated $q$-characters:
\begin{align*}
& \tchi(L(Y_{2,-2}Y_{1,-5})) = Y_{2,-2} Y_{1,-5} + \frac{Y_{1,-1}}{Y_{2,0} Y_{2,-2}} + \frac{Y_{1,-1} Y_{1,-5}}{Y_{2,0}} + \frac{Y_{2,-2} Y_{2,-4}}{Y_{1,-3}} + \frac{Y_{1,-1} Y_{2,-4}}{Y_{2,0} Y_{1,-3}} \\
& = Y_{2,-2} Y_{1,-5}( v_{1, -4} v_{2, -3} v_{2, -1} + v_{1, -4} v_{2, -1} + v_{1, -4} + v_{2, -1} + 1 ). 
\end{align*}

\begin{align*}
& \tchi(L(Y_{1,-1} Y_{2,-4})) = Y_{1,-1} Y_{2,-4} + \frac{Y_{1,-1} Y_{1,-3}}{Y_{2,-2}} = Y_{1,-1} Y_{2,-4}(1 + v_{2,-3}),
\end{align*}

\begin{align*}
& \tchi(L(Y_{2,0} Y_{2,-2} Y_{1,-5})) = Y_{2,0} Y_{2,-2} Y_{1,-5} + \frac{Y_{2,0} Y_{2,-2} Y_{2,-4}}{Y_{1,-3}} = Y_{2,0} Y_{2,-2} Y_{1,-5} (v_{1,-4} + 1), 
\end{align*}

\begin{align*}
& \tchi(L(Y_{2,0} Y_{1,-3} Y_{1,-5})) = Y_{2,0} Y_{1,-3} Y_{1,-5} + \frac{Y_{2,0} Y_{2,-2} Y_{1,-5}}{Y_{1,-1}} + \frac{Y_{2,0} Y_{2,-2} Y_{2,-4}}{Y_{1,-1} Y_{1,-3}} \\
& = Y_{2,0} Y_{1,-3} Y_{1,-5}(v_{1,-2} v_{1,-4} + v_{1,-2} + 1 ).
\end{align*}
The Newton polytope ${\bf N}_{\mathfrak{sl}_3, 2}^{(1)}$ is given by the following half-spaces:
\begin{align*}
& (-1, 0, 0, 0) \cdot x + 4 \ge 0,
(0, -1, 0, 0) \cdot x + 5 \ge 0,
(0, 0, -1, 0) \cdot x + 5 \ge 0,
(0, 0, 0, -1) \cdot x + 6 \ge 0, \\
& (0, 0, 1, 0) \cdot x + 0 \ge 0,
(0, 0, 1, -1) \cdot x + 3 \ge 0,
(0, 1, 0, 0) \cdot x + 0 \ge 0,
(1, -1, 0, 0) \cdot x + 3 \ge 0,\\
& 
(1, 0, -1, 0) \cdot x + 3 \ge 0,
(1, 0, 0, 0) \cdot x + 0 \ge 0,
(1, 0, 0, -1) \cdot x + 5 \ge 0,
(0, 0, 0, 1) \cdot x + 0 \ge 0,\\
& 
(0, 1, 0, -1) \cdot x + 3 \ge 0,
(0, 1, 1, -1) \cdot x + 2 \ge 0,
(1, -1, -1, 0) \cdot x + 7 \ge 0,
(-1, 0, 0, 1) \cdot x + 3 \ge 0.
\end{align*}
The outward normal vectors of these facets are exactly the $g$-vectors of prime modules (not including frozens) in $\mathcal{C}^{\mathfrak{sl}_3}_{2}$. 

We have the following truncated $q$-character:
\begin{align*}
& \tchi(L(Y_{1,-3}Y_{2,0})) = Y_{2,0} Y_{1,-3} + \frac{Y_{2,0} Y_{2,-2}}{Y_{1,-1}} = Y_{2,0} Y_{1,-3}(v_{1,-2} + 1),
\end{align*} 
\begin{align*}
& \scalemath{0.86}{ \tchi(L(Y_{2,0} Y_{1,-3} Y_{2,-2} Y_{1,-5})) = Y_{2,-2} Y_{1,-5} + \frac{Y_{2,-2} Y_{2,-4}}{Y_{1,-3}} + Y_{2,0} Y_{1,-3} Y_{2,-2} Y_{1,-5} + \frac{Y_{2,0} {Y_{2,-2}}^2 Y_{1,-5}}{Y_{1,-1}} + \frac{Y_{2,0} {Y_{2,-2}}^2 Y_{2,-4}}{Y_{1,-1} Y_{1,-3}} } \\
& = Y_{2,0} Y_{1,-3} Y_{2,-2} Y_{1,-5}( v_{1,-2} v_{2,-1} + v_{1,-2} v_{1,-4} + v_{1,-2} + v_{1,-2} v_{2,-1} v_{1,-4} + 1 ).
\end{align*}
The Newton polytope ${\bf N}_{\mathfrak{sl}_3, 2}^{(d)}$ has $16$ facets for any $d \ge 1$. On the other hand, there are $16$ prime modules (not including the two frozens) in $\mathcal{C}^{\mathfrak{sl}_3}_{2}$. Therefore there is a one to one correspondence between facets of ${\bf N}_{k,n}^{(d)}$ ($d \ge 2$) and prime modules in $\mathcal{C}^{\mathfrak{sl}_3}_{2}$.

\subsection{Example: $\mathfrak{g}$ is of type $B_n$ and $\ell=1$}

Consider the case of type $B_2$ and $\ell=1$. We choose a height function as shown in Figure \ref{fig:initial cluster B2 and l is 1}, see \cite{HL10, HL16} for the definition of the cluster algebra associated to $\mathcal{C}_{\ell}$.  

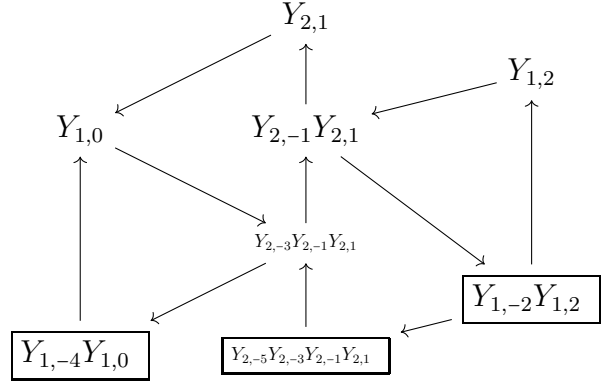
\begin{figure} 
\begin{tikzpicture}[scale=1.5]
		\node  (1) at (0, 0) {$Y_{1,0}$ };
		\node  (2) at (0, -2) {\fbox{$Y_{1,-4}Y_{1,0}$ }}; 
		\node  (3) at (2,1) {$Y_{2,1}$ }; 
		\node  (4) at (2,0) {$Y_{2,-1}Y_{2,1}$ };
		\node  (5) at (2,-1) {$ \scalemath{0.6}{ Y_{2,-3}Y_{2,-1}Y_{2,1} }$ };
		\node  (6) at (4,0.5) {$Y_{1,2}$ };
		\node  (7) at (4,-1.5) {\fbox{$Y_{1,-2}Y_{1,2}$ }};
            \node  (8) at (2,-2) {\fbox{$ \scalemath{0.6}{  Y_{2,-5}Y_{2,-3}Y_{2,-1}Y_{2,1} }$ }};
     
     \draw[->] (2)--(1); 
     \draw[->] (4)--(3);
      \draw[->] (5)--(4); 
     \draw[->] (7)--(6);
     \draw[->] (3)--(1);
     \draw[->] (1)--(5); 
     \draw[->] (6)--(4);
     \draw[->] (4)--(7); 
     \draw[->] (8)--(5); 
     \draw[->] (7)--(8); 
     \draw[->] (5)--(2); 
   
\end{tikzpicture} 
\caption{An initial cluster for $K_0(\mathcal{C}_{\ell}^{B_2})$, $\ell=1$.}
\label{fig:initial cluster B2 and l is 1}
\end{figure} 

There are $25$ prime modules (not including the $3$ frozen modules) in $\mathcal{C}_{1}^{B_2}$ are 
\begin{align*}
& L(Y_{1,2}), \ L(Y_{1,0}), \ L(Y_{1,-2}), \ L(Y_{1,-4}), \ L(Y_{2,1}), \ L(Y_{2,-1}), \ L(Y_{2,-3}), \ L(Y_{2,-5}), \\
& L(Y_{1, -4} Y_{2, 1}), \
L(Y_{1, -4} Y_{1, 2}), \ 
L(Y_{1, 0} Y_{2, -5}), \ 
L(Y_{1, 2} Y_{2, -3}), \ 
L(Y_{2, -5} Y_{2, 1}), \
L(Y_{2, -5} Y_{2, -3}), \\
& L(Y_{2, -3} Y_{2, -1}), \
L(Y_{2, -1} Y_{2, 1}), \ 
L(Y_{1, -4} Y_{2, -1} Y_{2, 1}), \ 
L(Y_{1, 0} Y_{2, -5} Y_{2, -3}), \ 
L(Y_{1, 2} Y_{2, -5} Y_{2, -3}), \\
& 
L(Y_{1, 2} Y_{2, -3} Y_{2, -1}), \
L(Y_{2, -5} Y_{2, -3} Y_{2, -1}), \ 
L(Y_{2, -3} Y_{2, -1} Y_{2, 1}), \ 
L(Y_{1, 0} Y_{1, 2} Y_{2, -5} Y_{2, -3}), \\
& 
L(Y_{1, 2} Y_{2, -5} Y_{2, -3} Y_{2, -1}), \
L(Y_{1, 2} Y_{2, -5} Y_{2, -3} Y_{2, -3} Y_{2, -1}).
\end{align*} 

We have the following truncated $q$-characters of prime modules (not including initial cluster variables and frozen variables) in $\mathcal{C}_{\ell}$:

\begin{align*}
& \tchi(L(Y_{2,-1}))=Y_{2,-1}(v_{2,0} + 1), \quad \tchi(L(Y_{1,-2}))=Y_{1,-2}(v_{1,0} + 1),
\end{align*}

\begin{align*}
& \tchi(L(Y_{1,2} Y_{2,-3}))=Y_{1,2} Y_{2,-3}(v_{2,-2} + 1), \quad \tchi(L(Y_{1,0} Y_{2,-5}))=Y_{1,0} Y_{2,-5}(v_{2,-4} + 1), 
\end{align*}

\begin{align*}
& \tchi(L(Y_{2,1} Y_{1,-4}))=Y_{2,1} Y_{1,-4}(v_{1,-2} + 1), \quad \tchi(L(Y_{2,-3}))=Y_{2,-3}(v_{1,0} v_{2,-2} + v_{2,-2} + 1),
\end{align*}

\begin{align*}
& \tchi(L(Y_{2,1} Y_{2,-5}))=Y_{2,1} Y_{2,-5}(v_{1,-2} v_{2,-4} + v_{2,-4} + 1), \\
& \tchi(L(Y_{2,-5}))=Y_{2,-5}(v_{1,-2} v_{2,-4} + v_{2,-4} + v_{2,0} v_{1,-2} v_{2,-4} + 1),
\end{align*}

\begin{align*}
& \tchi(L(Y_{2,-1} Y_{2,1} Y_{1,-4}))=Y_{2,-1} Y_{2,1} Y_{1,-4}(v_{2,0} v_{1,-2} + v_{1,-2} + 1), \\
& \tchi(L(Y_{1,2} Y_{2,-1} Y_{2,-3}))=Y_{1,2} Y_{2,-1} Y_{2,-3}(v_{2,0} v_{2,-2} + v_{2,0} + 1),
\end{align*}

\begin{align*}
& \tchi(L(Y_{1,2} Y_{1,-4}))=Y_{1,2} Y_{1,-4}(v_{2,0} v_{1,-2} + v_{1,-2} + v_{2,0} v_{1,-2} v_{2,-2} + 1), \\
& \tchi(L(Y_{2,-1} Y_{2,-3}))=Y_{2,-1} Y_{2,-3}(v_{2,0} v_{2,-2} + v_{2,0} + v_{1,0} v_{2,0} v_{2,-2} + 1),
\end{align*}

\begin{align*}
& \tchi(L(Y_{1,0} Y_{1,2} Y_{2,-3} Y_{2,-5}))=Y_{1,0} Y_{1,2} Y_{2,-3} Y_{2,-5}(v_{2,-2} v_{2,-4} + v_{2,-2} + 1), \\
& \tchi(L(Y_{1,2} Y_{2,-3} Y_{2,-5}))=Y_{1,2} Y_{2,-3} Y_{2,-5}(v_{2,-2} v_{2,-4} + v_{2,-2} + v_{1,-2} v_{2,-2} v_{2,-4} + 1),
\end{align*}

\begin{align*}
& \tchi(L(Y_{1,-4}))=Y_{1,-4}(v_{2,0} v_{1,-2} + v_{1,-2} + v_{2,0} v_{1,-2} v_{2,-2} + v_{1,0} v_{2,0} v_{1,-2} v_{2,-2} + 1), \\
& \tchi(L(Y_{1,2} Y_{2,-1} Y_{2,-3} Y_{2,-5}))=Y_{1,2} Y_{2,-1} Y_{2,-3} Y_{2,-5}(v_{2,0} v_{2,-2} + v_{2,0} + v_{2,0} v_{2,-2} v_{2,-4} + 1),
\end{align*}

\begin{align*}
& \tchi(L(Y_{1,0} Y_{2,-3} Y_{2,-5}))=Y_{1,0} Y_{2,-3} Y_{2,-5}(v_{1,0} v_{2,-2} + v_{2,-2} v_{2,-4} + v_{2,-2} + v_{1,0} v_{2,-2} v_{2,-4} + 1), 
\end{align*}

\begin{align*}
\tchi(L(Y_{2,-1} Y_{2,-3} Y_{2,-5})) =  & Y_{2,-1} Y_{2,-3} Y_{2,-5}(v_{2,0} v_{2,-2} + v_{2,0}  \\
& + v_{1,0} v_{2,0} v_{2,-2} + v_{2,0} v_{2,-2} v_{2,-4} + v_{1,0} v_{2,0} v_{2,-2} v_{2,-4} + 1),
\end{align*}

\begin{align*}
\tchi(L(Y_{2,-3} Y_{2,-5}))= & Y_{2,-3} Y_{2,-5}(v_{1,0} v_{2,-2} + v_{2,-2} v_{2,-4} + v_{2,-2} + v_{1,0} v_{2,-2} v_{2,-4} \\ 
& + v_{1,-2} v_{2,-2} v_{2,-4} + v_{1,0} v_{1,-2} v_{2,-2} v_{2,-4} + 1),
\end{align*}

\begin{align*}
& \tchi(L(Y_{1,2} Y_{2,-1} {Y_{2,-3}}^2 Y_{2,-5}))= Y_{1,2} Y_{2,-1} {Y_{2,-3}}^2 Y_{2,-5} (2 v_{2,0} v_{2,-2} + v_{2,0} {v_{2,-2}}^2 + v_{1,0} v_{2,0} {v_{2,-2}}^2 \\
& \qquad + v_{2,0} {v_{2,-2}}^2 v_{2,-4}  + v_{2,0} + v_{2,-2} + v_{1,0} v_{2,0} v_{2,-2} + v_{2,0} v_{2,-2} v_{2,-4} + v_{1,0} v_{2,0} {v_{2,-2}}^2 v_{2,-4} + 1).
\end{align*}

By using the above F-polynomials, we find that ${\bf N'}_{B_n, 1}^{(5)}$ has $25$ facets. The polytope ${\bf N'}_{B_n, 1}^{(d)}$ ($d \ge 5$) also has $25$ facets and it is the type $D_5$ associahedron. 

We expect that for every $n \in \ZZ_{\ge 2}$, ${\bf N'}_{B_n, 1}^{(d)}$ ($d$ is large enough) is the type $D_{2n+1}$ associahedron. 

\subsection{Tropical Fans for Quantum Affine Algebras}

Let $\mathcal{M}^{(0)}=\mathcal{M}$ be the set of all equivalence classes of Kirillov-Reshetikhin modules of $U_q(\widehat{\mathfrak{g}})$ in $\mathcal{C}_{\ell}$. By tropicalizing all $\tchi(L(M))/M$, $L(M) \in \mathcal{M}$, we obtain piecewise linear functions in the space of dimension $r-m$ parametrized by $y_{i,j}$ ($y_{i,j}$ is the tropical version of $v_{i,j}$), where $r$ is the number of fundamental modules and $m$ is the number of frozen variables in $\mathcal{C}_{\ell}$. Such a function is linear on a collection of cones; these cones assemble to define a polyhedral fan. The common refinement of these fans is the normal fan $\mathcal{N}({\bf N}_{\mathfrak{g}, \ell}^{(0)})$ of the Newton polytope ${\bf N}_{\mathfrak{g}, \ell}^{(0)}$. 

For $d\ge 1$, let $\mathcal{M}^{(d)}$ be the set of all equivalence classes of simple modules corresponding to rays of $\mathcal{N}({\bf N}_{\mathfrak{g}, \ell}^{(d-1)})$. By tropicalizing all $\chi_q(L(M))$, $L(M) \in \mathcal{M}^{(d)}$, and using the same procedure as above, we obtain the normal fan $\mathcal{N}({\bf N}_{\mathfrak{g}, \ell}^{(d)})$ of the Newton polytope ${\bf N}_{\mathfrak{g}, \ell}^{(d)}$ defined in Section \ref{subsec:Newton polytopes for quantum affine algebras}.

\section{Physical Motivation: Stringy Integrals and CEGM Scattering Amplitudes} \label{sec:physics}

In this section, we propose a formula which extends the main construction in the work of Arkani-Hamed, He, Lam \cite{AHL2019Stringy} on so-called Grassmannian string integrals, and Cachazo, Early, Guevara, Mizera (CEGM) \cite{CEGM2019} on generalized biadjoint scalar amplitudes. Grassmannian string integrals and generalized biadjoint scalar amplitudes are related by taking a certain $\alpha' \rightarrow 0$ limit of the stringy integral.

\subsection{Stringy Integrals For Grassmannian Cluster Algebras} \label{subsec:stringy integeral for Grkn}

From a physical point of view, the central objective of this subsection which we describe here is twofold. First, in this subsection we give an explicit formula for a completion of the stringy integral by making use of \textit{all} of the elements in Lusztig's dual canonical basis of $\CC[\Gr(k,n)]$; the rest of the subsection aims to provide a combinatorial framework for the evaluation of a limit which is standard in physics, the so-called $\alpha' \rightarrow 0$ limit of the Grassmannian string integral, which is known to be given (\cite[Claim 1]{AHL2019Stringy}) by the CEGM scattering equations formula.

Such calculations are still highly nontrivial, but the formula which we propose removes an enormous amount of redundancy by making use of character polynomials for only \textit{prime} tableaux. It is known that any simple 
$U_q(\widehat{\mathfrak{g}})$-module decomposes as the tensor product of prime simple modules \cite{CP94}. Therefore the $q$-character\footnote{There is a connection between the $q$-character of a simple module $L(M)$ and the polynomial $\ch_{T_M}$, where $T_M$ is the tableau corresponding to $M$, see Sections \ref{subsec:Grassmannian cluster algebras and semistandard Young tableaux} and \ref{subsec:dominant monomials and tableaux}.} of any simple $U_q(\widehat{\mathfrak{g}})$-module is the product of the $q$-characters of its prime factors \cite{FR98}.

Moreover, our formula is essentially nonrecursive using $\ch_T$ in Theorem 5.8 in \cite{CDFL}, and it is more general than possible constructions coming from cluster algebras which use only cluster variables.  

Arkani-Hamed, He, and Lam introduced Grassmannian string integrals in \cite[Equation (6.11)]{AHL2019Stringy}:
\begin{eqnarray}\label{eq: stringy integral original plucker case}
	\mathbf{I}_{k,n} & =&  (\alpha')^{a}\int_{\left(\mathbb{R}_{>0}^{n-k-1}\right)^{\times (k-1)}}\left(\prod_{(i,j)}\frac{dx_{i,j}}{x_{i,j}}\right)\left(\prod_{J} p_J^{-\alpha'c_J}(x_{i,j})\right),
\end{eqnarray}
where $a = (k-1)(n-k-1)$, $\alpha'$, $c_J$ are some parameters, $p_J$'s are Pl\"{u}cker coordinates, the product runs over all $k$-element subsets of $[n]$.  

We emphasize that the original formulations (\ref{eq: stringy integral original plucker case}) in \cite{CEGM2019} and \cite{AHL2019Stringy} involved only the finite collection of all Pl\"{u}cker coordinates.  

We now define the completion of the Grassmannian string integral, using for the integrand all prime elements in the dual canonical basis of $\CC[\Gr(k,n)]$.  

\begin{definition} \label{def:Grassmannian string integral, finite d}
For $2 \le k \le n-2$ and every $d\ge 1$, we define
\begin{eqnarray}\label{eq: stringy integral finite d}
	\mathbf{I}^{(d)}_{k,n} & =&  (\alpha')^{a}\int_{\left(\mathbb{R}_{>0}^{n-k-1}\right)^{\times (k-1)}}\left(\prod_{(i,j)}\frac{dx_{i,j}}{x_{i,j}}\right)\left(\prod_{T} \ch_T^{-\alpha'c_T}(x_{i,j})\right).
\end{eqnarray}
where the second product is over all tableaux $T$ such that the face $\mathbf{F}_T$ corresponding to $T$ (see Section \ref{subsec: the face corresponding to T}) is a (codimension one) facet of $\mathbf{N}^{(d-1)}_{k,n}$.
Here we abbreviate $a=(k-1)(n-k-1)$.  Also $c_T$, $x_{i,j}> 0$ are positive (real) parameters, and  $\alpha'$ is a parameter known in physics as the string tension.  The first product is over $(i,j) \in \lbrack 1,k-1\rbrack \times \lbrack 1,n-k-1\rbrack$, and we have chosen the normalization where $x_{i,n-k} = 1$ for all $i=1,\ldots, k-1$.  
\end{definition}
In the integral (\ref{eq: stringy integral finite d}) we have conditions under which the integral converges, namely that the parameters $\alpha_{i,j}$ and $c_T$ must be chosen such that the origin is in the interior of the Newton polytope, see \cite[Claim 1]{AHL2019Stringy} for details.

Denote by ${\rm PSSYT}_{k,n}^{r} \subset {\rm SSYT}(k, [n])$ the set of prime tableaux in ${\rm SSYT}(k, [n])$ with $r$ or less columns and by ${\rm PSSYT}_{k,n} \subset {\rm SSYT}(k, [n])$ the set of all prime tableaux in ${\rm SSYT}(k, [n])$.

It is natural\footnote{See talks by Arkani-Hamed, Frost, Plamondon, Salvatori, and Thomas in \cite{AFPST21_surface, AFPST21}.} to introduce the $d\rightarrow \infty$ limit of the Grassmannian string integral (\ref{eq: stringy integral finite d}):
\begin{eqnarray}\label{eq: stringy integral infinite}
	\mathbf{I}^{(\infty)}_{k,n} & =&  (\alpha')^{a}\int_{\left(\mathbb{R}_{>0}^{n-k-1}\right)^{\times (k-1)}} \left(\prod_{(i,j)}\frac{dx_{i,j}}{x_{i,j}}\right)\left( \prod_{T \in \text{PSSYT}_{k,n}} \ch_T^{-\alpha'c_T}(x_{i,j})\right).
\end{eqnarray}
For finite type cluster algebras, our integrand is finite. However, starting at $(k,n) = (3,9)$ the integrand involves an infinite product.

We also introduce another version of the Grassmannian string integral (\ref{eq: stringy integral finite d}) using all prime tableaux up to certain columns. 
\begin{definition} \label{def:Grassmannian string integral using prime tableaux up to certain columns}
For $2 \le k \le n-2$ and $r \ge 1$, we define
\begin{eqnarray}\label{eq: stringy integral using tableaux up to certain columns}
	\mathbf{I}'^{(r)}_{k,n} & =&  (\alpha')^{a}\int_{\left(\mathbb{R}_{>0}^{n-k-1}\right)^{\times (k-1)}} \left(\prod_{(i,j)}\frac{dx_{i,j}}{x_{i,j}}\right)\left( \prod_{T \in \text{PSSYT}_{k,n}^{r}} \ch_T^{-\alpha'c_T}(x_{i,j})\right),
\end{eqnarray}
where $a = (k-1)(n-k-1)$, and $\alpha'$, $c_T$ are certain parameters defined in the same way as Definition \ref{def:Grassmannian string integral, finite d}. 

\end{definition} 
Note that in the limit $r\rightarrow \infty$ the integrands for (\ref{eq: stringy integral using tableaux up to certain columns}) and (\ref{eq: stringy integral infinite}) coincide. In this way we have a combinatorial construction which relates prime tableaux to stringy integrals, and a geometric interpretation of the set of all prime tableaux in terms of a polytope.

The polynomials $\text{ch}_T$ that appear in the integrands of (\ref{eq: stringy integral finite d}) and (\ref{eq: stringy integral using tableaux up to certain columns}) are in bijection with prime tableaux and can be calculated using (\ref{eq:monomToTableaux}).

An important problem which may help with the evaluation will be investigated in Section \ref{subsec:u-equations}: to rewrite Equation \eqref{eq: stringy integral infinite} in terms of rational functions which are invariant under the torus action, that is the so-called $u$-variables \cite{AHLT2019}, and then to calculate the binary relations among them. See also \cite{E2021} for another physical application of binary relations in the context of CEGM scattering amplitudes.

By \cite[Claim 1]{AHL2019Stringy}, the leading order term in the series expansion around $\alpha' = 0$ has a beautiful interpretation as the volume of a polytope, where the simple poles correspond to facets.  The polytope is dual to the Newton polytope $\mathbf{N}_{k,n}^{(1)}$.  This leading order contribution was formulated originally in \cite{CEGM2019} by Cachazo, Early, Guevara and Mizera (CEGM) using the scattering equations formalism.

\begin{remark}
    The stringy integral in Equation \eqref{eq: stringy integral infinite} converges if and only if the origin is in the interior of the Newton polytope, see \cite[Claim 1]{AHL2019Stringy}; however, the $\alpha' \rightarrow 0$ limit is calculated by the CEGM scattering equations formula \cite{CEGM2019}, which has no such convergence limitation.
\end{remark}

It turns out that the limit $\alpha' \rightarrow 0$ of the Grassmannian string integral 
 (\ref{eq: stringy integral original plucker case}) coincides with the CEGM scattering equations formula \cite{AHL2019Stringy}.  Let us sketch the CEGM formula, referring to \cite{CEGM2019} for details.

First we define a scattering potential function 
$$\mathcal{S}^{(d=1)}_{k,n} = \sum_{J}\log(p_J) \mathfrak{s}_J,$$
where $p_J$ is the maximal $k\times k$ minor with column set $J = \{j_1,\ldots, j_k\}$, and the \textit{Mandelstam variables} $\mathfrak{s}_J$ are coordinate functions on the \textit{kinematic space}
$$\mathcal{K}(k,n) = \left\{(\mathfrak{s}_J) \in \mathbb{R}^{\binom{\n}{k}}: \sum_{J: J \ni i }\mathfrak{s}_J = 0,i=1,\ldots, n \right\}.$$
Then \cite{CEGM2019} defined the (planar) generalized biadjoint scalar amplitude
$$m^{(k)}_n = \sum_{c \in \text{crit}(\mathcal{S}^{(1)}_{k,n})}\frac{1}{\det'\Phi}\left(\prod_{j=1}^n\frac{1}{p_{j,j+1,\ldots, j+k-1}(c)}\right)^2,$$
where the sum is over all critical points $c$ of $\mathcal{S}^{(1)}_{k,n}$, and where $\det'\Phi$ is the so-called reduced Hessian determinant (see \cite[Equation 2.4]{CEGM2019} for details). For example,  
$$m^{(2)}_4 = \frac{1}{s_{12}} + \frac{1}{s_{23}},$$
$$m^{(2)}_5 = \frac{1}{s_{12}s_{34}} + \frac{1}{s_{23}s_{45}} + \frac{1}{s_{34}s_{15}} + \frac{1}{s_{12}s_{45}} + \frac{1}{s_{23}s_{15}},$$
and
	\begin{equation}
		\label{eq:14terms}  \begin{matrix}
			m^{(2)}_6 = & \,\,\frac{1}{s_{12} s_{34} s_{56}}
			+\frac{1}{s_{12} s_{56}  s_{123}}
			+\frac{1}{s_{23} s_{56}  s_{123}}
			+\frac{1}{s_{23} s_{56}  s_{234}}
			+\frac{1} {s_{34} s_{56} s_{234}}
			+ \frac{1}{s_{16} s_{23} s_{45}}
			+ \frac{1}{ s_{12} s_{34} s_{345}} \smallskip  \\
			&   +\, \frac{1}{s_{12} s_{45} s_{123}}
			+ \frac{1}{s_{12} s_{45} s_{345}}
			+\frac{1}{s_{16} s_{23} s_{234}}
			+\frac{1} {s_{16} s_{34} s_{234}}
			+\frac{1}{s_{16} s_{34} s_{345}}
			+\frac{1}{s_{16} s_{45} s_{345}}
			+\frac{1}{s_{23} s_{45} s_{123}}.
		\end{matrix}
	\end{equation}
In general, Cachazo-He-Yuan \cite{CHY14} introduced a compact formula for biadjoint scalar amplitudes (as well as amplitudes for many other Quantum Field Theories).

The $k\ge 3$ analog was discovered by Cachazo-Early-Guevara-Mizera (CEGM); they have have been the subject of intensive study since their introduction \cite{CEGM2019}.

A second expression for the leading order in the expansion around $\alpha' \rightarrow 0$ is a compact expression involving piecewise-linear functions,
$$\lim_{\alpha'\rightarrow0}\mathbf{I}^{(\infty)}_{k,n}=\int_{\mathbb{T}^{k-1,n-k}}\exp\left(-\sum_{T \in \text{PSSYT}_{k,n}}\mathfrak{s}_T\text{ch}^\text{Trop}_T(y_{i,j})\right)dy_{i,j},$$
where $\ch^\text{Trop}_T(y_{i,j})$ is the usual tropicalization of the character polynomial $\ch_T(x_{i,j})$, and where $\mathbb{T}^{k-1,n-k} = (\mathbb{T}^{n-k})^{\times (k-1)}$ and $\mathbb{T}^{n-k} = \RR^{n-k}/\RR(1, \ldots, 1)$. 

The fundamental tropical integral of this form was defined first in \cite{CE2020}, called there the \textit{global Schwinger parametrization} of Feynman diagrams.  In this work we propose a generalization of the integrand which includes all prime elements in Lusztig's dual canonical basis, as parameterized by prime tableaux.

Clearly there are many questions about this integral which we leave to future work.  One of these is highly nontrivial:
\begin{itemize}
    \item To evaluate the tropical limit, one has to either compute an infinite Minkowski sum, or else find a way to evaluate the CEGM formula for a scattering potential that involves an infinite summation indexed by prime tableaux.
\end{itemize}

\subsection{$u$-equations and $u$-variables} \label{subsec:u-equations}

In what follows, building on \cite[Section 6.2]{AHL2019Stringy}, we propose a system of so-called $u$-variables for the (infinite) Grassmannian string integral $\mathbf{I}^{(\infty)}_{k,n}$ defined in Equation \eqref{eq: stringy integral infinite}.  The motivation is to make the integrand manifestly compatible with the singularities of the function obtained by taking the $\alpha'\rightarrow 0$ limit.  To construct the infinite integrand is an important problem \cite[Section 12.3]{AHL2021}. The second step, to characterize the binary relations among the $u$-variables, will be considered in future work.

The new integrand is reorganized as a product of cross-ratios $u_T$ on the Grassmannian $\Gr(k,n)$, the so-called $u$-variables, one for each prime tableau $T$.  The $u$-variables \cite{AHLT2019} have been defined for finite-type cluster algebras arising from $\Gr(2,n)$ \cite{AHL2021} (and see the original work of Koba-Nielsen \cite{KN69} in the physics literature), but for general Grassmannians $\Gr(k,n)$ the cluster algebras are of infinite type and new methods are required \cite{AHL2019Stringy}. The main idea of our solution is to construct $u$-variables for Grassmannian string integrals as ratios of characters $\ch_T$ of prime tableaux $T$ (equivalently, as ratios of $q$-characters of prime modules of the quantum affine algebra $U_q(\widehat{\mathfrak{sl}_k})$).  In this way, our proposal for an integrand with infinite product of $u$-variables which satisfy certain binary-type identities, as has been explored in the finite type case in \cite{AHLT2019}.

We first formulate our proposal and then we label the $u$-variables and $u$-equations for $\mathbf{I}^{(2)}_{3,6}$ (in our notation) in \cite{AHL2019Stringy} to using prime tableaux.

\begin{definition} \label{def:stringy integeral for Grassmannians using u-variables}
For $k \le n$, we define
    \begin{eqnarray}
        \mathbf{I}^{(\infty)}_{k,n} & = & (\alpha')^{a}\int_{\left(\mathbb{R}_{>0}^{n-k-1}\right)^{\times (k-1)}} \prod_{i,j} \frac{dx_{i,j}}{x_{i,j}} \prod_{T \in {\rm PSSYT}_{k,n}} (u_T)^{\alpha' U_T},  
    \end{eqnarray}
where $\alpha'$, $U_T$ are some parameters, and $u_T$ is the $u$-variable corresponding to a prime tableau $T$ which is defined in (\ref{eq:u variables}). 
\end{definition}

Jensen, King, and Su \cite{JKS} introduced an additive categorification of Grassmannian cluster algebras using a category ${\rm CM}(B_{k,n})$ of Cohen-Macaulay $B_{k,n}$-modules, where $B_{k,n}$ is a certain quotient of the complete path algebra of a certain quiver. According to their result, there is a one to one correspondence between cluster variables $\CC[\Gr(k,n)]$ and reachable (meaning that the module can be obtained by mutations) rigid indecomposable modules in ${\rm CM}(B_{k,n})$. On the other hand, cluster variables of $\CC[\Gr(k,n)]$ are in one to one correspondence with reachable (meaning that the tableau can be obtained by mutations) prime real tableaux in $\SSYT(k, [n])$. Therefore there is a one to one correspondence between reachable rigid indecomposable modules in ${\rm CM}(B_{k,n})$ and reachable prime real tableaux in $\SSYT(k, [n])$. In particular, in finite type cases, there is a one to one correspondence between indecomposable modules (in finite type, all indecomposable modules are rigid) in ${\rm CM}(B_{k,n})$ and prime tableaux in $\SSYT(k, [n])$ (in finite type, all prime tableau are real). 

We conjecture that in general, there is a one to one correspondence between indecomposable modules in ${\rm CM}(B_{k,n})$ and prime tableaux in $\SSYT(k, [n])$. Denote by $M_T$ the indecomposable module in ${\rm CM}(B_{k,n})$ corresponding to a prime tableau $T$. We can label the Auslander-Reiten quiver \cite{AR75, JKS} by prime tableaux instead of indecomposable modules, see Figure \ref{fig: AR quiver for Gr36}. 

\begin{definition} \label{def:u-variables}
For every mesh 
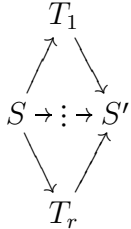
\begin{align*}
\xymatrix@-6mm@C-0.2cm{
& & T_1 \ar[rdd]\\
 \\
&  S  \ar[r]     \ar[rdd]\ar[ruu]&  \ar[r] \vdots   & S'  \\
 \\
& & T_r \ar[ruu]  \\
}    
\end{align*} 
in the Auslander-Reiten quiver of ${\rm CM}(B_{k,n})$, we define the corresponding $u$-variable as 
\begin{align} \label{eq:u variables}
u_{S} = \frac{\prod_{i=1}^{r} \ch_{T_i}}{\ch_{S} \ch_{S'}}.
\end{align}
\end{definition}
Here we label the $u$-variables by semistandard Young tableaux rather than noncrossing tuples. The mesh can be degenerate. For example, in Figure \ref{fig: AR quiver for Gr36}, $u_{126} = \frac{p_{136}}{p_{126}}$. 

\begin{conjecture} \label{conj:u variables are solutions of u equations general case for Grkn}
There are unique integers $a_{T,T'}$, where $T$, $T'$ are prime tableaux, such that $u$-variables (\ref{eq:u variables}) are solutions of the system of equations 
\begin{align*}
u_T  + \prod_{T' \in {\rm PSSYT}_{k,n}} u_{T'}^{a_{T, T'}} = 1, 
\end{align*} 
where $T$ runs over all non-frozen prime tableaux in ${\rm PSSYT}_{k,n}$. 
\end{conjecture}
The equations in Conjecture \ref{conj:u variables are solutions of u equations general case for Grkn} are called $u$-equations. 

\begin{remark}
General $u$-equations have been introduced in \cite{AFPST21_surface, AFPST21, AFPST23a, AFPST23b} in the setting of representations of quiver with relations and cluster categories of finite type. In our paper, we work in the setting of the Grassmannian cluster category ${\rm CM}(B_{k,n})$ \cite{JKS}. 
We expect that $a_{T,T'}$ is the compatibility degree defined in \cite{FG19} when tableaux $T$, $T'$ are cluster variables. 
\end{remark}

We give an example to explain Conjecture \ref{conj:u variables are solutions of u equations general case for Grkn}.
\begin{example}
In the case of $\CC[\Gr(3,6)]$, the $u$-equations are
\begin{align*}
& u_{124} + u_{135}u_{136}u_{235}u_{236}u_{356}u_{135,246} =1, \\
& u_{125} + u_{136}u_{346}u_{135,246}u_{246}u_{236}u_{146}=1, \\
& u_{135} + u_{135,246}u_{246}^2u_{124,356}u_{245}u_{346}u_{236}u_{146}u_{256}u_{124} = 1, \\
& u_{124,356} + u_{135}u_{136}u_{145}u_{146}u_{235}u_{236}u_{245}u_{246}u_{135,246}^2 =1, 
\end{align*}
and their cyclic shifts. Note that the cyclic shifts of the indices of all Pl\"{u}cker coordinates in $\ch_T$ corresponds to promotions of $T$, \cite{Sch72}. 

The solutions of the $u$-equations can be read from the Auslander-Reiten quiver. We have
\begin{align*}
& u_{126} = \frac{p_{136}}{p_{126}}, \ u_{345} = \frac{p_{346}}{p_{345}}, \ u_{125} = \frac{p_{126}p_{135}}{p_{125}p_{136}}, \ u_{136} = \frac{\ch_{135,246}}{p_{136}p_{245}}, \ u_{245} = \frac{p_{345}p_{246}}{p_{245}p_{346}}, 
\end{align*}
\begin{align*} 
& u_{346} = \frac{\ch_{124,356}}{p_{346}p_{125}}, \ u_{124,356} = \frac{p_{125}p_{134}p_{356}}{\ch_{124,356}p_{135}}, \ u_{134} = \frac{p_{135}p_{234}}{p_{134}p_{235}}, \ u_{135} = \frac{p_{136}p_{145}p_{235}}{p_{135}\ch_{135,246}}, \end{align*}
\begin{align*} 
& u_{235} = \frac{\ch_{135,246}}{p_{235}p_{146}}, \ u_{135,246} = \frac{p_{146}p_{245}p_{236}}{\ch_{135,246}p_{246}}, \ u_{146} = \frac{p_{246}p_{156}}{p_{146}p_{256}}, \ u_{246} = \frac{p_{346}p_{256}p_{124}}{p_{246}\ch_{124,356}}, 
\end{align*}
\begin{align*} 
& u_{256} = \frac{\ch_{124,356}}{p_{256}p_{134}}, \ u_{234} = \frac{p_{235}}{p_{234}}, \ u_{156} = \frac{p_{256}}{p_{156}}, \ u_{356} = \frac{p_{135}p_{456}}{p_{356}p_{145}}, \ u_{145} = \frac{\ch_{135,246}}{p_{145}p_{236}}, 
\end{align*}
\begin{align*} 
& u_{236} = \frac{p_{246}p_{123}}{p_{236}p_{124}}, \ u_{124} = \frac{\ch_{124,356}}{p_{124}p_{356}}, \ u_{456} = \frac{p_{145}}{p_{456}}, \ u_{123} = \frac{p_{124}}{p_{123}}, 
\end{align*}
where we use $\ch_{T_1, \ldots, T_r}$ to denote $\ch_T$, and $T_i$'s are columns of $T$. Here $\ch_{124,356} = p_{124}p_{356}-p_{123}p_{456}$, and $\ch_{135,246}=p_{145}p_{236}-p_{123}p_{456}$. 
We checked that the $u$-variables satisfy $u$-equations  directly. The solution agrees with Section 9.3 of \cite{AHL2019Stringy}.
\end{example}

\begin{figure}
\adjustbox{scale=0.6,center}{%
\begin{tikzcd}
	&& {\begin{matrix} 1 \\ 2\\ 6 \end{matrix}} &&&& {\begin{matrix} 3 \\ 4\\ 5 \end{matrix}} \\
	& {\begin{matrix} 1 \\ 2\\ 5 \end{matrix}} && {\begin{matrix} 1 \\ 3\\ 6 \end{matrix}} && {\begin{matrix} 2 \\ 4\\ 5 \end{matrix}} && {\begin{matrix} 3 \\ 4\\ 6 \end{matrix}} \\
	&& {\begin{matrix} 2 \\ 3\\ 4 \end{matrix}} &&&& {\begin{matrix} 1 \\ 5\\ 6 \end{matrix}} &&& {\begin{matrix} 1 \\ 2\\ 5 \end{matrix}} \\
	& {\begin{matrix} 1 \\ 3\\ 4 \end{matrix}} && {\begin{matrix} 2 \\ 3\\ 5 \end{matrix}} && {\begin{matrix} 1 \\ 4\\ 6 \end{matrix}} && {\begin{matrix} 2 \\ 5\\ 6 \end{matrix}} && {\begin{matrix} 1 \\ 3\\ 4 \end{matrix}} \\
	{\begin{matrix} 1 & 3 \\ 2 & 5 \\ 4 & 6 \end{matrix}} && {\begin{matrix} 1 \\ 3\\ 5 \end{matrix}} && {\begin{matrix} 1 & 2 \\ 3 & 4\\ 5&6 \end{matrix}} && {\begin{matrix} 2 \\ 4\\ 6 \end{matrix}} && {\begin{matrix} 1 &3\\ 2&5\\4& 6 \end{matrix}} \\
	& {\begin{matrix} 3 \\ 5\\ 6 \end{matrix}} && {\begin{matrix} 1 \\ 4\\ 5 \end{matrix}} && {\begin{matrix} 2 \\ 3\\ 6 \end{matrix}} && {\begin{matrix} 1 \\ 2\\ 4 \end{matrix}} && {\begin{matrix} 3 \\ 5\\ 6 \end{matrix}} \\
	&& {\begin{matrix} 4 \\ 5\\ 6 \end{matrix}} &&&& {\begin{matrix} 1 \\ 2\\ 3 \end{matrix}}
	\arrow[from=2-2, to=1-3]
	\arrow[from=1-3, to=2-4]
	\arrow[from=4-2, to=3-3]
	\arrow[from=3-3, to=4-4]
	\arrow[from=4-2, to=5-3]
	\arrow[from=5-3, to=4-4]
	\arrow[from=2-2, to=5-3]
	\arrow[from=5-1, to=2-2]
	\arrow[from=5-1, to=4-2]
	\arrow[from=5-1, to=6-2]
	\arrow[from=6-2, to=5-3]
	\arrow[from=5-3, to=2-4]
	\arrow[from=5-3, to=6-4]
	\arrow[from=6-2, to=7-3]
	\arrow[from=7-3, to=6-4]
	\arrow[from=4-4, to=5-5]
	\arrow[from=2-4, to=5-5]
	\arrow[from=6-4, to=5-5]
	\arrow[from=5-5, to=4-6]
	\arrow[from=5-5, to=2-6]
	\arrow[from=2-6, to=5-7]
	\arrow[from=4-6, to=5-7]
	\arrow[from=4-6, to=3-7]
	\arrow[from=3-7, to=4-8]
	\arrow[from=5-7, to=4-8]
	\arrow[from=2-6, to=1-7]
	\arrow[from=1-7, to=2-8]
	\arrow[from=5-7, to=2-8]
	\arrow[from=2-8, to=5-9]
	\arrow[from=4-8, to=5-9]
	\arrow[from=5-9, to=4-10]
	\arrow[from=5-9, to=3-10]
	\arrow[from=5-9, to=6-10]
	\arrow[from=5-7, to=6-8]
	\arrow[from=6-8, to=5-9]
	\arrow[from=5-5, to=6-6]
	\arrow[from=6-6, to=5-7]
	\arrow[from=6-6, to=7-7]
	\arrow[from=7-7, to=6-8]
\end{tikzcd}
}
\caption{The Auslander-Reiten quiver for ${\rm CM}(B_{3,6})$ with vertices labelled by tableaux.}
\label{fig: AR quiver for Gr36}
\end{figure}
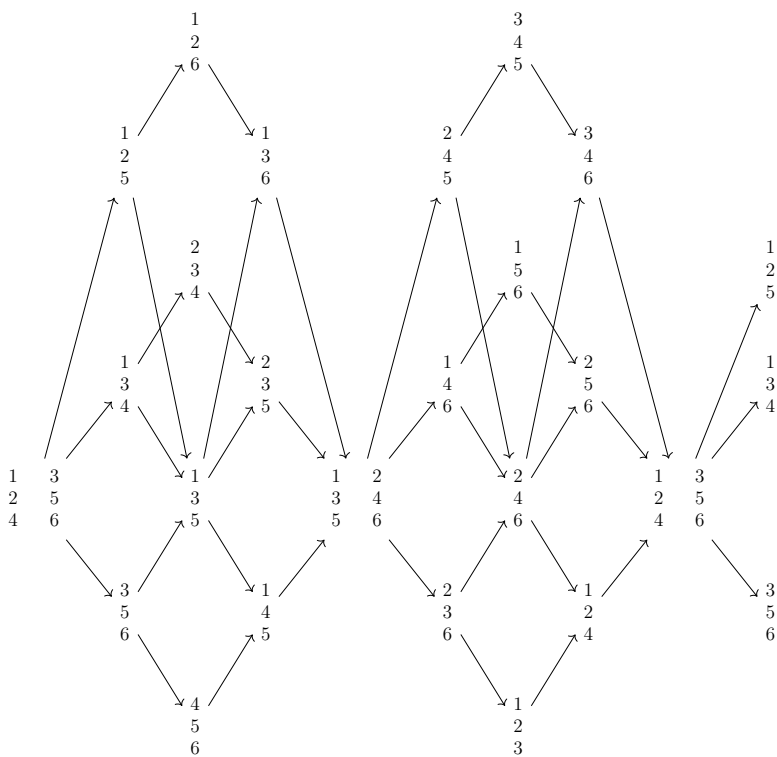

The same computations can be done for other finite type cases. The Auslander-Reiten quivers for Grassmannian cluster categories ${\rm CM}(B_{k,n})$ of finite type have been computed in \cite{JKS} (the vertices are labelled by Cohen-Macaulay modules in ${\rm CM}(B_{k,n})$) and \cite{DLaLi} (the vertices are labelled by tableaux). 

When $\CC[\Gr(k,n)]$ is of infinite type, the Auslander-Reiten quiver of the Grassmannian cluster category ${\rm CM}(B_{k,n})$ has infinitely many components. We will study Conjecture \ref{conj:u variables are solutions of u equations general case for Grkn} about the $u$-equations and their solutions in the future.   

\subsection{Stringy Integrals For Quantum Affine Algebras} \label{subsec:stringy integerals for quantum affine algebras}

We generalize the stringy integrals in Section \ref{subsec:stringy integeral for Grkn} to the setting for any quantum affine algebra as follows. 

Let $\mathfrak{g}$ be a simple Lie algebra over $\CC$ and $\ell \ge 0$. Recall that $\hat{I} = \{(i,s): i \in I, s = \xi(i) - 2d_ir, r \in [0, \ell] \}$, where $\xi: I \to \ZZ$ is a chosen height function, see Section \ref{sec:quantum affine algebras}.

\begin{definition} \label{def:stringy integral for quantum affine algebras}
For every simple Lie algebra $\mathfrak{g}$ over $\CC$, $\ell \ge 1$, and $d \ge 1$, we define  
\begin{eqnarray}\label{eq: stringy integral for quantum affine algebras}
\mathbf{I}^{(d)}_{\mathfrak{g},\ell} & =&  (\alpha')^{|\hat{I}|}\int_{\mathbb{R}_{>0}^{|\hat{I}|}}\left(\prod_{(i,s) \in \hat{I}}\frac{dY_{i,s}}{Y_{i,s}}\right)\left(\prod_{M} \tchi(L(M))^{-\alpha'c_M} \right),
\end{eqnarray}
where the product is over all dominant monomials $M$ such that the modules $L(M)$ correspond to facets of $\mathbf{N}^{(d-1)}_{\mathfrak{g}, \ell}$, and $\alpha'$, $c_M$ are some parameters.
\end{definition} 

We also define another version of stringy integrals for quantum affine algebras.

\begin{definition} \label{def:stringy integral for quantum affine algebras using degree}
For every simple Lie algebra $\mathfrak{g}$ over $\CC$, $\ell \ge 1$, and $d \ge 1$, we define  
\begin{eqnarray}\label{eq: stringy integral for quantum affine algebras, another version}
\mathbf{I'}^{(d)}_{\mathfrak{g},\ell} & =&  (\alpha')^{|\hat{I}|}\int_{\mathbb{R}_{>0}^{|\hat{I}|}}\left(\prod_{(i,s) \in \hat{I}}\frac{dY_{i,s}}{Y_{i,s}}\right)\left(\prod_{M} \tchi(L(M))^{-\alpha'c_M} \right),
\end{eqnarray}
where the product is over all dominant monomials $M$ in $\mathcal{P}^+_{\ell}$ of degree less or equal to $d$, and $\alpha'$, $c_M$ are some parameters.
\end{definition} 

We also define stringy integrals for quantum affine algebras in the case when $d \to \infty$ as follows.

\begin{definition} \label{def: stringy integral for quantum affine algebras, d is inifite}
For every simple Lie algebra $\mathfrak{g}$ over $\CC$ and $\ell \ge 1$, we define
\begin{eqnarray}\label{eq: stringy integral for quantum affine algebras, d is inifite}
\mathbf{I}^{(\infty)}_{\mathfrak{g},\ell} & =&  (\alpha')^{|\hat{I}|}\int_{\mathbb{R}_{>0}^{|\hat{I}|}}\left(\prod_{(i,s) \in \hat{I}}\frac{dY_{i,s}}{Y_{i,s}}\right)\left(\prod_{M} \tchi(L(M))^{-\alpha'c_M} \right),
\end{eqnarray}
where the product is over all 
dominant monomials $M$ such that $L(M)$'s are prime modules in $\mathcal{C}_{\ell}$, and $\alpha'$, $c_M$ are some parameters.
\end{definition}

We hope that the stringy integrals for quantum affine algebras will have applications to physics. 

\section{Limit $g$-vectors, Limit Facets, and Prime Non-real Modules} \label{sec:limit g vector}
In this section, we study prime non-real modules of quantum affine algebras using limit $g$-vectors. 

\subsection{Limit $g$-vectors and Limit Facets} \label{subsec:limit g vectors and limit facets}

It is observed in \cite{DFGK, HP21} that some prime non-real elements in the dual canonical basis of $\CC[\Gr(k,n)]$ which can be computed using limit $g$-vectors (these limit $g$-vectors are called limit rays in \cite{DFGK, HP21}). We generalize the concept of limit $g$-vectors to any cluster algebra in the following.

For a vector $v = (v_1, \ldots, v_m)$ in $\RR^{m}$, denote its $l^2$-norm by $ \left\Vert v \right\Vert = \sqrt{\sum_{i=1}^m |v_i|^2}$.
\begin{definition} \label{def: limit g-vectors}
For a cluster algebra $\mathcal{A}$ of infinite type of rank $m$, we say that a sequence of (distinct) $g$-vectors $g_1, g_2, \ldots$ of $\mathcal{A}$ has a limit $g$ if the greatest common factor of entries of $g$ is $1$ and for every $\epsilon >0$, there is a positive integer $N$ such that for every $j \ge N$, there is some positive real number $c_j$ such that $\left\Vert c_jg-g_j \right\Vert < \epsilon$. 
\end{definition}

\begin{definition} \label{def: limit normal vector}
Let $\mathfrak{g}$ be a simple Lie algebra over $\CC$ and let $\ell \in \ZZ_{\ge 1}$, $d \in \ZZ_{\ge 0}$. We say that a facet of a Newton polytope ${\bf N}_{\mathfrak{g}, \ell}^{(d)}$ defined in Section \ref{sec:Newton polytopes for quantum affine algebras} for a quantum affine algebra is a limit facet if the $g$-vector of the module corresponding to the facet is a limit $g$-vector of a sequence of $g$-vectors of $U_q(\widehat{\mathfrak{g}})$-modules. 
\end{definition}

We conjecture that every simple module corresponding to a limit $g$-vector is prime non-real. 
\begin{conjecture} \label{conj:limit g vectors are prime non-real}
Let $\mathfrak{g}$ be a simple Lie algebra over $\CC$ and let $\ell \in \ZZ_{\ge 1}$. If the $g$-vector of a simple module $L(M)$ in $\mathcal{C}_{\ell}$ is a limit $g$-vector of the cluster algebra corresponding to $\mathcal{C}_{\ell}$, then $L(M)$ is prime non-real. 
\end{conjecture}

Conjecture \ref{conj:limit g vectors are prime non-real} can be generalized to a more general setting. Suppose that $\mathcal{A}$ is a cluster algebra and it has a linear basis $B$ and every element $b$ in $B$ corresponds to a unique $g$-vector $g_b$ of $\mathcal{A}$. We say that an element $b$ in $B$ is real if $b^2 \in B$. We say that $b \in B$ is prime if $b \ne b'b''$ for any non-trivial elements $b', b'' \in B$. 
\begin{conjecture}\label{conj:limit g vectors are prime non-real for general basis}
Let $\mathcal{A}$ be a cluster algebra and suppose that it has a linear basis $B$ and every element $b$ in $B$ corresponds to a unique $g$-vector $g_b$ of $\mathcal{A}$. For any $b \in B$, if the $g$-vector $g_b$ corresponding to $b$ is a limit $g$-vector of the cluster algebra $\mathcal{A}$, then $b$ is prime non-real. 
\end{conjecture}

\subsection{An example of limit $g$-vector}

In the case of Grassmannian cluster algebras, in Section 7 of \cite{CDFL}, it is shown that given any tableau, one can recover its $g$-vector as follows. Any tableau $T \in \SSYT(k, [n])$ can be written uniquely as $S_1^{e_1} \cup \cdots \cup S_m^{e_m}$ for some integers $e_1, \ldots, e_m \in \ZZ$, where $S_1, \ldots, S_m$ are the tableaux in the initial cluster (we choose an order of the initial cluster variables). The vector $(e_1, \ldots, e_m)$ is the $g$-vector of $T$. 

We explain an example of limit $g$-vector in the case of $\CC[\Gr(4,8)]$. We fix the order 
\begin{align*}
& [1, 2, 3, 5], [1, 2, 4, 5], [1, 3, 4, 5], [1, 2, 3, 6], [1, 2, 5, 6], [1, 4, 5, 6], \\
& [1, 2, 3, 7], [1, 2, 6, 7], [1, 5, 6, 7], [1, 2, 3, 4], [2, 3, 4, 5], [3, 4, 5, 6], \\
& [4, 5, 6, 7], [5, 6, 7, 8], [1, 2, 3, 8], [1, 2, 7, 8], [1, 6, 7, 8]
\end{align*}
of the initial cluster, where each list corresponds to a Pl\"{u}cker coordinate. Mutate at the vertices in Figure \ref{fig:Gr48-limit-ray} alternatively where the tableaux $\scalemath{0.6}{ \begin{ytableau} 4 \\ 6 \\ 7 \\ 8 \end{ytableau} }$ and $\scalemath{0.6}{ \begin{ytableau} 1 & 3 \\ 2 & 4 \\ 4 & 7 \\ 6 & 8 \end{ytableau} }$ resides, we obtain the following $g$-vectors 
\begin{align*}
& (-1, 1, 0, 1, -1, 0, 0, 0, 0, 0, 0, 1, 0, 0, 0, 1, 0),\\ 
& (-2, 2, 0, 2, -1, -1, 0, -1, 1, 0, 0, 2, 0, 0, 0, 2, 0),\\ 
& (-3, 3, 0, 3, -1, -2, 0, -2, 2, 0, 0, 3, 0, 0, 0, 3, 0),\\ 
& (-4, 4, 0, 4, -1, -3, 0, -3, 3, 0, 0, 4, 0, 0, 0, 4, 0), \ldots
\end{align*}
When the mutation step $r$ is large enough, the $g$-vector we obtain is 
\begin{align*}
(-r,r,0,r,-1,-r+1,0,-r+1,r-1,0,0,r,0,0,0,r,0).     
\end{align*}
The limit $g$-vector of the sequence is 
\begin{align*}
(-1, 1, 0, 1, 0, -1, 0, -1, 1, 0, 0, 1, 0, 0, 0, 1, 0).
\end{align*}
This limit $g$-vector corresponds to the prime non-real tableau $\scalemath{0.6}{\begin{ytableau} 1 & 3 \\ 2 & 5 \\ 4 & 7 \\ 6 & 8 \end{ytableau} }$. This non-real tableau correspond to the non-real module $L(Y_{2, 0}Y_{1, -3}Y_{3, -3} Y_{2, -6})$. Up to shift the second indices, this module is the non-real module found in Section 13.6 in \cite{HL10}.

\begin{figure}
\scalebox{0.56}{
\begin{xy} 0;<1pt,0pt>:<0pt,-1pt>:: 
(0,147) *+{\begin{ytableau} 4 \\ 5 \\ 6 \\ 8 \end{ytableau}} ="0",
(86,145) *+{\begin{ytableau} 1 & 4 \\ 2 & 5 \\ 3 & 7 \\ 6 & 8 \end{ytableau}} ="1",
(44,86) *+{\begin{ytableau} 1 \\ 2 \\ 3 \\ 8 \end{ytableau}} ="2",
(692,143) *+{\begin{ytableau} 2 \\ 6 \\ 7 \\ 8 \end{ytableau}} ="3",
(393,146) *+{\begin{ytableau} 1 & 3 \\ 2 & 4 \\ 4 & 7 \\ 6 & 8 \end{ytableau}} ="4",
(221,143) *+{\begin{ytableau} 1 & 4 \\ 2 & 5 \\ 4 & 7 \\ 6 & 8 \end{ytableau}} ="5",
(300,29) *+{\begin{ytableau} 1 \\ 2 \\ 7 \\ 8 \end{ytableau}} ="6",
(583,142) *+{\begin{ytableau} 1 & 3 \\ 2 & 6 \\ 4 & 7 \\ 5 & 8 \end{ytableau}} ="7",
(482,143) *+{\begin{ytableau} 1 & 3 \\ 2 & 6 \\  4 & 7 \\ 6 & 8 \end{ytableau}} ="8",
(391,0) *+{\begin{ytableau} 4 \\ 6 \\ 7 \\ 8 \end{ytableau}} ="9",
(595,65) *+{\begin{ytableau} 1 \\ 6 \\ 7 \\ 8 \end{ytableau}} ="10",
(637,228) *+{\begin{ytableau} 2 \\ 3 \\ 4 \\ 5 \end{ytableau}} ="11",
(259,266) *+{\begin{ytableau} 3 \\ 4 \\ 5 \\ 6 \end{ytableau}} ="12",
(77,236) *+{\begin{ytableau} 4 \\ 5 \\ 6 \\ 7 \end{ytableau}} ="13",
(491,2) *+{\begin{ytableau} 5 \\ 6 \\ 7 \\ 8 \end{ytableau}} ="14",
(437,286) *+{\begin{ytableau} 1 \\ 2 \\ 3 \\ 4 \end{ytableau}} ="15",
(306,125) *+{\begin{ytableau} 3 \\ 4 \\ 7 \\ 8 \end{ytableau}} ="16",
"1", {\ar"0"},
"0", {\ar"2"},
"0", {\ar"13"},
"2", {\ar"1"},
"5", {\ar"1"},
"1", {\ar"6"},
"1", {\ar"15"},
"6", {\ar"2"},
"7", {\ar"3"},
"3", {\ar"10"},
"3", {\ar"11"},
"5", {\ar"4"},
"4", {\ar"6"},
"8", {\ar"4"},
"4", {\ar"9"}, 
"4", {\ar@*{[|(1)]}"4";"9"<3pt>},
"4", {\ar"12"},
"4", {\ar"15"},
"16", {\ar"4"},
"6", {\ar"5"},
"9", {\ar"5"},
"12", {\ar"5"},
"15", {\ar"5"},
"6", {\ar"8"},
"10", {\ar"6"},
"8", {\ar"7"},
"11", {\ar"7"},
"7", {\ar"12"},
"7", {\ar"15"},
"9", {\ar"8"},
"12", {\ar"8"},
"15", {\ar"8"},
"14", {\ar"9"},
"9", {\ar"16"},
"12", {\ar"11"},
"13", {\ar"12"},
"15", {\ar"16"},
\end{xy} }
\caption{Mutate alternatively at the vertices of the double arrow, we obtain a limit $g$-vector $(-1, 1, 0, 1, 0, -1, 0, -1, 1, 0, 0, 1, 0, 0, 0, 1, 0)$. This $g$-vector corresponds to the prime non-real tableau $[[1,2,4,6],[3,5,7,8]]$ with columns $1,2,4,6$ and $3,5,7,8$.}
\label{fig:Gr48-limit-ray}
\end{figure}
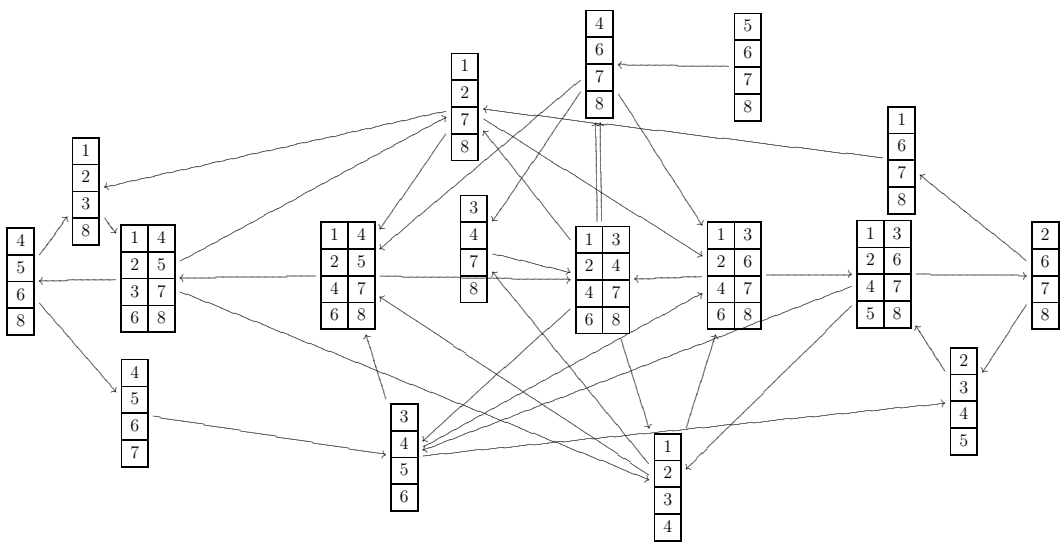

\subsection{Limit $g$-vectors in $\CC[\Gr(3,9)]$ and $\CC[\Gr(4,8)]$}

Recall that we say that a tableau in $\SSYT(k, [n])$ has rank $r$ if the tableau has $r$ columns and an element in $\CC[\Gr(k,n)]$ has rank $r$ if the tableau corresponding to it has $r$ columns. We say that a $g$-vector has rank $r$ is the tableaux corresponding to it has rank $r$. 

We compute cluster variables and limit $g$-vectors for $\CC[\Gr(k,n)]$ in terms of tableaux up to certain numbers of columns. That is, at each step of mutation, if we obtain some tableau with columns more than some number $m$, we mutate this vertex again. In this way, we can collect only tableaux (cluster variables) with columns less or equal to $m$. 

We compute limit $g$-vectors for $\CC[\Gr(3,9)]$ and $\CC[\Gr(4,8)]$ up to rank $56$ (the corresponding tableaux have at most $56$ columns). The sequence of numbers of rank $r$ ($r \ge 1$) limit $g$-vectors for $\CC[\Gr(3,9)]$ is 
\begin{align*}
& 0,0,3,0,0,3,0,0,6,0,0,6,0,0,12,0,0,6,0,0,18,0,0,12,0,0,12,0,0,12, \\
& 0,0,18,0,0,12,0,0,30,0,0,12,0,0,36,0,0,18,0,0,24,0,0,24, \ldots
\end{align*}
The sequence of numbers of rank $r$ ($r \ge 1$) limit $g$-vectors for $\CC[\Gr(4,8)]$ is 
\begin{align*}
& 0,2,0,2,0,4,0,4,0,8,0,4,0,12,0,8,0,12,0,8,0,12,0,8,0, \\
& 20,0,8,0,24,0,12,0,16,0,16,0,32,0,12,0,36,0,16,0,24,0,20,0,44, \ldots
\end{align*}

Based on the computations, we have the following interesting conjecture. Denote by $\phi(m)$ the Euler totient function which counts the numbers less or equal to $m$ and prime to $m$.

\begin{conjecture}
The number of rank $r$ ($r \ge 1$) prime non-real tableaux in $\SSYT(3, [9])$ corresponding to limit $g$-vectors is 
\begin{align*}
f_{3,9,r} = 
\begin{cases}
0, & r \pmod 3 = i, \ i \in \{1,2\}, \\
3 \phi(r/3), & r \pmod 3 = 0.
\end{cases}
\end{align*}
Any prime non-real tableau in $\SSYT(3, [9])$ is of the form $T^{(1)} \cup \cdots \cup T^{(m)}$ for some $m \in \ZZ_{\ge 1}$ and some $T^{(j)}$'s corresponding to limit $g$-vectors. 

The number of rank $r$ ($r \ge 1$) prime non-real tableaux in $\SSYT(4, [8])$ corresponding to limit $g$-vectors is 
\begin{align*}
f_{4,8,r} = 
\begin{cases}
0, & r \pmod 2 = 1, \\
2 \phi(r/2), & r \pmod 2 = 0.
\end{cases}
\end{align*}
Any prime non-real tableau in $\SSYT(4, [8])$ is of the form $T^{(1)} \cup \cdots \cup T^{(m)}$ for some $m \in \ZZ_{\ge 1}$ and some $T^{(j)}$'s corresponding to limit $g$-vectors. 
\end{conjecture}

\begin{remark}
Let $T$ be the tableau with columns $[[1,2,4,6],[3,5,7,8]]$. Then $T$ corresponds to a limit $g$-vector. It is shown in Section IV in \cite{ALS21} that $\ch(T \cup T) = \ch(T)^2-\ch(T')$ for some $T'$ ($T'$ is the union of some frozen tableaux) and $\ch(T \cup T \cup T) = \ch(T)^3 - 2 \ch(T) \ch(T')$. We expect that $\ch(T)^2-\ch(T')$ cannot be factored as $\ch(T'')\ch(T''')$ for any non-trivial $T''$, $T'''$ and $\ch(T \cup T)$ is prime. 

Although $\ch(T \cup T \cup T) = \ch(T)^3 - 2 \ch(T) \ch(T') = \ch(T)( \ch(T)^2 - 2 \ch(T') )$, we still expect that $\ch(T \cup T \cup T)$ cannot be factored as $\ch(T'')\ch(T''')$ for any non-trivial $T''$, $T'''$. Indeed, if $\ch(T \cup T \cup T) = \ch(T)( \ch(T)^2 - 2 \ch(T') ) = \ch(T'')\ch(T''')$, we must have that one of $\ch(T''), \ch(T''')$ is $\ch(T)$ and the other is $\ch(T)^2 - 2 \ch(T')$. Say, $\ch(T''')=\ch(T)^2 - 2 \ch(T') = \ch(T \cup T) - \ch(T')$. Since $T'$ is the product of some frozen tableaux, by comparing the $g$-vectors of both sides, we have that $T''' = T \cup T$. This implies that $\ch(T')=0$ which is a contradiction. Therefore we expect that $\ch(T \cup T \cup T)$ is also prime. 
\end{remark}

\subsection{Limit $g$-vectors for $\Gr(4,9)$}

Using the algorithm in 
Theorem 1.1 in \cite{BL}, we find that there are $18$ two-column prime non-real tableaux and $252$ three-column prime non-real tableaux in $\SSYT(4,[9])$. We also checked that by computer that they all correspond to limit $g$-vectors. 

The $18$ prime non-real tableaux are obtained from the two prime non-real tableaux $\scalemath{0.6}{\begin{ytableau} 1 & 3 \\ 2 & 5 \\ 4 & 7 \\ 6 & 8 \end{ytableau} }$,  $\scalemath{0.6}{\begin{ytableau} 1 & 2 \\ 3 & 4 \\ 5 & 6 \\ 7 & 8 \end{ytableau} }$ by replacing $1<2<\cdots <8$ by $a_1< \cdots < a_8 \in [9]$. 

Up to promotion, the $252$ three-column prime non-real tableaux in $\SSYT(4,[9])$ are
\begin{align*}
& 
\scalemath{0.66}{
\begin{ytableau}
1 & 2 & 3 \\
2 & 5 & 6 \\
4 & 7 & 8 \\
7 & 8 & 9
\end{ytableau}, \ \begin{ytableau}
1 & 1 & 2 \\
3 & 3 & 4 \\
5 & 6 & 7 \\
8 & 9 & 9
\end{ytableau}, \ \begin{ytableau}
1 & 1 & 3 \\
2 & 4 & 5 \\
3 & 6 & 7 \\
6 & 8 & 9
\end{ytableau}, \ \begin{ytableau}
1 & 1 & 4 \\
2 & 3 & 6 \\
4 & 5 & 8 \\
7 & 8 & 9
\end{ytableau}, \ \begin{ytableau}
1 & 2 & 3 \\
2 & 5 & 6 \\
4 & 7 & 8 \\
7 & 9 & 9
\end{ytableau}, \ \begin{ytableau}
1 & 1 & 4 \\
2 & 3 & 6 \\
4 & 5 & 8 \\
7 & 9 & 9
\end{ytableau}, \ \begin{ytableau}
1 & 2 & 2 \\
3 & 4 & 5 \\
6 & 6 & 7 \\
8 & 8 & 9
\end{ytableau},  \ \begin{ytableau}
1 & 1 & 2 \\
3 & 4 & 5 \\
6 & 7 & 8 \\
8 & 9 & 9
\end{ytableau}, \
\begin{ytableau}
1 & 1 & 3 \\
2 & 5 & 6 \\
4 & 7 & 8 \\
6 & 8 & 9
\end{ytableau}, \ \begin{ytableau}
1 & 2 & 4 \\
3 & 4 & 6 \\
5 & 5 & 8 \\
7 & 8 & 9
\end{ytableau}, } \\
& \scalemath{0.66}{ \begin{ytableau}
1 & 3 & 4 \\
2 & 6 & 6 \\
5 & 7 & 8 \\
7 & 8 & 9
\end{ytableau}, \ \begin{ytableau}
1 & 2 & 2 \\
3 & 4 & 5 \\
6 & 6 & 7 \\
8 & 9 & 9
\end{ytableau}, \ \begin{ytableau}
1 & 1 & 2 \\
3 & 4 & 5 \\
5 & 6 & 6 \\
7 & 8 & 9
\end{ytableau}, \ \begin{ytableau}
1 & 2 & 3 \\
4 & 4 & 5 \\
5 & 6 & 7 \\
8 & 9 & 9
\end{ytableau}, \ \begin{ytableau}
1 & 1 & 3 \\
2 & 5 & 6 \\
4 & 7 & 8 \\
6 & 9 & 9
\end{ytableau}, \ \begin{ytableau}
1 & 2 & 2 \\
3 & 3 & 4 \\
5 & 6 & 7 \\
8 & 8 & 9
\end{ytableau}, \ \begin{ytableau}
1 & 2 & 4 \\
3 & 4 & 6 \\
5 & 5 & 8 \\
7 & 9 & 9
\end{ytableau}, \ \begin{ytableau}
1 & 2 & 2 \\
3 & 3 & 4 \\
5 & 6 & 7 \\
8 & 9 & 9
\end{ytableau}, \ \begin{ytableau}
1 & 1 & 3 \\
2 & 4 & 5 \\
3 & 7 & 8 \\
6 & 9 & 9
\end{ytableau}, \ \begin{ytableau}
1 & 2 & 4 \\
2 & 3 & 6 \\
4 & 5 & 8 \\
7 & 8 & 9
\end{ytableau}, } \\
& \scalemath{0.66}{ \begin{ytableau}
1 & 2 & 3 \\
4 & 5 & 6 \\
6 & 7 & 8 \\
8 & 9 & 9
\end{ytableau}, \ \begin{ytableau}
1 & 1 & 2 \\
3 & 3 & 4 \\
5 & 6 & 7 \\
7 & 8 & 9
\end{ytableau}, \ \begin{ytableau}
1 & 2 & 2 \\
3 & 4 & 5 \\
6 & 7 & 8 \\
8 & 9 & 9
\end{ytableau}, \ \begin{ytableau}
1 & 2 & 2 \\
3 & 4 & 5 \\
5 & 6 & 6 \\
7 & 8 & 9
\end{ytableau}, \ \begin{ytableau}
1 & 2 & 3 \\
2 & 4 & 5 \\
3 & 7 & 8 \\
6 & 9 & 9
\end{ytableau}, \ \begin{ytableau}
1 & 2 & 4 \\
3 & 3 & 5 \\
4 & 5 & 8 \\
6 & 7 & 9
\end{ytableau}, \ \begin{ytableau}
1 & 3 & 3 \\
2 & 5 & 6 \\
4 & 7 & 8 \\
6 & 9 & 9
\end{ytableau}, \ \begin{ytableau}
1 & 1 & 3 \\
2 & 3 & 6 \\
4 & 5 & 8 \\
7 & 8 & 9
\end{ytableau}. }
\end{align*}

On the other hand, it is stated in \cite[Section 5.3]{HP21} that there are $g$-vectors (they call them rays) that are neither cluster variables nor limit $g$-vectors (limit $g$-vectors are called limit rays in \cite{HP21}). This means that there are prime non-real modules which cannot be obtained by computing limit $g$-vectors. 
	
\section{Limit $g$-vectors for Type $D_n$ Quantum Affine Algebras} \label{sec:limit g vectors type Dn}

In a recent paper \cite{BC22}, Brito and Chari found for the first time some examples of non-real $U_q(\widehat{\mathfrak{g}})$-modules when $\mathfrak{g}$ is of type $D_4$ which do not arise from a type $A_n$ module. One non-real module they found is $L(Y_{1, -11}Y_{1,-9}Y_{2,-6}Y_{1,-3}Y_{1,-1})$ (up to shifting the indices $s$ in $Y_{i,s}$). In Section \ref{sec:limit g vector}, we see that limit $g$-vectors can be used to produce non-real modules.  In this section, we construct one explicit example of non-real module of type $D_4$ using a limit $g$-vector.  

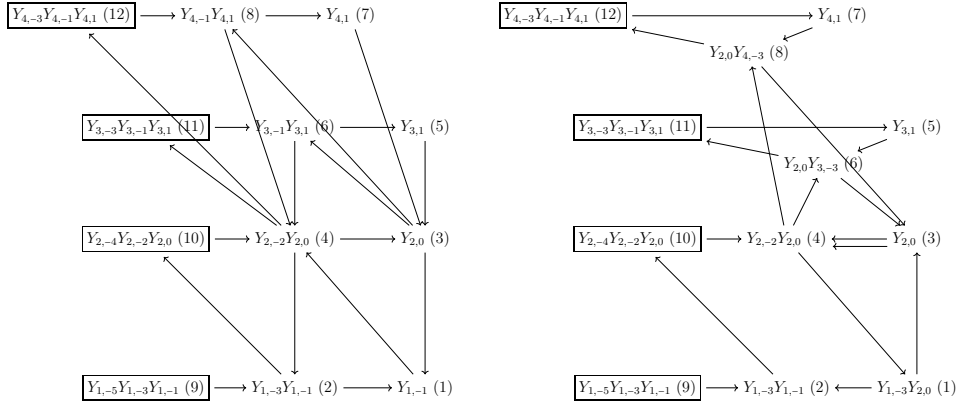
\begin{figure}
\begin{minipage}[t]{0.4\textwidth}
\scalebox{0.5}{
\begin{tikzpicture}[scale=0.99]
		\node  (1) at (-2, -2) {\fbox{$Y_{4,-3}Y_{4,-1}Y_{4,1}$ (12)}};
		\node  (2) at (2,-2) {$Y_{4,-1}Y_{4,1}$ (8)}; 
		\node  (3) at (5.5,-2) {$Y_{4,1}$ (7)}; 
		\node  (4) at (0, -5) {\fbox{$Y_{3,-3}Y_{3,-1}Y_{3,1}$ (11)}};
		\node  (5) at (4, -5) {$Y_{3,-1}Y_{3,1}$ (6)};
		\node  (6) at (7.5, -5) {$Y_{3,1}$ (5)};
		\node  (7) at (0, -8) {\fbox{$Y_{2,-4}Y_{2,-2}Y_{2,0}$ (10)}};
		\node  (8) at (4, -8) {$Y_{2,-2}Y_{2,0}$ (4)};
		\node  (9) at (7.5, -8) {$Y_{2,0}$ (3)};
		\node  (10) at (0, -12) {\fbox{$Y_{1,-5}Y_{1,-3}Y_{1,-1}$ (9)}};
		\node  (11) at (4, -12) {$Y_{1,-3}Y_{1,-1}$ (2)};
		\node  (12) at (7.5, -12) {$Y_{1,-1}$ (1)};
     
     \draw[->] (1)--(2); 
     \draw[->] (2)--(3);
      \draw[->] (4)--(5); 
     \draw[->] (5)--(6);
      \draw[->] (7)--(8); 
     \draw[->] (8)--(9);
      \draw[->] (10)--(11); 
     \draw[->] (11)--(12);
     \draw[<-] (12)--(9);  
     \draw[<-] (9)--(6);
     \draw[<-] (9)--(3);
     \draw[<-] (8)--(12);
     \draw[<-] (5)--(9);
     \draw[<-] (2)--(9);
     \draw[<-] (11)--(8);
       \draw[<-] (8)--(5);
     \draw[<-] (8)--(2);
     \draw[<-] (7)--(11);
     \draw[<-] (4)--(8);
     \draw[<-] (1)--(8);
\end{tikzpicture} }
\end{minipage}
\begin{minipage}[t]{0.4\textwidth}
\scalebox{0.5}{
\begin{tikzpicture}[scale=0.99]
		\node  (12) at (-2, -2) {\fbox{$Y_{4,-3}Y_{4,-1}Y_{4,1}$ (12)}};
		\node  (11) at (3,-3) {$Y_{2,0}Y_{4,-3}$ (8)}; 
		\node  (10) at (5.5,-2) {$Y_{4,1}$ (7)}; 
		\node  (9) at (0, -5) {\fbox{$Y_{3,-3}Y_{3,-1}Y_{3,1}$ (11)}};
		\node  (8) at (5, -6) {$Y_{2,0}Y_{3,-3}$ (6)};
		\node  (7) at (7.5, -5) {$Y_{3,1}$ (5)};
		\node  (6) at (0, -8) {\fbox{$Y_{2,-4}Y_{2,-2}Y_{2,0}$ (10)}};
		\node  (5) at (4, -8) {$Y_{2,-2}Y_{2,0}$ (4)};

\node  (5s) at (5.1, -8.2) {};
  
  \node  (4s) at (6.85, -8.2) {};
		\node  (4) at (7.5, -8) {$Y_{2,0}$ (3)};
		\node  (3) at (0, -12) {\fbox{$Y_{1,-5}Y_{1,-3}Y_{1,-1}$ (9)}};
		\node  (2) at (4, -12) {$Y_{1,-3}Y_{1,-1}$ (2)};
		\node  (1) at (7.5, -12) {$Y_{1,-3}Y_{2,0}$ (1)};
     
     \draw[->] (1)--(2); 
     \draw[->] (1)--(4);
      \draw[->] (2)--(6); 
     \draw[->] (3)--(2);
     \draw[->] (4)--(5);
     \draw[->] (4s)--(5s);
      \draw[->] (5)--(1);
      \draw[->] (5)--(8);
      \draw[->] (5)--(11);
      \draw[->] (6)--(5);
      \draw[->] (7)--(8);
      \draw[->] (8)--(4);
      \draw[->] (8)--(9);
      \draw[->] (9)--(7);
      \draw[->] (10)--(11);
      \draw[->] (11)--(12);
      \draw[->] (11)--(4);
      \draw[->] (12)--(10);
\end{tikzpicture} }
\end{minipage}
            \caption{The left hand side is an initial cluster of $K_0(\mathcal{C}_2^{D_4})$. The numbers in the brackets are labels of the vertices. The right hand side is the quiver obtained from the initial quiver by mutating at the vertices $1, 6, 8$. The arrow from vertex $3$ to $4$ is a double arrow.}
            \label{fig:initial cluster D4 and l is 2}
\end{figure} 

\subsection{The type $D_4$ module $L(Y_{2,-4}Y_{2,0})$ corresponds to a limit $g$-vector} 

Let $\mathfrak{g}$ be of type $D_4$ and $\ell = 2$. An initial cluster for the cluster algebra $K_0(\mathcal{C}_{2}^{D_4})$ is shown in Figure \ref{fig:initial cluster D4 and l is 2}. 

The module $L(Y_{2,-4}Y_{2,0})$ in type $D_4$ corresponds to a limit $g$-vector. This limit $g$-vector is obtained as follows. We label the vertices of the initial quiver as shown in Figure \ref{fig:initial cluster D4 and l is 2}. After mutations at the vertices $1, 6, 8$, we obtain a quiver with a double arrow from the vertex $3$ to the vertex $4$, see the quiver on the right hand side in Figure \ref{fig:initial cluster D4 and l is 2}. 

We take the order of initial cluster variables as:
\begin{align*}
& L(Y_{1, -1}), \  L(Y_{1, -3} Y_{1, -1}), \  L(Y_{2, 0}), \  L(Y_{2, -2} Y_{2, 0}), \ L(Y_{3, 1}), \ L(Y_{3, -1} Y_{3, 1}), \ L(Y_{4, 1}), \\ 
& L(Y_{4, -1} Y_{4, 1}),  \ L(Y_{1, -5} Y_{1, -3} Y_{1, -1}), \ L(Y_{2, -4} Y_{2, -2} Y_{2, 0}), \  L(Y_{3, -3} Y_{3, -1} Y_{3, 1}), \ L(Y_{4, -3} Y_{4, -1} Y_{4, 1}).
\end{align*}
Now we mutate the vertices $3, 4$ alternatively and obtain the $g$-vectors:
\begin{align*}
& (0, 0, 1, 0, 0, 0, 0, 0, 0, 0, 0, 0), \quad (0, 0, 0, 1, 0, 0, 0, 0, 0, 0, 0, 0), \\
& (-1, 1, 2, 0, 0, -1, 0, -1, 0, 0, 1, 1), \quad
(-1, 1, 3, -1, 0, -1, 0, -1, 0, 1, 1, 1), \\
& (-1, 1, 4, -2, 0, -1, 0, -1, 0, 2, 1, 1), \quad
(-1, 1, r+2, -r, 0, -1, 0, -1, 0, r, 1, 1), \quad r \ge 3. 
\end{align*}
The limit $g$-vector for this mutation sequence is $(0,0,1,-1,0,0,0,0,0,1,0,0)$. This is the $g$-vector of the module $L(Y_{2,-4}Y_{2,0})$. 

We will verify that the module $L(Y_{2,-4}Y_{2,0})$ in type $D_4$ is prime non-real by using $(q,t)$-characters \cite{Nak03, Nak04} (see also \cite{HL15, Bit21}) below. Note that the module $L(Y_{2,-4}Y_{2,0})$ in type $A_n$ ($n \ge 2$) is real (it is a snake module \cite{DLL19}). 

We first recall the results of \cite{Nak03, Nak04} about $(q,t)$-characters. 

\subsection{$(q,t)$-characters}

Let $C(z)$ be the quantum Cartan matrix of $\mathfrak{g}$ \cite{FR98} and let $\widetilde{C}(z) = (\widetilde{C}_{ij}(z))$ be the inverse of $C(z)$, see Section \ref{sec:quantum affine algebras}. The entries of $(\widetilde{C}_{ij}(z))$ have power series expressions in $z$ of the form \cite{HL15} $\widetilde{C}_{ij}(z) = \sum_{m \ge 1} \widetilde{C}_{ij}(m) z^m$. Nakajima \cite{Nak03, Nak04} introduced $(q,t)$-characters of $U_q(\widehat{\mathfrak{g}})$-modules which are $t$-deformations of $q$-characters. Let $K_t(\mathcal{C}_\ell)$ be the $t$-deformation of the Grothendieck ring $K_0(\mathcal{C}_\ell)$ \cite{HL15, Bit21}. Let $\hat{I}$ be the set of vertices of the initial quiver of the cluster algebra $K_0(\mathcal{C}_\ell)$. Denote by ${\bf Y}_t$ the $\ZZ[t^{\pm 1}]$-algebra generated by $Y_{i,p}^{\pm 1}$, $(i,p) \in \hat{I}$, subject to the relations (\cite{Nak03, Nak04}, see also \cite{Bit21, HL15}):
\begin{align*}
    Y_{i,p}*Y_{j,s}=t^{N(i,p;j,s)}Y_{j,s}*Y_{i,p},
\end{align*}
where we use Nakajima's convention \cite{Nak03, Nak04} (in type ADE, $d_i=1$):
\begin{align*}
N(i,p;j,s) = 2(\widetilde{C}_{ij}(s-p-d_i)-\widetilde{C}_{ij}(p-s-d_i)).
\end{align*}
For any family $\{u_{i,p} \in \ZZ : (i,p) \in \hat{I}\}$, denote 
\begin{align*}
    \prod_{(i,p) \in \hat{I}} Y_{i,p}^{u_{i,p}} = t^{-\frac{1}{2} \sum_{(i,p) < (j,s)} u_{i,p}u_{j,s} N(i,p;j,s) } \overrightarrow{*}_{(i,p) \in \hat{I}} Y_{i,p}^{u_{i,p}}. 
\end{align*}
The expression on the right hand side of the above equation does not depend on the order of $Y_{i,p}$'s and so $\prod_{(i,p) \in \hat{I}} Y_{i,p}^{u_{i,p}}$ is well-defined. The monomial $\prod_{(i,p) \in \hat{I}} Y_{i,p}^{u_{i,p}}$ is called a commutative monomial \cite{Nak03, HL15}. For a dominant monomial $m = \prod_{(i,p) \in \hat{I}} Y_{i,p}^{u_{i,p}(m)}$, (\cite{Nak03, Nak04}, see also \cite{HL15, Bit21}) the standard module $M(m)$ is the tensor product of the fundamental modules corresponding to each of the factors in $m$ in a particular order. In this paper, we choose the order as $Y_{i,s} < Y_{j, t}$ if and only if $s < t$. The truncated $(q,t)$-character of $M(m)$ is given by 
\begin{align*}
[M(m)]_t = t^{\alpha(m)} \overrightarrow{*}_{p \in \ZZ} \prod_{i \in I} \tchiqt(L(Y_{i,p}))^{*u_{i,p}(m)},
\end{align*}
where $\alpha(m)$ is the integer such that $m$ occurs with multiplicity one in the expansion of $[M(m)]_t$ on the basis of the commutative monomials of ${\bf Y}_t$ and the product $\overrightarrow{*}_{p \in \ZZ}$ is taken as increasing order. Since $\tchiqt(L(Y_{i,p}))$ and $\tchiqt(L(Y_{i,p'}))$ commute for any $p,p'$, the above expression is well-defined. 

In \cite{Nak03, Nak04}, a $\ZZ$-algebra anti-automorphism of ${\bf Y}_t$ called bar-involution is defined by: $t \mapsto t^{-1}$, $Y_{i,p} \mapsto Y_{i,p}$, $(i,p) \in \hat{I}$. 

For a simple module $L(m)$, denote by $[L(m)]_t$ its $(q,t)$-character. The following theorem by Nakajima \cite{Nak03, Nak04} gives an algorithm to compute (truncated) $(q,t)$-characters of a simple $U_q(\widehat{\mathfrak{g}})$-module: for every dominant monomial $m \in \mathcal{P}_{\ell}^+$, there is a unique element $[L(m)]_t$ of $K_t(\mathcal{C}_{\ell})$ such that 
\begin{itemize}
    \item $\overline{[L(m)]_t} = [L(m)]_t$, 
    
    \item $[L(m)]_t \in [M(m)]_t + \sum_{m' < m} t^{-1} \ZZ[t^{-1}] [M(m')]_t$. 
\end{itemize}
This result is generalized to non-simply-laced types in \cite{He04, FHOO22}.

\subsection{The Type $D_4$ Module $L(Y_{2,-4}Y_{2,0})$ is Prime Non-real} \label{subsec:D4 nonreal example}

The quantum Cartan matrix in type $D_4$ is $\scalemath{0.7}{\left( \begin {array}{cccc} {\frac {{z}^{2}+1}{z}}&-1&0&0
\\ \noalign{\medskip}-1&{\frac {{z}^{2}+1}{z}}&-1&-1
\\ \noalign{\medskip}0&-1&{\frac {{z}^{2}+1}{z}}&0
\\ \noalign{\medskip}0&-1&0&{\frac {{z}^{2}+1}{z}}\end {array} \right) }$.
In the following, we also write $m_1m_2^{-1}$ as $\frac{m_1}{m_2}$ for two dominant monomials $m_1, m_2$. By modified Frenkel-Mukhin algorithm \cite{FM00, Nak03, Nak04}, we have that $\tchiqt(L(Y_{2,0}))=Y_{2,0}$,
\begin{align*}
\scalemath{0.8}{ \tchiqt(Y_{2,-4}) } & \scalemath{0.8}{ = Y_{{2,-4
}}+ {\frac {Y_{{1,-3}}Y_{{3,-3}}Y_{{4,-3}}}{Y_{{2,-2}}}}+(t+\frac{1}{t}){\frac {Y_{{2,-2
}}}{Y_{{2,0}}}}+{\frac {Y_{{1,-1}}Y_{{1,-3}}}{Y_{{2,0}}}}+{\frac {Y_{{
1,-3}}Y_{{3,-3}}}{Y_{{4,-1}}}}+{\frac {Y_{{4,-3}}Y_{{1,-3}}}{Y_{{3,-1}
}}} +{\frac {Y_{{1,-1}}Y_{{3,-1}}Y_{{4,-1}}}{{Y_{{2,0}}}^{2}}} } \\
& \scalemath{0.8}{ +{\frac {Y_{{3,-3}}Y_{{3,-1}}}{Y_{{2,0}}}}+{\frac {Y_{{4,-1}}Y_{{4,-
3}}}{Y_{{2,0}}}} +{\frac {Y_{{4,-3}}Y_{{3,-3}}}{Y_{{1,-1}}}}+{\frac {Y_
{{1,-3}}Y_{{2,-2}}}{Y_{{4,-1}}Y_{{3,-1}}}}+{\frac {Y_{{1,-1}}Y_{{3,-1}
}}{Y_{{2,0}}Y_{{4,1}}}} +{\frac {Y_{{1,-1}}Y_{{4,-1}}}{Y_{{2,0}}Y_{{3,1
}}}}+{\frac {Y_{{3,-3}}}{Y_{{3,1}}}}+{\frac {Y_{{4,-3}}}{Y_{{4,1}}}}  } \\
& \scalemath{0.8}{
+{\frac {Y_{{4,-3}}
Y_{{2,-2}}}{Y_{{1,-1}}Y_{{3,-1}}}} +{\frac {Y_{{1,-1}}}{Y_{{4,1}}Y_{{3,
1}}}} +{\frac {Y_{{2,-2}}}{Y_{{4,-1}}Y_{{4,1}}}} +{\frac {Y_{{2,-2}}}{Y_
{{3,-1}}Y_{{3,1}}}} +{\frac {{Y_{{2,-2}}}^{2}}{Y_{{1,-1}}Y_{{3,-1}}Y_{{
4,-1}}}} +{
\frac {Y_{{2,-2}}Y_{{3,-3}}}{Y_{{1,-1}}Y_{{4,-1}}}}, }
\end{align*}
where the monomials on the right hand side are commutative monomials. 
Therefore $\tchiqt(L(Y_{2,-4}))*\tchiqt(L(Y_{2,0})) =p_1 + t p_2 + t^2 p_3$, where
\begin{align*}
\scalemath{0.8}{ p_1 = {\frac {Y_{{1,-1}}Y_{{3,-1}}Y_{{4,-1}}}{Y_{{2,0}}}}+{\frac {Y_{{1,-1}}
Y_{{3,-1}}}{Y_{{4,1}}}}+{\frac {Y_{{1,-1}}Y_{{4,-1}}}{Y_{{3,1}}}}+{
\frac {Y_{{2,0}}Y_{{1,-1}}}{Y_{{4,1}}Y_{{3,1}}}}+Y_{{2,-2}}, }
\end{align*}
\begin{align*}
\scalemath{0.8}{ p_2 = Y_{{1,-1}}Y_{{1,-3}}+Y_{{3,-1}}Y_{{3,-3}}+Y_{{4,-1}}Y_{{4,-3}}+{\frac 
{Y_{{3,-3}}Y_{{2,0}}}{Y_{{3,1}}}}+{\frac {Y_{{4,-3}}Y_{{2,0}}}{Y_{{4,1
}}}}+{\frac {Y_{{2,0}}Y_{{2,-2}}}{Y_{{4,-1}}Y_{{4,1}}}}+{\frac {Y_{{2,0
}}Y_{{2,-2}}}{Y_{{3,-1}}Y_{{3,1}}}}, }
\end{align*}
\begin{align*}
\scalemath{0.8}{ p_3 } & \scalemath{0.8}{ = {\frac {Y_{{1,-3}}Y_{{3,-3}}Y_{{4,-3}}Y_{{2,0}}}{Y_{{2,-2}}}}+Y_{{2,-2
}}+{\frac {Y_{{1,-3}}Y_{{3,-3}}Y_{{2,0}}}{Y_{{4,-1}}}}+{\frac {Y_{{1,-
3}}Y_{{4,-3}}Y_{{2,0}}}{Y_{{3,-1}}}}+Y_{{2,-4}}Y_{{2,0}}+{\frac {Y_{{3
,-3}}Y_{{4,-3}}Y_{{2,0}}}{Y_{{1,-1}}}}+{\frac {Y_{{2,-2}}Y_{{1,-3}}Y_{
{2,0}}}{Y_{{3,-1}}Y_{{4,-1}}}} } \\
& \quad \scalemath{0.8}{ +{\frac {Y_{{2,-2}}Y_{{3,-3}}Y_{{2,0}}}{
Y_{{1,-1}}Y_{{4,-1}}}}+{\frac {Y_{{2,-2}}Y_{{4,-3}}Y_{{2,0}}}{Y_{{1,-1
}}Y_{{3,-1}}}}+{\frac {{Y_{{2,-2}}}^{2}Y_{{2,0}}}{Y_{{1,-1}}Y_{{3,-1}}
Y_{{4,-1}}}}. }
\end{align*}
It follows that $\tchiqt(L(Y_{2,-4}Y_{2,0})) = p_3$. Since $\tchi(L(Y_{2,-4}Y_{2,0})) \ne \tchi(L(Y_{2,-4}))\tchi(L(Y_{2,0}))$, we have that $L(Y_{2,-4}Y_{2,0})$ is prime. 

By computing $\tchiqt(L(Y_{2,-4}Y_{2,0})) * \tchiqt(L(Y_{2,-4}Y_{2,0}))$, we found that the dominant monomials in $\tchiqt(L(Y_{2,-4}Y_{2,0})) * \tchiqt(L(Y_{2,-4}Y_{2,0}))$ are:
\begin{align} \label{eq:dominant monomials in the product}
Y_{2,-4}^2Y_{2,0}^2, \ Y_{1, -3} Y_{2, 0} Y_{3, -3} Y_{4, -3}, \ Y_{2,-2}^2, \ 2 Y_{2,-4} Y_{2,-2} Y_{2,0},
\end{align}
where $2 Y_{2,-4} Y_{2,-2} Y_{2,0}$ means that the monomial $Y_{2,-4} Y_{2,-2} Y_{2,0}$ appears two times. 
Therefore in the decomposition 
\begin{align} \label{eq:decomposition of 2m420 times 2m420 with mi}
\tchiqt(L(Y_{2,-4}Y_{2,0})) * \tchiqt(L(Y_{2,-4}Y_{2,0})) = \sum_{i} f_i(t) \tchiqt( L(m_i) ),
\end{align}
where $f_i(t)$ is a polynomial in $t$, we have that every $m_i$ can only be chosen from the monomials in (\ref{eq:dominant monomials in the product}).  

By computing $\tchiqt(L(Y_{1,-3})) * \tchiqt(L(Y_{3,-3})) * \tchiqt(L(Y_{4,-3})) * \tchiqt(L(Y_{2,0}))$, we obtain that 
\begin{align*}
& \scalemath{0.8}{ \tchiqt(L(Y_{1, -3} Y_{2, 0} Y_{3, -3} Y_{4, -3})) } = \scalemath{0.8}{ Y_{{1,-3}}Y_{{2,0}}Y_{{3,-3}}Y_{{4,-3}}+(t+\frac{1}{t}) {Y_{{2,-2}}}^{2}+Y_{{1,-3}}Y_
{{1,-1}}Y_{{2,-2}} +{\frac {Y_{{1,-3}}Y_{{2,-2}}Y_{{2,0}}Y_{{3,-3}}}{Y_
{{4,-1}}}} } \\
& \qquad \scalemath{0.8}{ +{\frac {Y_{{1,-3}}Y_{{2,-2}}Y_{{2,0}}Y_{{4,-3}}}{Y_{{3,-1}}
}}+{\frac {Y_{{1,-1}}Y_{{2,-2}}Y_{{3,-1}}Y_{{4,-1}}}{Y_{{2,0}}}}+Y_{{2
,-2}}Y_{{3,-3}}Y_{{3,-1}} +Y_{{2,-2}}Y_{{4,-3}}Y_{{4,-1}} +{\frac {Y_{{2
,-2}}Y_{{2,0}}Y_{{3,-3}}Y_{{4,-3}}}{Y_{{1,-1}}}} } \\
& \qquad \scalemath{0.8}{ +{\frac {Y_{{1,-3}}{Y_
{{2,-2}}}^{2}Y_{{2,0}}}{Y_{{3,-1}}Y_{{4,-1}}}} +{\frac {Y_{{1,-1}}Y_{{2
,-2}}Y_{{3,-1}}}{Y_{{4,1}}}}+{\frac {Y_{{1,-1}}Y_{{2,-2}}Y_{{4,-1}}}{Y
_{{3,1}}}} +{\frac {Y_{{2,-2}}Y_{{2,0}}Y_{{3,-3}}}{Y_{{3,1}}}}+{\frac {
Y_{{2,-2}}Y_{{2,0}}Y_{{4,-3}}}{Y_{{4,1}}}}+{\frac {{Y_{{2,-2}}}^{2}Y_{
{2,0}}Y_{{3,-3}}}{Y_{{1,-1}}Y_{{4,-1}}}} } \\
&  \qquad \scalemath{0.8}{  +{\frac {{Y_{{2,-2}}}^{2}Y_{{2,0}}}{Y_{{4,
-1}}Y_{{4,1}}}}+{\frac {{Y_{{2,-2}}}^{2}Y_{{2,0}}}{Y_{{3,-1}}Y_{{3,1}}
}}+{\frac {{Y_{{2,-2}}}^{3}Y_{{2,0}}}{Y_{{1,-1}}Y_{{3,-1}}Y_{{4,-1}}}} +{\frac {{Y_{{2,-2}}}^{2}Y_{{2
,0}}Y_{{4,-3}}}{Y_{{1,-1}}Y_{{3,-1}}}}+{\frac {Y_{{1,-1}}Y_{{2,-2}}Y_{
{2,0}}}{Y_{{4,1}}Y_{{3,1}}}}. }
\end{align*}
We checked that there is a monomial ${\frac {Y_{{1,-1}}Y_{{2,-2}}Y_{{3,-1}}}{Y_{{4,1}}}}$ appearing in $\tchiqt(L(Y_{1, -3} Y_{2, 0} Y_{3, -3} Y_{4, -3}))$ but not appearing in $\tchiqt(L(Y_{2,-4}Y_{2,0})) * \tchiqt(L(Y_{2,-4}Y_{2,0}))$. Therefore any $m_i$ on the right hand side of (\ref{eq:decomposition of 2m420 times 2m420 with mi}) cannot be $Y_{1, -3} Y_{2, 0} Y_{3, -3} Y_{4, -3}$. Similarly, any $m_i$ on the right hand side of (\ref{eq:decomposition of 2m420 times 2m420 with mi}) cannot be $Y_{2,-2}^2$. 
Therefore the only possible dominant monomial appearing on the right hand side of (\ref{eq:decomposition of 2m420 times 2m420 with mi}) are $Y_{2,-4}^2Y_{2,0}^2$ and $Y_{2,-4}Y_{2,-2}Y_{2,0}$. 

By computing $f=\tchiqt(L(Y_{2,-4})) * \tchiqt(L(Y_{2,-4})) * \tchiqt(L(Y_{2,0})) * \tchiqt(L(Y_{2,0}))$ and checking the coefficient of $\frac{1}{t^8}f$, we find that the monomial $Y_{2,-4}Y_{2,-2}Y_{2,0}$ appears in the truncated $(q,t)$-character of $L(Y_{2,-4}^2Y_{2,0}^2)$ exactly one time. Since the monomial $Y_{2,-4}Y_{2,-2}Y_{2,0}$ appears two times in $\tchiqt(Y_{2,-4}Y_{2,0}) * \tchiqt(Y_{2,-4}Y_{2,0})$, we have that 
\begin{align*}
& \tchiqt(Y_{2,-4}Y_{2,0}) * \tchiqt(Y_{2,-4}Y_{2,0}) = \tchiqt(L(Y_{2,-4}^2Y_{2,0}^2)) + \tchiqt( L(Y_{2,-4}Y_{2,-2}Y_{2,0}) ),
\end{align*} 
and $\tchiqt(L(Y_{2,-4}^2Y_{2,0}^2)) = p_1 + (t+ \frac{1}{t})p_2 + (t^2 + \frac{1}{t^2})p_3 + (t^3+\frac{1}{t^3})p_4$, where
\begin{align*}
& \scalemath{0.8}{ p_1 = {Y_{{2,-4}}}^{2}{Y_{{2,0}}}^{2
} + {\frac {{Y_{{1,-3}}}^{2}{Y_{{4,-3}}}^{2}{Y_{{2,0}}}^{2}}{{Y_{{3,-1}}}^
{2}}}+2\,{\frac {Y_{{1,-3}}{Y_{{4,-3}}}^{2}Y_{{3,-3}}{Y_{{2,0}}}^{2}}{
Y_{{1,-1}}Y_{{3,-1}}}}+2\,{\frac {{Y_{{1,-3}}}^{2}Y_{{3,-3}}Y_{{4,-3}}
{Y_{{2,0}}}^{2}}{Y_{{4,-1}}Y_{{3,-1}}}}  +{\frac {{Y_{{2,-2}}}^{2}{Y_{{4,-3}}}^{2}{Y_{{2,0}}}^{2}}{{Y_{{1,-1}}
}^{2}{Y_{{3,-1}}}^{2}}} } \\
& \qquad \scalemath{0.8}{ +{\frac {{Y_{{2,-2}}}^{4}{Y_{{2,0}}}^{2}}{{Y_{{
1,-1}}}^{2}{Y_{{3,-1}}}^{2}{Y_{{4,-1}}}^{2}}}+{\frac {{Y_{{2,-2}}}^{2}
{Y_{{3,-3}}}^{2}{Y_{{2,0}}}^{2}}{{Y_{{1,-1}}}^{2}{Y_{{4,-1}}}^{2}}}  +{
\frac {{Y_{{1,-3}}}^{2}{Y_{{3,-3}}}^{2}{Y_{{4,-3}}}^{2}{Y_{{2,0}}}^{2}
}{{Y_{{2,-2}}}^{2}}}   +{\frac {{Y_{{1,-3}}}^{2}{Y_{{2,-2}}}^{2}{Y_{{2,0}
}}^{2}}{{Y_{{3,-1}}}^{2}{Y_{{4,-1}}}^{2}}}+2\,{\frac {{Y_{{2,-2}}}^{2}
Y_{{1,-3}}Y_{{3,-3}}{Y_{{2,0}}}^{2}}{Y_{{1,-1}}Y_{{3,-1}}{Y_{{4,-1}}}^
{2}}} } \\
& \qquad \scalemath{0.8}{ +2\,{\frac {{Y_{{2,-2}}}^{2}Y_{{1,-3}}Y_{{4,-3}}{Y_{{2,0}}}^{2}}{
{Y_{{3,-1}}}^{2}Y_{{4,-1}}Y_{{1,-1}}}}+2\,{\frac {Y_{{1,-3}}{Y_{{3,-3}
}}^{2}Y_{{4,-3}}{Y_{{2,0}}}^{2}}{Y_{{1,-1}}Y_{{4,-1}}}}  +{\frac {{Y_{{1
,-3}}}^{2}{Y_{{3,-3}}}^{2}{Y_{{2,0}}}^{2}}{{Y_{{4,-1}}}^{2}}}+{\frac {
{Y_{{3,-3}}}^{2}{Y_{{4,-3}}}^{2}{Y_{{2,0}}}^{2}}{{Y_{{1,-1}}}^{2}}}+{Y
_{{2,-2}}}^{2} }\\
& \qquad \scalemath{0.8}{  +2\,{\frac {Y_{{3,-3}}Y_{{4,-3}}{Y_{{2,-2}}}^{2}{Y_{{2,0
}}}^{2}}{Y_{{4,-1}}Y_{{3,-1}}{Y_{{1,-1}}}^{2}}}+ Y_{{2,-2}}Y_{{2,-4}
}Y_{{2,0}}, }
\end{align*}
\begin{align*}
& \scalemath{0.8}{  p_2 = {\frac {{Y_{{2,-2}}}^{2}Y_{{4,-3}}Y_{{2,0}}}{Y_{{1,-1}}Y_{{3,-1}}}}+{
\frac {{Y_{{2,-2}}}^{3}Y_{{2,0}}}{Y_{{1,-1}}Y_{{3,-1}}Y_{{4,-1}}}}+{
\frac {Y_{{2,-2}}Y_{{4,-3}}Y_{{2,-4}}{Y_{{2,0}}}^{2}}{Y_{{1,-1}}Y_{{3,
-1}}}}+{\frac {{Y_{{1,-3}}}^{2}Y_{{3,-3}}Y_{{2,-2}}{Y_{{2,0}}}^{2}}{{Y
_{{4,-1}}}^{2}Y_{{3,-1}}}}  +{\frac {{Y_{{3,-3}}}^{2}{Y_{{4,-3}}}^{2}Y_{
{1,-3}}{Y_{{2,0}}}^{2}}{Y_{{1,-1}}Y_{{2,-2}}}} }\\
& \quad \scalemath{0.8}{  +{\frac {{Y_{{2,-2}}}^{3
}Y_{{1,-3}}{Y_{{2,0}}}^{2}}{Y_{{1,-1}}{Y_{{3,-1}}}^{2}{Y_{{4,-1}}}^{2}
}} +{\frac {{Y_{{2,-2}}}^{3}Y_{{4,-3}}{Y_{{2,0}}}^{2}}{{Y_{{1,-1}}}^{2}
{Y_{{3,-1}}}^{2}Y_{{4,-1}}}} +{\frac {Y_{{2,-2}}Y_{{3,-3}}Y_{{2,-4}}{Y_
{{2,0}}}^{2}}{Y_{{1,-1}}Y_{{4,-1}}}}  +{\frac {{Y_{{1,-3}}}^{2}{Y_{{3,-3
}}}^{2}Y_{{4,-3}}{Y_{{2,0}}}^{2}}{Y_{{4,-1}}Y_{{2,-2}}}}+{\frac {{Y_{{
2,-2}}}^{3}Y_{{3,-3}}{Y_{{2,0}}}^{2}}{{Y_{{1,-1}}}^{2}{Y_{{4,-1}}}^{2}
Y_{{3,-1}}}} }\\
& \quad \scalemath{0.8}{  +{\frac {Y_{{1,-3}}{Y_{{3,-3}}}^{2}Y_{{2,-2}}{Y_{{2,0}}}^{
2}}{{Y_{{4,-1}}}^{2}Y_{{1,-1}}}}+{\frac {Y_{{2,0}}Y_{{1,-3}}Y_{{2,-2}}
Y_{{3,-3}}}{Y_{{4,-1}}}}  +{\frac {{Y_{{2,-2}}}^{2}Y_{{1,-3}}Y_{{2,0}}}{
Y_{{4,-1}}Y_{{3,-1}}}}+{\frac {{Y_{{2,-2}}}^{2}Y_{{3,-3}}Y_{{2,0}}}{Y_
{{1,-1}}Y_{{4,-1}}}}+{\frac {{Y_{{1,-3}}}^{2}Y_{{2,-2}}Y_{{4,-3}}{Y_{{
2,0}}}^{2}}{{Y_{{3,-1}}}^{2}Y_{{4,-1}}}} }\\
& \quad \scalemath{0.8}{  +{\frac {Y_{{2,-4}}Y_{{1,-3}}Y
_{{4,-3}}{Y_{{2,0}}}^{2}}{Y_{{3,-1}}}}   +{\frac {Y_{{1,-3}}Y_{{3,-3}}Y_{
{2,-4}}{Y_{{2,0}}}^{2}}{Y_{{4,-1}}}}+{\frac {Y_{{2,-4}}Y_{{3,-3}}Y_{{4
,-3}}{Y_{{2,0}}}^{2}}{Y_{{1,-1}}}}+Y_{{1,-3}}Y_{{3,-3}}Y_{{4,-3}}Y_{{2
,0}}+{\frac {Y_{{1,-3}}Y_{{3,-3}}Y_{{4,-3}}Y_{{2,-4}}{Y_{{2,0}}}^{2}}{
Y_{{2,-2}}}}}\\
& \quad \scalemath{0.8}{  +{\frac {{Y_{{1,-3}}}^{2}Y_{{3,-3}}{Y_{{4,-3}}}^{2}{Y_{{2,0
}}}^{2}}{Y_{{2,-2}}Y_{{3,-1}}}} +{\frac {{Y_{{2,-2}}}^{2}Y_{{2,-4}}{Y_{
{2,0}}}^{2}}{Y_{{1,-1}}Y_{{3,-1}}Y_{{4,-1}}}} +{\frac {Y_{{2,-4}}Y_{{1,
-3}}Y_{{2,-2}}{Y_{{2,0}}}^{2}}{Y_{{4,-1}}Y_{{3,-1}}}}+{\frac {Y_{{2,-2
}}{Y_{{4,-3}}}^{2}Y_{{3,-3}}{Y_{{2,0}}}^{2}}{{Y_{{1,-1}}}^{2}Y_{{3,-1}
}}} +{\frac {Y_{{2,-2}}{Y_{{3,-3}}}^{2}Y_{{4,-3}}{Y_{{2,0}}}^{2}}{{Y_{{
1,-1}}}^{2}Y_{{4,-1}}}} }\\
& \qquad \scalemath{0.8}{  +{\frac {Y_{{1,-3}}{Y_{{4,-3}}}^{2}Y_{{2,-2}}{Y
_{{2,0}}}^{2}}{Y_{{1,-1}}{Y_{{3,-1}}}^{2}}}+{\frac {Y_{{1,-3}}Y_{{4,-3
}}Y_{{2,-2}}Y_{{2,0}}}{Y_{{3,-1}}}}+{\frac {Y_{{3,-3}}Y_{{4,-3}}Y_{{2,
-2}}Y_{{2,0}}}{Y_{{1,-1}}}}   +3\,{\frac {Y_{{1,-3}}Y_{{2,-2}}Y_{{3,-3}}Y
_{{4,-3}}{Y_{{2,0}}}^{2}}{Y_{{1,-1}}Y_{{3,-1}}Y_{{4,-1}}}}, }
\end{align*}
\begin{align*}
& \scalemath{0.8}{   p_3 = {\frac {{Y_{{2,-2}}}^{2}Y_{{1,-3}}Y_{{3,-3}}{Y_{{2,0}}}^{2}}{Y_{{1,-1}
}Y_{{3,-1}}{Y_{{4,-1}}}^{2}}}+{\frac {{Y_{{2,-2}}}^{2}Y_{{1,-3}}Y_{{4,
-3}}{Y_{{2,0}}}^{2}}{{Y_{{3,-1}}}^{2}Y_{{4,-1}}Y_{{1,-1}}}}+{\frac {Y_
{{3,-3}}Y_{{4,-3}}{Y_{{2,-2}}}^{2}{Y_{{2,0}}}^{2}}{Y_{{4,-1}}Y_{{3,-1}
}{Y_{{1,-1}}}^{2}}} +{\frac {Y_{{1,-3}}{Y_{{4,-3}}}^{2}Y_{{3,-3}}{Y_{{2
,0}}}^{2}}{Y_{{1,-1}}Y_{{3,-1}}}}  +{\frac {{Y_{{1,-3}}}^{2}Y_{{3,-3}}Y_
{{4,-3}}{Y_{{2,0}}}^{2}}{Y_{{4,-1}}Y_{{3,-1}}}} }\\
& \qquad \scalemath{0.8}{  +{\frac {Y_{{1,-3}}{Y_{
{3,-3}}}^{2}Y_{{4,-3}}{Y_{{2,0}}}^{2}}{Y_{{1,-1}}Y_{{4,-1}}}}, }
\end{align*}
\begin{align*}
\scalemath{0.8}{   p_4 = {\frac {Y_{{1,-3}}Y_{{2,-2}}Y_{{3,-3}}Y_{{4,-3}}{Y_{{2,0}}}^{2}}{Y_{{1
,-1}}Y_{{3,-1}}Y_{{4,-1}}}}. }
\end{align*}

Therefore the module $L(Y_{2,-4}Y_{2,0})$ in type $D_4$ is non-real.

\section{Discussion} \label{sec:discussion}

In this work we make a connection between tropical geometry, 
representation theory of quantum affine algebras, and scattering amplitudes in physics. 

In mathematical side, we introduce a sequence of Newton polytopes and in the case of $U_q(\widehat{\mathfrak{sl}_k})$, we construct explicitly simple modules from given facets of a Newton polytope. We conjecture that the obtained simple modules are prime modules, see Conjecture \ref{conj:prime modules are rays}. 


On physics side, we generalize the Grassmannian string integral to the setting that the integrand is the infinite product of 
 prime elements in the dual canonical basis of $\CC[\Gr(k,n)]$, see Section \ref{sec:physics}, and more generally the infinite product of the $q$-characters of all prime modules in the category $\mathcal{C}_{\ell}$ of a quantum affine algebra, see Section \ref{subsec:stringy integerals for quantum affine algebras}. 

 We also define the so called $u$-variables for every prime element in the dual canonical basis of $\CC[\Gr(k,n)]$ and we conjecture that the $u$-variables are unique solutions of $u$-equations, see Section \ref{subsec:u-equations}. The $u$-equations are important in the study of scattering amplitudes in physics, see \cite{AHL2019Stringy}.  

Our work raises many related questions. On mathematical side, it is important to give an explicit construction of dominant monomials corresponding to facets of Newton polytopes defined in Section \ref{sec:Newton polytopes for quantum affine algebras} and prove that every prime module in the category $\mathcal{C}_{\ell}$ corresponds to a facet of ${\bf N}_{\mathfrak{g}, \ell}^{(d)}$ for some $d$, see Conjecture \ref{conj:newton polytope and prime modules conjecture quantum affine algebra case}. It is also important to study compatibility of prime modules using Newton polytopes, see Conjecture \ref{conj:newton polytope and prime modules conjecture quantum affine algebra case}. 

On physics side, it is important to compute explicitly $u$-equations and $u$-variables and verify that $u$-variables are unique solutions of $u$-equations, see Conjecture \ref{conj:u variables are solutions of u equations general case for Grkn}. 

In the simplest examples, the Newton polytopes for representations of quantum affine algebras defined in Section \ref{sec:Newton polytopes for quantum affine algebras} are associahedra. It would be very interesting to study the relation between the Newton polytopes for representations of quantum affine algebras and the surfacehedra defined in \cite{AFPST21_surface}.

Finally, let us discuss in some detail an exciting potential research direction which was beyond the scope of the present work to pursue.  In \cite{E2021}, $u$-equations were introduced in order to define a certain  generalized worldsheet associahedron, related to the moduli space of $n$ points in $\mathbb{P}^{k-1}$.  A parameterization of the solution to the $u$-equations was conjectured when $k=3,4$; the details are being worked out in \cite{ETh2023}. What is striking is that these $u$-equations are manifestly finite; there is a Newton polytope which is (conjecturally) simple with a face lattice which is anti-isomorphic to the noncrossing complex $\mathbf{NC}_{k,n}$.  In particular there are finitely many $u$-variables and finitely many $u$-equations.  It would be very interesting to determine if this can be explained in the context of the present paper.  Does the generalized worldsheet associahedron relate to the solution to the $u$-equations which we propose in the present work?  If so, which prime tableaux are involved?  We leave such fascinating questions to future work.

\end{document}